\documentclass[11pt, reqno]{amsart}
\usepackage{amssymb}
\usepackage{eucal}
\usepackage{amsmath}
\usepackage{amscd}
\usepackage[dvips]{color}
\usepackage{multicol}
\usepackage[all]{xy}           
\usepackage{graphicx}
\usepackage{color}
\usepackage{colordvi}
\usepackage{xspace}
\usepackage{ulem}
\usepackage{cancel}
\usepackage{tikz}
\usepackage{longtable}
\usepackage{multirow}
\usepackage{ifpdf}
\ifpdf
 \usepackage[colorlinks,final,hyperindex]{hyperref}
\else
 \usepackage[colorlinks,final,hyperindex,hypertex]{hyperref}
\fi

\begin{document}
\newcommand {\emptycomment}[1]{} 

\newcommand{\tabincell}[2]{\begin{tabular}{@{}#1@{}}#2\end{tabular}}

\newcommand{\nc}{\newcommand}
\newcommand{\delete}[1]{}
\nc{\mfootnote}[1]{\footnote{#1}} 
\nc{\todo}[1]{\tred{To do:} #1}

\delete{
\nc{\mlabel}[1]{\label{#1}}  
\nc{\mcite}[1]{\cite{#1}}  
\nc{\mref}[1]{\ref{#1}}  
\nc{\meqref}[1]{\eqref{#1}} 
\nc{\bibitem}[1]{\bibitem{#1}} 
}

\nc{\mlabel}[1]{\label{#1}  
{\hfill \hspace{1cm}{\bf{{\ }\hfill(#1)}}}}
\nc{\mcite}[1]{\cite{#1}{{\bf{{\ }(#1)}}}}  
\nc{\mref}[1]{\ref{#1}{{\bf{{\ }(#1)}}}}  
\nc{\meqref}[1]{\eqref{#1}{{\bf{{\ }(#1)}}}} 
\nc{\mbibitem}[1]{\bibitem[\bf #1]{#1}} 

\nc{\papl}{pre-anti-pre-Lie\xspace}


\newtheorem{thm}{Theorem}[section]
\newtheorem{lem}[thm]{Lemma}
\newtheorem{cor}[thm]{Corollary}
\newtheorem{pro}[thm]{Proposition}
\newtheorem{conj}[thm]{Conjecture}
\theoremstyle{definition}
\newtheorem{defi}[thm]{Definition}
\newtheorem{ex}[thm]{Example}
\newtheorem{rmk}[thm]{Remark}
\newtheorem{pdef}[thm]{Proposition-Definition}
\newtheorem{condition}[thm]{Condition}



\title[A bialgebra theory for transposed Poisson algebras]
{A bialgebra theory for transposed Poisson algebras via
anti-pre-Lie bialgebras and anti-pre-Lie-Poisson bialgebras}

\author{Guilai Liu}
\address{Chern Institute of Mathematics \& LPMC, Nankai University, Tianjin 300071, China}
\email{1120190007@mail.nankai.edu.cn}

\author{Chengming Bai}
\address{Chern Institute of Mathematics \& LPMC, Nankai University, Tianjin 300071, China }
\email{baicm@nankai.edu.cn}


\begin{abstract}
The approach for Poisson bialgebras characterized by Manin triples
with respect to the invariant bilinear forms on both the
commutative associative algebras and the Lie algebras is not
available for giving a bialgebra theory for transposed Poisson
algebras. Alternatively, we consider Manin triples with respect to
the commutative 2-cocycles on the Lie algebras instead.
Explicitly,  we first introduce the notion of anti-pre-Lie
bialgebras as the equivalent structure of Manin triples of Lie
algebras with respect to the commutative 2-cocycles. Then we
introduce the notion of anti-pre-Lie Poisson bialgebras,
characterized by Manin triples of transposed Poisson algebras with
respect to the bilinear forms which are invariant on the
commutative associative algebras and commutative 2-cocycles on the
Lie algebras, giving a bialgebra theory for transposed Poisson
algebras.  Finally the coboundary cases and the related structures
such as analogues of the classical Yang-Baxter equation and
$\mathcal O$-operators are studied.
\end{abstract}


\subjclass[2020]{
17A36,  
17A40,  
17B10, 
17B38, 
17B40, 
17B60, 
17B63,  
17D25.  
}

\keywords{transposed Poisson algebra; anti-pre-Lie algebra;
commutative 2-cocycle; anti-pre-Lie Poisson algebra; bialgebra}

\maketitle


\tableofcontents

\allowdisplaybreaks
\section{Introduction}\
This paper aims to give a bialgebra theory for 
transposed Poisson algebras in terms of the bialgebra structures
corresponding to Manin triples of transposed Poisson algebras with
respect to the invariant bilinear forms on the commutative
associative algebras and the commutative 2-cocycles on the Lie
algebras.

\subsection{Transposed Poisson algebras}\

Poisson algebras arose in the study of Poisson geometry
(\cite{BV1,Li77,Wei77}), and are closely related to a lot of
topics in mathematics and physics  such as classical mechanics,
quantum mechanics and deformation quantization. The notion of
transposed Poisson algebras is given as the dual notion of Poisson
algebras, which exchanges the roles of the two binary operations
in the Leibniz rule defining the Poisson algebras.

\begin{defi}(\cite{Bai2020})\label{defi:transposed Poisson algebra}
    A \textbf{transposed Poisson algebra} is a triple
    $(A,\cdot,[-,-])$, where the pair $(A,\cdot)$ is a commutative associative
    algebra, the pair $(A,[-,-])$ is a Lie algebra, and the following equation holds:
    \begin{equation}\label{eq:defi:transposed Poisson algebra}
        2z\cdot[x,y]=[z\cdot x,y]+[x,z\cdot y],\;\; \forall x,y,z\in A.
    \end{equation}
\end{defi}

Transposed Poisson algebras share common properties with Poisson
algebras, such as the closure under taking tensor products and the
Koszul self-duality as an operad. They also closely relate to a
lot of other algebraic structures such as Novikov-Poisson algebras
(\cite{Xu1997}) and 3-Lie algebras (\cite{Fil}). In particular,
due to the relationships between transposed Poisson algebras and
3-Lie algebras, the factor $2$ in Eq.~(\ref{eq:defi:transposed
Poisson algebra}) is interpreted as the arity of the operation of
the Lie algebra (\cite{Bai2020}). There are further studies on
transposed Poisson algebras in
\cite{BFK,BOK,FKL,KK,KK2,KLZ,LS,YH}. On the other hand, there are
the following examples of transposed Poisson algebras constructed
from  commutative differential (associative) algebras, and
conversely, any unital transposed Poisson algebra $(A,\cdot,
[-,-])$ in the sense that $(A,\cdot)$ is a unital commutative
associative algebra is exactly obtained this way.

\begin{ex}(\cite{Bai2020})\label{ex:TPA}
Let $(A,\cdot)$ be a commutative associative algebra with a
derivation $P$. Then there is a (Witt type) Lie algebra
(\cite{SXZ,Xu}) defined by
\begin{equation}\label{eq:Lie algebras form differential commutative associative algebras}
    [x,y]=P(x)\cdot y-x\cdot P(y),\;\;\forall x,y\in A.
\end{equation}
Moreover, $(A,\cdot,[-,-])$ is a transposed Poisson algebra.
\end{ex}


A  bialgebra structure is a vector space equipped with both an
algebra structure and a coalgebra structure satisfying certain
compatibility conditions. The known examples of such structures
include Lie bialgebras (\cite{CP1,Dri}) and antisymmetric
infinitesimal bialgebras (\cite{Agu2000, Agu2001,
Agu2004,Bai2010}). Lie bialgebras, as equivalent structures of
Manin triples of Lie algebras (with respect to the invariant
bilinear forms), are closely related to Poisson-Lie groups and
play an important role in the infinitesimalization of quantum
groups (\cite{CP1}). Antisymmetric infinitesimal bialgebras, as
the  associative analogue for Lie bialgebras, can be characterized
as double constructions of Frobenius algebras which are widely
applied in 2d topological field and string theory
(\cite{Kock,Lau}).

There is a  bialgebra theory for Poisson algebras given in
\cite{NB} and thus the notion of Poisson bialgebras was introduced
as a combination of Lie bialgebras and commutative and
cocommutative infinitesimal bialgebras satisfying certain
compatible conditions. Equivalently,  a Poisson bialgebra is
characterized as a Manin triple of Poisson algebras which is both
a Manin triple of Lie algebras (with respect to the invariant
bilinear form) and a double construction of commutative Frobenius
algebras simultaneously.

It is quite natural to consider giving a bialgebra theory for
transposed Poisson algebras. The main idea is still to
characterize the bialgebra theory for transposed Poisson algebras
obtained from a specific kind of Manin triples of transposed
Poisson algebras with respect to the bilinear forms satisfying
suitable conditions. Unfortunately, the approach for Poisson
bialgebras in terms of Manin triples with the respect to the
invariant bilinear forms on both the commutative associative
algebras and the Lie algebras  in \cite{NB} is not available for
transposed Poisson algebras. In fact, if there is a nondegenerate
symmetric bilinear form $\mathcal{B}$ on a transposed Poisson
algebra $(A,\cdot,[-,-])$ such that it is {\bf invariant} on both
$(A,\cdot)$ and $(A,[-,-])$ 
in the sense that
\begin{equation}
    \mathcal{B}(x\cdot y,z)=\mathcal{B}(x,y\cdot z),\;\;
    \mathcal{B}([x,y],z)=\mathcal{B}(x,[y,z]),\;\;\forall x,y,z\in
    A,
\end{equation}
then one shows (see Proposition \ref{pro:tpa bilinear form}) that $(A,\cdot,[-,-])$ 
satisfies
\begin{equation}\label{eq:coherent TPA}
        [x,y]\cdot z=[x\cdot y,z]=0,\;\;\forall x,y,z\in A.
    \end{equation}
It is regarded as a ``trivial" case for
Eq.~(\ref{eq:defi:transposed Poisson algebra}). Note that it is a
Poisson algebra which is also regarded as a ``trivial" case for
the Leibniz rule. Also note that this fact exactly exhibits an
obvious difference between Poisson algebras and transposed Poisson
algebras.

On the other hand, recall that a \textbf{commutative 2-cocycle}
(\cite{Dzh}) on a Lie algebra $(\mathfrak{g},[-,-])$ is a
symmetric bilinear form $ \mathcal{B}$ such that
\begin{equation}\label{defi:2-cocycle}
\mathcal{B}([x,y],z)+\mathcal{B}([y,z],x)+\mathcal{B}([z,x],y)=0,\;\;\forall x,y,z\in\mathfrak{g}.
\end{equation}
Commutative 2-cocycles appear in the study of non-associative
algebras satisfying certain skew-symmetric identities
(\cite{Dzh09}), and in the description of the second cohomology of
current Lie algebras (\cite{Zu}). For a transposed Poisson algebra
$(A,\cdot,[-,-])$ given in Example \ref{ex:TPA}, if $(A,\cdot)$ is
a commutative symmetric Frobenius algebra, that is, if there
exists a nondegenerate symmetric bilinear form $\mathcal{B}$ on
$A$ such that it is invariant on $(A,\cdot)$, then $\mathcal B$ is
a commutative 2-cocycle on $(A,[-,-])$ due to the following fact:

\begin{pro} {\rm (\cite{LB2022})}\label{pro:comm 2-cocycle}
Let $(A,\cdot)$ be a commutative associative algebra with a
derivation $P$. Let $\mathcal{B}$ be a symmetric invariant
bilinear form on $(A,\cdot)$. Then $\mathcal{B}$ is  a commutative
2-cocycle on the Witt type Lie algebra $(A,[-,-])$ defined
 by Eq.~(\ref{eq:Lie algebras form differential commutative
associative algebras}).
\end{pro}
Hence there exists a nontrivial transposed Poisson algebra with a
nondegenerate symmetric bilinear form which  is invariant on the
commutative associative algebra and a commutative 2-cocycle on the
Lie algebra. Therefore it is reasonable to consider a Manin triple
of transposed Poisson algebras with respect to such a bilinear
form instead of the invariant bilinear form on both the
commutative associative algebras and the Lie algebras. Obviously,
it is  a combination of a double construction of commutative
Frobenius algebras and a Manin triple of Lie algebras with respect
to the commutative 2-cocycle. Note that the former corresponds to
a commutative and cocommutative infinitesimal bialgebra in
\cite{Bai2010}.


\subsection{
Anti-pre-Lie algebras and anti-pre-Lie bialgebras}

As the first step for considering a bialgebra theory for
transposed Poisson algebras along the approach given at the end of
the previous subsection, we consider the bialgebra structures
corresponding to Manin triples of Lie algebras with respect to the
commutative 2-cocycles, which are closely related to the following
algebraic structures.

\begin{defi}\label{defi:anti-pre-Lie algebras}(\cite{LB2022})
    Let $A$ be a vector space with a bilinear operation $\circ: A\otimes A\rightarrow A$. The pair $(A,\circ)$ is called an
    \textbf{anti-pre-Lie algebra} if the following equations are
    satisfied:
    \begin{equation}\label{eq:defi:anti-pre-Lie algebras1}
        x\circ(y\circ z)-y\circ(x\circ z)=[y,x]\circ z,
    \end{equation}
    \begin{equation}\label{eq:defi:anti-pre-Lie algebras2}
        [x,y]\circ z+[y,z]\circ x+[z,x]\circ y=0,
    \end{equation}
    where
        $[x,y]=x\circ y-y\circ x$,
    for all $x,y,z\in A$.
\end{defi}

For an anti-pre-Lie algebra $(A,\circ)$, 
the bilinear operation $[-,-]$ defines a Lie algebra, which is
called the \textbf{sub-adjacent Lie algebra} of $(A,\circ)$ and
denoted by $(\mathfrak{g}(A),[-,-])$, and $(A,\circ)$ is called a
\textbf{compatible anti-pre-Lie algebra} structure on
$(\mathfrak{g}(A),[-,-])$.
Conversely,  anti-pre-Lie algebras are characterized as a class of
Lie-admissible algebras whose negative left multiplication
operators make representations of the commutator Lie algebras,
justifying the notion by the comparison with pre-Lie algebras
(\cite{Bur}) which are characterized as a class of Lie-admissible
algebras whose left multiplication operators do so. Moreover, they
are regarded as the underlying algebraic structures of Lie
algebras with nondegenerate commutative 2-cocycles due to the
following relationship between them.

\begin{thm}\label{thm:commutative 2-cocycles and anti-pre-Lie algebras} {\rm
(\cite{LB2022})}
    Let $\mathcal{B}$ be a nondegenerate commutative 2-cocycle on a
    Lie algebra $(\frak g,[-,-])$. Then there exists a  unique compatible
    anti-pre-Lie algebra structure $\circ$ on $(\frak g,[-,-])$ such that
    \begin{equation}\label{eq:thm:commutative 2-cocycles and anti-pre-Lie algebras}
        \mathcal{B}(x\circ y,z)=\mathcal{B}(y,[x,z]), \;\;\forall x,y,z\in
        \frak g.
    \end{equation}
\end{thm}

Therefore Manin triples of Lie algebras with respect to the
commutative 2-cocycles are interpreted in terms of anti-pre-Lie
algebras and in particular, they are equivalent to certain
bialgebra structures for anti-pre-Lie algebras, namely,
anti-pre-Lie bialgebras. Furthermore, both them are equivalent to
the matched pairs of Lie algebras with respect to the
representations given by the compatible anti-pre-Lie algebras. The
study of coboundary cases of anti-pre-Lie bialgebras leads to the
introduction of an analogue of the classical Yang-Baxter equation
in a Lie algebra, called the anti-pre-Lie Yang-Baxter equation
(APL-YBE). The skew-symmetric solutions of the APL-YBE in
anti-pre-Lie algebras give anti-pre-Lie bialgebras. Moreover, the
notions of $\mathcal{O}$-operators on anti-pre-Lie algebras and
\papl (pre-APL) algebras are introduced to provide skew-symmetric
solutions of the APL-YBE in anti-pre-Lie algebras .

\subsection{
Anti-pre-Lie Poisson algebras and anti-pre-Lie Poisson bialgebras}
Now we consider, as a bialgebra theory for transposed Poisson
algebras, the bialgebra structures corresponding to Manin triples
of transposed Poisson algebras with respect to the invariant
bilinear forms on  the commutative associative algebras and the
commutative 2-cocycles on the Lie algebras.

The notion of anti-pre-Lie Poisson algebras was introduced as
follows.

\begin{defi}\label{defi:anti-pre-Lie Poisson}(\cite{LB2022})
    An \textbf{anti-pre-Lie Poisson algebra} is a triple
    $(A,\cdot,\circ)$, where the pair $(A,\cdot)$ is a commutative
    associative algebra and the pair $(A,\circ)$ is an anti-pre-Lie algebra such that the following equations hold:
    \begin{eqnarray}\label{eq:defi:anti-pre-Lie Poisson1}
        2(x\circ y)\cdot z-2(y\circ x)\cdot z&=&y\cdot(x\circ z)-x\cdot(y\circ
        z),\\
    \label{eq:defi:anti-pre-Lie Poisson2}
        2x\circ(y\cdot z)&=&(z\cdot x)\circ y+z\cdot(x\circ
        y),\;\;\forall x,y,z\in A.
\end{eqnarray}
\end{defi}

Anti-pre-Lie Poisson algebras play a similar role here as
anti-pre-Lie algebras in the previous subsection. In particular,
there are relationships between anti-pre-Lie Poisson algebras and
transposed Poisson algebras which are analogues of the
relationships between anti-pre-Lie algebras and the sub-adjacent
Lie algebras. Explicitly, for an anti-pre-Lie Poisson algebra
$(A,\cdot,\circ)$ with $(A,[-,-])$ being the sub-adjacent Lie
algebra of $(A,\circ)$, $(A,\cdot, [-,-])$ is a transposed Poisson
algebra, called the {\bf sub-adjacent transposed Poisson algebra}.
Conversely, anti-pre-Lie Poisson algebras are characterized in
terms of representations of the sub-adjacent transposed Poisson
algebras on the dual spaces of themselves. Moreover, like
Theorem~\ref{thm:commutative 2-cocycles and anti-pre-Lie
algebras}, for a transposed Poisson algebra $(A,\cdot, [-,-])$
with a nondegenerate symmetric bilinear form $\mathcal B$ such
that it is invariant on $(A,\cdot)$ and a commutative 2-cocycle on
$(A,[-,-])$, there is an anti-pre-Lie Poisson algebra
$(A,\cdot,\circ)$, where $(A,\circ)$ is defined by
Eq.~(\ref{eq:thm:commutative 2-cocycles and anti-pre-Lie
algebras}) through $\mathcal B$ (\cite{LB2022}).

Hence the role of anti-pre-Lie Poisson algebras in the study of
the bialgebra structures corresponding to Manin triples of
transposed Poisson algebras with respect to the invariant bilinear
forms on the commutative associative algebras and the commutative
2-cocycles on the Lie algebras is the same as the role of
anti-pre-Lie algebras in the study of the bialgebra structures
corresponding to Manin triples of Lie algebras with respect to the
commutative 2-cocycles. Consequently, we introduce the notion of
anti-pre-Lie Poisson bialgebras as the equivalent structures for
the above Manin triples of transposed Poisson algebras as well as
the needed bialgebra theory for transposed Poisson algebras. The
study of coboundary cases and the related structures such as an
analogue of the classical Yang-Baxter equation and $\mathcal
O$-operators on anti-pre-Lie Poisson algebras is still available.

\subsection{Layout of the paper}
This paper is organized as follows.

In Section \ref{S2}, we introduce the notion of anti-pre-Lie
bialgebras as the bialgebra structures corresponding to Manin
triples of Lie algebras with respect to the commutative
2-cocycles. Both of them are interpreted in terms of certain
matched pairs of Lie algebras  as well as the compatible
anti-pre-Lie algebras. The study of coboundary cases leads to the
introduction of the APL-YBE, whose skew-symmetric solutions give
coboundary anti-pre-Lie bialgebras. The notions of
$\mathcal{O}$-operators of anti-pre-Lie algebras and pre-APL 
algebras are introduced to construct skew-symmetric solutions of
the APL-YBE in anti-pre-Lie algebras.

In Section \ref{S4}, we characterize anti-pre-Lie Poisson algebras
in terms of representations of the sub-adjacent transposed Poisson
algebras on the dual spaces of themselves. Then we introduce the
notion of anti-pre-Lie Poisson bialgebras as the bialgebra
structures corresponding to Manin triples of transposed Poisson
algebras with respect to the invariant bilinear forms on the
commutative associative algebras and the commutative 2-cocycles on
the Lie algebras, characterized by certain matched pairs of
anti-pre-Lie Poisson algebras and transposed Poisson algebras. The
study of coboundary cases and the related structures is given.

Throughout this paper,  unless otherwise specified, all the vector
spaces and algebras are finite-dimensional over an algebraically
closed field $\mathbb {K}$ of characteristic zero, although many
results and notions remain valid in the infinite-dimensional case.

\section{Anti-pre-Lie bialgebras}\label{S2}
We introduce the notions of representations and matched pairs of
anti-pre-Lie algebras, and give their relationships with
representations and matched pairs of the sub-adjacent Lie
algebras. Then we introduce the notion of Manin triples of Lie
algebras with respect to the commutative 2-cocycles and give their
equivalence with certain matched pairs of Lie algebras as well as
the compatible anti-pre-Lie algebras. Consequently, we introduce
the notion of anti-pre-Lie bialgebras as their equivalent
structures. Finally, we study the coboundary anti-pre-Lie
bialgebras, which lead to the introduction of the anti-pre-Lie
Yang-Baxter equation (APL-YBE). In particular, a skew-symmetric
solution of the APL-YBE in an anti-pre-Lie algebra gives a
coboundary anti-pre-Lie bialgebra. We also introduce the notions
of $\mathcal{O}$-operators of anti-pre-Lie algebras and \papl
(pre-APL) algebras to construct skew-symmetric solutions of the
APL-YBE in anti-pre-Lie algebras.

\subsection{Representations and matched pairs of anti-pre-Lie algebras}\


Recall some basic facts on the representations of Lie algebras. A
\textbf{representation} of a Lie algebra $(\mathfrak{g},[-,-])$ is
a pair $(\rho,V)$, such that $V$ is a vector space and
$\rho:\mathfrak{g}\rightarrow\mathfrak{gl}(V)$ is a Lie algebra
homomorphism for the natural Lie algebra structure on
$\mathfrak{gl}(V)=\mathrm{End}(V)$.
In particular, the linear map ${\rm ad}:\mathfrak{g}\rightarrow
\mathfrak{gl}(\mathfrak{g})$  defined by ${\rm ad}(x)(y)=[x,y]$
for all $x,y\in \mathfrak{g}$, gives a representation
$(\mathrm{ad},\mathfrak{g})$, called the \textbf{adjoint
representation} of $(\mathfrak{g},[-,-])$.

For a vector space $V$ and a linear map $\rho:\mathfrak{g}\rightarrow\mathfrak{gl}(V)$, the pair $(\rho,V)$ is a representation of a Lie algebra $(\mathfrak{g},[-,-])$ if and only if 
$\mathfrak{g}\oplus V$
is a ({\bf semi-direct product}) Lie algebra  by defining the
multiplication $[-,-]_{\mathfrak{g}\oplus V}$ (often still denoted
by $[-,-]$ for simplicity) on $\mathfrak{g}\oplus V$ by
\begin{equation}\label{eq:SDLie}
    [x+u,y+v]_{\mathfrak{g}\oplus V}=[x,y]+\rho(x)v-\rho(y)u,\;\;\forall x,y\in \mathfrak{g}, u,v\in V.
\end{equation}
We denote it by $\mathfrak{g}\ltimes_{\rho}V$.

Now we give the notion of representations of anti-pre-Lie algebras.
\begin{defi}\label{defi:rep anti-pre-Lie algebra}
    Let $(A,\circ)$ be an anti-pre-Lie algebra. A \textbf{representation} of $(A,\circ)$ is a triple $(l_{\circ},r_{\circ},V)$, such that $V$ is a vector space, and $l_{\circ},r_{\circ}:A\rightarrow \mathrm{End}(V)$ are linear maps satisfying
\begin{eqnarray}
    l_{\circ}(y\circ x)-l_{\circ}(x\circ y)&=&l_{\circ}(x)l_{\circ}(y)-l_{\circ}(y)l_{\circ}(x),\label{eq:defi:rep anti-pre-Lie algebra1}\\
    r_{\circ}(x\circ y)&=&l_{\circ}(x)r_{\circ}(y)+r_{\circ}(y)l_{\circ}(x)-r_{\circ}(y)r_{\circ}(x),\label{eq:defi:rep anti-pre-Lie algebra2}\\
     l_{\circ}(y\circ x)-l_{\circ}(x\circ y)&=&r_{\circ}(x)l_{\circ}(y)-r_{\circ}(y)l_{\circ}(x)-r_{\circ}(x)r_{\circ}(y)+r_{\circ}(y)r_{\circ}(x),\label{eq:defi:rep anti-pre-Lie algebra3}
\end{eqnarray}
\delete{
    \begin{equation}\label{eq:defi:rep anti-pre-Lie algebra1}
        l_{\circ}(y\circ x)-l_{\circ}(x\circ y)=l_{\circ}(x)l_{\circ}(y)-l_{\circ}(y)l_{\circ}(x),
    \end{equation}
    \begin{equation}\label{eq:defi:rep anti-pre-Lie algebra2}
        r_{\circ}(x\circ y)=l_{\circ}(x)r_{\circ}(y)+r_{\circ}(y)l_{\circ}(x)-r_{\circ}(y)r_{\circ}(x),
    \end{equation}
    \begin{equation}\label{eq:defi:rep anti-pre-Lie algebra3}
        l_{\circ}(y\circ x)-l_{\circ}(x\circ y)=r_{\circ}(x)l_{\circ}(y)-r_{\circ}(y)l_{\circ}(x)-r_{\circ}(x)r_{\circ}(y)+r_{\circ}(y)r_{\circ}(x),
    \end{equation}}
    for all $x,y\in A$.
\end{defi}

\begin{ex}
    Let $(A,\circ)$ be an anti-pre-Lie algebra.
    Define linear maps $\mathcal{L}_{\circ},\mathcal{R}_{\circ}:A\rightarrow\mathrm{End}(A)$ by $\mathcal{L}_{\circ}(x)y=\mathcal{R}_{\circ}(y)x=x\circ y$, for all $x,y\in A$.
    Then $(\mathcal{L}_{\circ},\mathcal{R}_{\circ},A)$ is  a representation of $(A,\circ)$, called the \textbf{adjoint representation} of $(A,\circ).$
\end{ex}

\begin{pro}\label{pro:repandsemidirectproduct}
    Let $(A,\circ)$ be an anti-pre-Lie algebra and $\big(\mathfrak{g}(A),[-,-]\big)$ be the sub-adjacent Lie algebra of $(A,\circ)$. Let $V$ be a vector space and $l_{\circ},r_{\circ}:A\rightarrow \mathrm{End}(V)$ be two linear maps.
    \begin{enumerate}
\item\label{rep1} $(l_{\circ},r_{\circ},V)$ is a representation of
$(A,\circ)$ if and only if the direct sum $A\oplus V$ of vector
spaces is a (\textbf{semi-direct product}) anti-pre-Lie algebra by
defining the bilinear operation $\circ_{A\oplus V}$ (often still
denoted by $\circ$) on $A\oplus V$ by
        \begin{equation}\label{eq:pro:repandsemidirectproduct1}
            (x+u)\circ_{A\oplus V}(y+v)=x\circ y+l_{\circ}(x)v+r_{\circ}(y)u,\;\;\forall x,y\in A, u,v\in V.
        \end{equation}
        We denote it by $A\ltimes_{l_{\circ},r_{\circ}}V$.
\item Let $(l_{\circ},r_{\circ},V)$ be a representation of
$(A,\circ)$. \label{rep2}
\begin{itemize}
\item[(a)] \label{rep2-1} $(-l_{\circ},V)$ is a representation of
$\big(\mathfrak{g}(A),[-,-]\big)$. In particular,
$(-\mathcal{L}_{\circ},A)$ is a representation of
$\big(\mathfrak{g}(A),[-,-]\big)$. \item[(b)]\label{rep2-2}
$(l_{\circ}-r_{\circ},V)$ is a representation of
$\big(\mathfrak{g}(A),[-,-]\big)$.
\end{itemize}
    \end{enumerate}
\end{pro}
\begin{proof}
(\ref{rep1}). It is a special case of the matched pair of
anti-pre-Lie algebras in Theorem \ref{thm:matched pairs of
anti-pre-Lie algebras} when $B=V$ is equipped with the zero
multiplication.

(\ref{rep2}) (a). It follows directly from Eq.~(\ref{eq:defi:rep
anti-pre-Lie algebra1}).

 (\ref{rep2}) (b). It is a special case
of Proposition~ \ref{pro:from matched pairs of anti-pre-Lie
algebras to matched pairs of Lie algebras} when $B=V$ is equipped
with the zero multiplication.
\delete{ By Item
(\ref{rep1}), there is a semi-direct product anti-pre-Lie algebra
$A\ltimes_{l_{\circ},r_{\circ}}V$, whose sub-adjacent Lie algebra
structure on $A\oplus V$ is given by
        \begin{eqnarray*}
        [x+u,y+v]_{\mathfrak{g}(A)\oplus V}&=&(x+u)\circ_{A\oplus V}(y+v)-(y+v)\circ_{A\oplus V}(x+u)\\
        &=&x\circ y+l_{\circ}(x)v+r_{\circ}(y)u-(y\circ x+l_{\circ}(y)u+r_{\circ}(x)v)\\
        &=&[x,y]+(l_{\circ}-r_{\circ})(x)v-(l_{\circ}-r_{\circ})(y)u,
        \end{eqnarray*}
        for all $x,y\in A, u,v\in V$. Thus $(l_{\circ}-r_{\circ},V)$ is a representation of
        $(\mathfrak{g}(A),[-,-])$.}
\end{proof}

Let $A$ and $V$ be vector spaces. For a linear map
$\rho:A\rightarrow{\rm End}(V)$, we set
$\rho^{*}:A\rightarrow\mathrm{End}(V^{*})$ by
\begin{equation}
\langle\rho^{*}(x)u^{*},v\rangle=-\langle
u^{*},\rho(x)v\rangle,\;\;\forall x\in A, u^{*}\in V^{*}, v\in V.
\end{equation}
Here $\langle\ ,\ \rangle$ is the usual pairing between $V$ and
$V^*$. It is known that if $(\rho,V)$ is a representation of a Lie
algebra $(\mathfrak{g},[-,-])$, then $(\rho^{*},V^{*})$ is also a
representation of $(\mathfrak{g},[-,-])$. In particular,
$(\mathrm{ad}^{*},\mathfrak{g}^{*})$ is a representation of
$(\mathfrak{g},[-,-])$.

\begin{pro}\label{pro:dual rep anti-pre-Lie algebra}
    Let $(l_{\circ}, r_{\circ},V)$ be a representation of an anti-pre-Lie algebra $(A,\circ)$.
    \begin{enumerate}
        \item\label{it:a}$(-l^{*}_{\circ}, V^{*})$ is a representation of the sub-adjacent Lie algebra $\big(\mathfrak{g}(A),[-,-]\big)$. In particular, $(-\mathcal{L}^{*}_{\circ}, A^{*})$ is a representation of $\big(\mathfrak{g}(A),[-,-]\big)$.
        \item \label{it:b} $(r_{\circ}^{*}-l_{\circ}^{*}$,$r_{\circ}^{*}$,$V^{*})$ is a representation of $(A,\circ)$.  In particular, $(\mathcal{R}_{\circ}^{*}-\mathcal{L}_{\circ}^{*}=-\mathrm{ad}^{*},\mathcal{R}_{\circ}^{*},A^{*})$ is a representation of $(A,\circ)$.
    \end{enumerate}
\end{pro}
\begin{proof}
(\ref{it:a}). It follows from Proposition
\ref{pro:repandsemidirectproduct} (\ref{rep2}) (a).

(\ref{it:b}).
    Let $x,y\in A, u^{*} \in V^{*}, v\in V$. Then we have
{\small \begin{eqnarray*}
        &&\langle \big( (r_{\circ}^{*}-l_{\circ}^{*})(y\circ x)-(r_{\circ}^{*}-l_{\circ}^{*})(x\circ y)-(r_{\circ}^{*}-l_{\circ}^{*})(x)(r_{\circ}^{*}-l_{\circ}^{*})(y)+(r_{\circ}^{*}-l_{\circ}^{*})(y)(r_{\circ}^{*}-l_{\circ}^{*})(x) \big) u^{*},v\rangle\\
        &&=\langle u^{*}, \big( (l_{\circ}-r_{\circ})(y\circ x)-(l_{\circ}-r_{\circ})(x\circ y)-(l_{\circ}-r_{\circ})(y)(l_{\circ}-r_{\circ})(x)+(l_{\circ}-r_{\circ})(x)(l_{\circ}-r_{\circ})(y) \big) v\rangle\\
        &&\overset{(\ref{eq:defi:rep anti-pre-Lie algebra3})}{=}\langle u^{*},\big( -r_{\circ}(y\circ x)+r_{\circ}(x\circ y)-l_{\circ}(y)l_{\circ}(x)+l_{\circ}(y)r_{\circ}(x)+l_{\circ}(x)l_{\circ}(y)-l_{\circ}(x)r_{\circ}(y) \big) v \rangle\\
        &&\overset{(\ref{eq:defi:rep anti-pre-Lie algebra1}),(\ref{eq:defi:rep anti-pre-Lie algebra2})}{=}\langle u^{*},\big( l_{\circ}(x)r_{\circ}(y)+r_{\circ}(y)l_{\circ}(x)-r_{\circ}(y)r_{\circ}(x)
        -l_{\circ}(y)r_{\circ}(x)-r_{\circ}(x)l_{\circ}(y)+r_{\circ}(x)r_{\circ}(y)\\
        &&\hspace{1cm}  +l_{\circ}(y)r_{\circ}(x)-l_{\circ}(x)r_{\circ}(y)+l_{\circ}(y\circ x)-l_{\circ}(x\circ y)\big) v\rangle\\
        &&\overset{(\ref{eq:defi:rep anti-pre-Lie algebra3})}{=}0.
    \end{eqnarray*}}
    Thus Eq.~(\ref{eq:defi:rep anti-pre-Lie algebra1}) holds for the
     triple   $(r_{\circ}^{*}-l_{\circ}^{*}$,$r_{\circ}^{*}$,$V^{*})$. Similarly, Eqs.~(\ref{eq:defi:rep anti-pre-Lie algebra2}) and (\ref{eq:defi:rep anti-pre-Lie algebra3}) hold for the triple
      $(r_{\circ}^{*}-l_{\circ}^{*}$,$r_{\circ}^{*}$,$V^{*})$. Thus $(r_{\circ}^{*}-l_{\circ}^{*}$,$r_{\circ}^{*}$,$V^{*})$ is a representation of $(A,\circ)$.
\end{proof}

Next we consider matched pairs of anti-pre-Lie algebras.
\begin{defi}\label{defi:matched pairs of anti-pre-Lie algebras}
        Let $(A,\circ_{A})$ and $(B,\circ_{B})$ be two anti-pre-Lie algebras. Suppose that $(l_{\circ_{A}},r_{\circ_{A}},B)$ and $(l_{\circ_{B}}$, $r_{\circ_{B}}$, $A)$ are  representations of $(A,\circ_{A})$ and $(B,\circ_{B})$ respectively. Let $\big(\mathfrak{g}(A),[-,-]_{A}\big)$ and $\big(\mathfrak{g}(B)$,
        $[-,-]_{B}\big)$ be the sub-adjacent Lie algebras of  $(A,\circ_{A})$ and $(B,\circ_{B})$ respectively.  Suppose that the following equations hold:
        \begin{small}
    \begin{equation*}\label{eq:defi:matched pairs of anti-pre-Lie algebras1}
        r_{\circ_{B}}\big ( l_{\circ_{A}}(y)a\big)x+x\circ_{A}r_{\circ_{B}}(a)y-r_{\circ_{B}}\big( l_{\circ_{A}}(x)a\big) y-y\circ_{A}r_{\circ_{B}}(a)x=r_{\circ_{B}}(a)([y,x]_{A}),
    \end{equation*}
    \begin{equation*}\label{eq:defi:matched pairs of anti-pre-Lie algebras2}
        x\circ_{A}l_{\circ_{B}}(a)y+r_{\circ_{B}}\big( r_{\circ_{A}}(y)a \big) x-l_{\circ_{B}}(a)(x\circ_{A}y)=(l_{\circ_{B}}-r_{\circ_{B}})(a)x\circ_{A}y+l_{\circ_{B}}\big( (r_{\circ_{A}}-l_{\circ_{A}})(x)a\big) y,
    \end{equation*}
    \begin{equation*}\label{eq:defi:matched pairs of anti-pre-Lie algebras3}
    r_{\circ_{B}}(a)([x,y]_{A})=(l_{\circ_{B}}-r_{\circ_{B}})(a)y\circ_{A}x+(r_{\circ_{B}}-l_{\circ_{B}})(a)x\circ_{A}y+l_{\circ_{B}}\big((r_{\circ_{A}}-l_{\circ_{A}})(y)a\big)x+l_{\circ_{B}}\big((l_{\circ_{A}}-r_{\circ_{A}})(x)a\big)y,
    \end{equation*}
    \begin{equation*}\label{eq:defi:matched pairs of anti-pre-Lie algebras4}
        r_{\circ_{A}}\big(l_{\circ_{B}}(b)x\big)a+a\circ_{B}r_{\circ_{A}}(x)b-r_{\circ_{A}}\big(l_{\circ_{B}}(a)x\big)b-b\circ_{B}r_{\circ_{A}}(x)a=r_{\circ_{A}}(x)([b,a]_{B}),
    \end{equation*}
    \begin{equation*}\label{eq:defi:matched pairs of anti-pre-Lie algebras5}
        a\circ_{B}l_{\circ_{A}}(x)b+r_{\circ_{A}}\big(r_{\circ_{B}}(b)x\big)a-l_{\circ_{A}}(x)(a\circ_{B}b)=(l_{\circ_{A}}-r_{\circ_{A}})(x)a\circ_{B}b+l_{\circ_{A}}\big((r_{\circ_{B}}-l_{\circ_{B}})(a)x\big)b,
    \end{equation*}
    \begin{equation*}\label{eq:defi:matched pairs of anti-pre-Lie algebras6}
    r_{\circ_{A}}(x)([a,b]_{B})=(l_{\circ_{A}}-r_{\circ_{A}})(x)b\circ_{B}a+(r_{\circ_{A}}-l_{\circ_{A}})(x)a\circ_{B}b+l_{\circ_{A}}\big((r_{\circ_{B}}-l_{\circ_{B}})(b)x\big)a+l_{\circ_{A}}\big((l_{\circ_{B}}-r_{\circ_{B}})(a)x\big)b,
    \end{equation*}
\end{small}
    for all $x,y\in A,a,b\in B$. Such a structure is called a \textbf{matched pair of anti-pre-Lie algebras} $(A,\circ_{A})$ and $(B,\circ_{B})$. We denote it by $(A,B,l_{\circ_{A}},r_{\circ_{A}},l_{\circ_{B}},r_{\circ_{B}})$.
\end{defi}

\begin{thm}\label{thm:matched pairs of anti-pre-Lie algebras}
    Let $(A,\circ_{A})$ and $(B,\circ_{B})$ be two anti-pre-Lie algebras. Suppose that
    $l_{\circ_{A}},r_{\circ_{A}}:A\rightarrow\mathrm{End}(B)$ and
        $l_{\circ_{B}},r_{\circ_{B}}:B\rightarrow\mathrm{End}(A)$ are linear maps.
    Define a bilinear operation on $A\oplus B$ by
    \begin{equation}\label{thm:matched pairs of anti-pre-Lie algebras1}
        (x+a)\circ(y+b)=x\circ_{A}y+l_{\circ_{B}}(a)y+r_{\circ_{B}}(b)x+l_{\circ_{A}}(x)b+r_{\circ_{A}}(y)a+a\circ_{B}b,
    \end{equation}
    for all $x,y\in A, a,b\in B$. Then $(A\oplus B,\circ)$ is an anti-pre-Lie algebra if and only if  $(A,B,l_{\circ_{A}},r_{\circ_{A}},l_{\circ_{B}}$,
    $r_{\circ_{B}})$ is a matched pair of anti-pre-Lie algebras. In this case, we denote this anti-pre-Lie algebra structure on $A\oplus B$
    by $A\bowtie^{l_{\circ_{A}},r_{\circ_{A}}}_{l_{\circ_{B}},r_{\circ_{B}}}B$. Conversely,
    every anti-pre-Lie algebra which is the direct sum of the underlying vector spaces of two subalgebras can be obtained from a matched pair of anti-pre-Lie algebras by this construction.
\end{thm}
\begin{proof}
It is straightforward. \delete{    Let $x,y,z\in A, a,b,c\in B$.
    Since $(A,\circ_{A})$ and $(B,\circ_{B})$ are anti-pre-Lie algebras, equation
    \begin{small}
        \begin{equation}\label{thm:matched pairs of anti-pre-Lie algebras2}
        (x+a)\circ((y+b)\circ(z+c))-(y+b)\circ((x+a)\circ(z+c))=((y+b)\circ(x+a))\circ(z+c)-((x+a)\circ(y+b))\circ(z+c)
        \end{equation}
    \end{small}
    holds on $A\oplus B$ if and only if  the following equations hold:
    \begin{eqnarray}
    x\circ(y\circ a)-y\circ(x\circ a)&=&(y\circ_{A} x)\circ a-(x\circ_{A} y)\circ a,\label{thm:matched pairs of anti-pre-Lie algebras3}\\
        x\circ(a\circ y)-a\circ(x\circ_{A} y)&=&(a\circ x)\circ y-(x\circ a)\circ y,\label{thm:matched pairs of anti-pre-Lie algebras4}\\
            a\circ(x\circ b)-x\circ(a\circ_{B} b)&=&(x\circ a)\circ b-(a\circ x)\circ b,\label{thm:matched pairs of anti-pre-Lie algebras6}\\
            a\circ(b\circ x)-b\circ(a\circ x)&=&(b\circ_{B} a)\circ x-(a\circ_{B} b)\circ x.\label{thm:matched pairs of anti-pre-Lie algebras5}
    \end{eqnarray}
    \delete{\begin{equation}\label{thm:matched pairs of anti-pre-Lie algebras3}
        x\circ(y\circ a)-y\circ(x\circ a)=(y\circ_{A} x)\circ a-(x\circ_{A} y)\circ a,
    \end{equation}
    \begin{equation}\label{thm:matched pairs of anti-pre-Lie algebras4}
        x\circ(a\circ y)-a\circ(x\circ_{A} y)=(a\circ x)\circ y-(x\circ a)\circ y,
    \end{equation}
    \begin{equation}\label{thm:matched pairs of anti-pre-Lie algebras5}
        a\circ(b\circ x)-b\circ(a\circ x)=(b\circ_{B} a)\circ x-(a\circ_{B} b)\circ x,
    \end{equation}
    \begin{equation}\label{thm:matched pairs of anti-pre-Lie algebras6}
        a\circ(x\circ b)-x\circ(a\circ_{B} b)=(x\circ a)\circ b-(a\circ x)\circ b.
    \end{equation}}
    By Eq.~(\ref{thm:matched pairs of anti-pre-Lie algebras1}), we have
    \begin{eqnarray*}
        x\circ(y\circ a)&=&x\circ(l_{\circ_{A}}(y)a+r_{\circ_{B}}(a)y)\\
        &=&l_{\circ_{A}}(x)l_{\circ_{A}}(y)a+r_{\circ_{B}}(l_{\circ_{A}}(y)a)x+x\circ_{A}r_{\circ_{B}}(a)y,\\
        -y\circ(x\circ a)&=&-l_{\circ_{A}}(y)l_{\circ_{A}}(x)a-r_{\circ_{B}}(l_{\circ_{A}}(x)a)y-y\circ_{A}r_{\circ_{B}}(a)x,\\
        (y\circ_{A} x)\circ a&=&l_{\circ_{A}}(y\circ_{A}x)a+r_{\circ_{B}}(a)(y\circ_{A}x),\\
        -(x\circ_{A} y)\circ a&=&-l_{\circ_{A}}(x\circ_{A}y)a-r_{\circ_{B}}(a)(x\circ_{A}y).
    \end{eqnarray*}
    Thus Eq.~(\ref{thm:matched pairs of anti-pre-Lie algebras3}) holds if and only if Eq.~(\ref{eq:defi:matched pairs of anti-pre-Lie algebras1}) and the following equation hold:
    $$l_{\circ_{A}}(x)l_{\circ_{A}}(y)a-l_{\circ_{A}}(y)l_{\circ_{A}}(x)a=l_{\circ_{A}}(y\circ_{A}x)a-l_{\circ_{A}}(x\circ_{A}y)a.$$
    Similarly, Eq.(\ref{thm:matched pairs of anti-pre-Lie algebras4}) hold if and only if Eq.~(\ref{eq:defi:matched pairs of anti-pre-Lie algebras2}) and the following equation hold:
    $$r_{\circ_{A}}(x\circ_{A}y)a=l_{\circ_{A}}(x)r_{\circ_{A}}(y)a+r_{\circ_{A}}(y)l_{\circ_{A}}(x)a-r_{\circ_{A}}(y)r_{\circ_{A}}(x)a.$$
    On the other hand, since $(A,\circ_{A})$ and $(B,\circ_{B})$ are anti-pre-Lie algebras, equation
    \begin{small}
    \begin{equation}\label{thm:matched pairs of anti-pre-Lie algebras7}
    \begin{array}{ll}
    &((y+b)\circ(x+a))\circ(z+c)-((x+a)\circ(y+b))\circ(z+c)+((z+c)\circ(y+b))\circ(x+a)\\
    &-((y+b)\circ(z+c))\circ(x+a)+((x+a)\circ(z+c))\circ(y+b)-((z+c)\circ(x+a))\circ(y+b)=0
    \end{array}
    \end{equation}
    \end{small}
    \delete{\begin{equation}\label{thm:matched pairs of anti-pre-Lie algebras7}
        \begin{array}{ll}
            0&=((y+b)\circ(x+a))\circ(z+c)-((x+a)\circ(y+b))\circ(z+c)\\
            &\ \  +((z+c)\circ(y+b))\circ(x+a)-((y+b)\circ(z+c))\circ(x+a)\\
            &\ \  +((x+a)\circ(z+c))\circ(y+b)-((z+c)\circ(x+a))\circ(y+b)\\
        \end{array}
    \end{equation}}
    holds if and only if the following equations hold:
    \begin{equation}\label{thm:matched pairs of anti-pre-Lie algebras8}
        [y,x]_{A}\circ a+(a\circ y)\circ x-(y\circ a)\circ x+(x\circ a)\circ y-(a\circ x)\circ y=0,
    \end{equation}
    \begin{equation}\label{thm:matched pairs of anti-pre-Lie algebras9}
        [b,a]_{B}\circ x+(x\circ b)\circ a-(b\circ x)\circ a+(a\circ x)\circ b-(x\circ a)\circ b=0.
    \end{equation}
\delete{    By Eq.~(\ref{thm:matched pairs of anti-pre-Lie algebras1}), we have
    \begin{eqnarray*}
    [y,x]_{A}\circ a&=&l_{\circ_{A}}([y,x]_{A})a+r_{\circ_{B}}(a)([y,x]_{A}),\\
    (a\circ y)\circ x&=&l_{\circ_{B}}(a)y\circ_{A}x+l_{\circ_{B}}(r_{\circ_{A}}(y)a)x+r_{\circ_{A}}(x)r_{\circ_{A}}(y)a,\\
    -(y\circ a)\circ x&=&-l_{\circ_{B}}(l_{\circ_{A}}(y)a)x-r_{\circ_{A}}(x)l_{\circ_{A}}(y)a-r_{\circ_{B}}(a)y\circ_{A}x,\\
    (x\circ a)\circ y&=&l_{\circ_{B}}(l_{\circ_{A}}(x)a)y+r_{\circ_{A}}(y)l_{\circ_{A}}(x)a+r_{\circ_{B}}(a)x\circ_{A}y,\\
    -(a\circ x)\circ y&=&-l_{\circ_{B}}(a)x\circ_{A}y-l_{\circ_{B}}(r_{\circ_{A}}(x)a)y-r_{\circ_{A}}(y)r_{\circ_{B}}(x)a.
    \end{eqnarray*}}
     Again by Eq.~(\ref{thm:matched pairs of anti-pre-Lie algebras1}), we have that
     Eq.~(\ref{thm:matched pairs of anti-pre-Lie algebras8}) holds if and only if Eq.~(\ref{eq:defi:matched pairs of anti-pre-Lie algebras3}) and the following equation hold:
    $$l_{\circ_{A}}([y,x]_{A})a+r_{\circ_{A}}(x)r_{\circ_{A}}(y)a-r_{\circ_{A}}(x)l_{\circ_{A}}(y)a+r_{\circ_{B}}(y)l_{\circ_{B}}(x)a-r_{\circ_{A}}(y)r_{\circ_{A}}(x)a=0.$$
    In conclusion, Eqs.~(\ref{thm:matched pairs of anti-pre-Lie algebras3}),(\ref{thm:matched pairs of anti-pre-Lie algebras4}) and (\ref{thm:matched pairs of anti-pre-Lie algebras8}) hold if and only if Eqs.~(\ref{eq:defi:matched pairs of anti-pre-Lie algebras1}), (\ref{eq:defi:matched pairs of anti-pre-Lie algebras2}), (\ref{eq:defi:matched pairs of anti-pre-Lie algebras3}) hold, and $(l_{\circ_{A}},r_{\circ_{A}},B)$ is a representation of $(A,\circ_{A})$.
    By symmetry,   Eqs.~(\ref{thm:matched pairs of anti-pre-Lie algebras6}),(\ref{thm:matched pairs of anti-pre-Lie algebras5}) and (\ref{thm:matched pairs of anti-pre-Lie algebras9}) hold if and only if Eqs.~(\ref{eq:defi:matched pairs of anti-pre-Lie algebras4}), (\ref{eq:defi:matched pairs of anti-pre-Lie algebras5}), (\ref{eq:defi:matched pairs of anti-pre-Lie algebras6}) hold, and $(l_{\circ_{B}},r_{\circ_{B}},A)$ is a representation of $(B,\circ_{B})$. Thus $(A\oplus B,\circ)$ is an anti-pre-Lie algebra if and only if  $(A,B,l_{\circ_{A}},r_{\circ_{A}},l_{\circ_{B}},r_{\circ_{B}})$ is a matched pair of anti-pre-Lie algebras.
}
\end{proof}

Recall the notion of matched pairs of Lie algebras (\cite{Maj}).
Let $(\mathfrak{g},[-,-]_{\mathfrak{g}})$ and
$(\mathfrak{h},[-,-]_{\mathfrak{h}})$ be two Lie algebras. Suppose
that $(\rho_{\mathfrak{g}},\mathfrak{h})$ and
$(\rho_{\mathfrak{h}},\mathfrak{g})$ are representations of
$(\mathfrak{g},[-,-]_{\mathfrak{g}})$ and
$(\mathfrak{h},[-,-]_{\mathfrak{h}})$ respectively. If the
following equations are satisfied:
\begin{equation}\label{eq:matched pair of Lie algebras1}
    \rho_{\mathfrak{g}}(x)[a,b]_{\mathfrak{h}}-[\rho_{\mathfrak{g}}(x)a,b]_{\mathfrak{h}}-[a,\rho_{\mathfrak{g}}(x)b]_{\mathfrak{h}}+\rho_{\mathfrak{g}}\big(\rho_{\mathfrak{h}}(a)x\big)b-\rho_{\mathfrak{g}}\big(\rho_{\mathfrak{h}}(b)x\big)a=0,
\end{equation}
\begin{equation}\label{eq:matched pair of Lie algebras2}
    \rho_{\mathfrak{h}}(a)[x,y]_{\mathfrak{g}}-[\rho_{\mathfrak{h}}(a)x,y]_{\mathfrak{g}}-[x,\rho_{\mathfrak{h}}(a)y]_{\mathfrak{g}}+\rho_{\mathfrak{h}}\big(\rho_{\mathfrak{g}}(x)a\big)y-\rho_{\mathfrak{h}}\big(\rho_{\mathfrak{g}}(y)a\big)x=0,
\end{equation}
for all $x,y\in \mathfrak{g},a,b\in \mathfrak{h}$, then
$(\mathfrak{g},\mathfrak{h},\rho_{\mathfrak{g}},\rho_{\mathfrak{h}})$ is called a \textbf{matched pair
of Lie algebras}. In fact, for Lie algebras $(\mathfrak{g},[-,-]_{\mathfrak{g}})$ , $(\mathfrak{h},[-,-]_{\mathfrak{h}})$ and linear maps $\rho_{\mathfrak{g}}:\mathfrak{g}\rightarrow\mathrm{End}(\mathfrak{h})$,$\rho_{\mathfrak{h}}:\mathfrak{h}\rightarrow\mathrm{End}(\mathfrak{g})$, there is a Lie algebra structure
on the vector space $\mathfrak{g}\oplus \mathfrak{h}$ given by
\begin{equation}\label{eq:Lie}
    [x+a,y+b]=[x,y]_{\mathfrak{g}}+\rho_{\mathfrak{h}}(a)y-\rho_{\mathfrak{h}}(b)x+[a,b]_{\mathfrak{h}}+\rho_{\mathfrak{g}}(x)b-\rho_{\mathfrak{g}}(y)a,\;\;\forall x,y\in \mathfrak{g}, a,b\in \mathfrak{h}
\end{equation}
if and only if
$(\mathfrak{g},\mathfrak{h},\rho_{\mathfrak{g}},\rho_{\mathfrak{h}})$
is a matched pair of Lie algebras. In this case, we denote the Lie
algebra structure on  $\mathfrak{g}\oplus \mathfrak{h}$ by
$\mathfrak{g}\bowtie_{\rho_{\mathfrak{h}}}^{\rho_{\mathfrak{g}}}\mathfrak{h}$.
Conversely, every Lie algebra which is the direct sum of the
underlying vector spaces of two subalgebras can be obtained from a
matched pair of Lie algebras by this construction.


\begin{pro}\label{pro:from matched pairs of anti-pre-Lie algebras to matched pairs of Lie algebras}
    Let $(A,\circ_{A})$ and $(B,\circ_{B})$ be two anti-pre-Lie algebras and their sub-adjacent Lie algebras be $\big(\mathfrak{g}(A),[-,-]_{A}\big)$ and $\big(\mathfrak{g}(B),[-,-]_{B}\big)$ respectively. If $(A,B,l_{\circ_{A}},r_{\circ_{A}},l_{\circ_{B}},r_{\circ_{B}})$ is a matched pair of anti-pre-Lie algebras, then $\big(\mathfrak{g}(A),\mathfrak{g}(  B),l_{\circ_{A}}-r_{\circ_{A}},l_{\circ_{B}}-r_{\circ_{B}}\big)$ is a matched pair of Lie algebras.
\end{pro}
\begin{proof}
    By Theorem \ref{thm:matched pairs of anti-pre-Lie algebras} , there is an anti-pre-Lie algebra $A\bowtie^{l_{\circ_{A}},r_{\circ_{A}}}_{l_{\circ_{B}},r_{\circ_{B}}}B$, whose sub-adjacent Lie algebra is given by
    \begin{small}
    \begin{eqnarray*}
    [x+a,y+b]&=&(x+a)\circ (y+b)-(y+b)\circ (x+a)\\
    &=&x\circ_{A}y+l_{\circ_{B}}(a)y+r_{\circ_{B}}(b)x+l_{\circ_{A}}(x)b+r_{\circ_{A}}(y)a+a\circ_{B}b\\
    &&-\big(y\circ_{A}x+l_{\circ_{B}}(b)x+r_{\circ_{B}}(a)y+l_{\circ_{A}}(y)a+r_{\circ_{A}}(x)b+b\circ_{B}a\big)\\
    &=&[x,y]_{A}+(l_{\circ_{B}}-r_{\circ_{B}})(a)y-(l_{\circ_{B}}-r_{\circ_{B}})(b)x+[a,b]_{B}+(l_{\circ_{A}}-r_{\circ_{A}})(x)b-(l_{\circ_{A}}-r_{\circ_{A}})(y)a,
    \end{eqnarray*}
\end{small}
for all $x,y\in A, a,b\in B$. Thus  $\big(\mathfrak{g}(A),\mathfrak{g}( B),l_{\circ_{A}}-r_{\circ_{A}},l_{\circ_{B}}-r_{\circ_{B}}\big)$ is a matched pair of Lie algebras.
\end{proof}

\subsection{Manin triples of  Lie algebras with respect to the commutative 2-cocycles and anti-pre-Lie bialgebras}

\begin{defi}
Let $(\mathfrak{g},[-,-]_{\mathfrak{g}})$ be a Lie algebra. Assume that there is a Lie algebra structure $(\mathfrak{g}^{*},[-,-]_{\mathfrak{g}^{*}})$ on the dual space $\mathfrak{g}^{*}$. Suppose that there is a Lie algebra structure $(\mathfrak{g}\oplus\mathfrak{g}^{*},[-,-])$ on the direct sum $\mathfrak{g}\oplus\mathfrak{g}^{*}$ of vector spaces such that the nondegenerate symmetric
bilinear form $\mathcal{B}_{d}$ defined by
\begin{equation}\label{eq:defi:Manin triples of Lie algebras}
    \mathcal{B}_{d}(x+a^{*},y+b^{*})=\langle x,b^{*}\rangle+\langle a^{*},y\rangle, \;\;\forall x,y\in A,a^{*},b^{*}\in A^{*},
\end{equation}
is a commutative 2-cocycle on
$(\mathfrak{g}\oplus\mathfrak{g}^{*},[-,-])$, and
$(\mathfrak{g},[-,-]_{\mathfrak{g}})$ and
$(\mathfrak{g}^{*},[-,-]_{\mathfrak{g}^{*}})$ are Lie subalgebras
of $(\mathfrak{g}\oplus\mathfrak{g}^{*},[-,-])$. Such a structure
is called a \textbf{(standard) Manin triple of Lie algebras with
respect to the commutative 2-cocycle}. We denote it by
$\big((\mathfrak{g}\oplus\mathfrak{g}^{*},[-,-],\mathcal{B}_{d}),\mathfrak{g},\mathfrak{g}^{*}\big)$.
\end{defi}

\begin{lem}\label{lem:111}
Let
$\big((\mathfrak{g}\oplus\mathfrak{g}^{*},[-,-],\mathcal{B}_{d}),\mathfrak{g},\mathfrak{g}^{*}\big)$
be a Manin triple of Lie algebras with respect to the commutative
2-cocycle. Then there exists a compatible anti-pre-Lie algebra
structure $\circ$ on $\frak g\oplus \frak g^*$ defined by
Eq.~(\ref{eq:thm:commutative 2-cocycles and anti-pre-Lie
algebras}) through $\mathcal B_d$. Moreover, with this operation,
$(\frak g,\circ_{\frak g}=:\circ|_{\frak g\otimes \frak g})$ and
$(\frak g^*,\circ_{\frak g^*}=:\circ|_{\frak g^*\otimes \frak
g^*})$ are anti-pre-Lie subalgebras whose sub-adjacent Lie
algebras are $(\mathfrak{g},[-,-]_{\mathfrak{g}})$ and
$(\mathfrak{g}^{*},[-,-]_{\mathfrak{g}^{*}})$ respectively.
\end{lem}

\begin{proof}
The first conclusion follows from Theorem~\ref{thm:commutative
2-cocycles and anti-pre-Lie algebras}. Let $x,y\in \frak g$.
Suppose that $x\circ_{\frak g} y=w+w^*$, where $w\in \frak g$ and
$w^*\in \frak g^*$. Then we have
$$\langle w^*,z\rangle=\mathcal B_d(x\circ_{\frak g} y,z)=\mathcal B_d(y, [x,z])=0,\;\;\forall z\in \frak g.$$
Hence $w^*=0$ and thus $x\circ_{\frak g} y\in \frak g$. So $(\frak
g,\circ_{\frak g})$ is an anti-pre-Lie subalgebra. Similarly,
$(\frak g^*,\circ_{\frak g^*})$ is also an anti-pre-Lie
subalgebra. Obviously, the sub-adjacent Lie algebras of $(\frak
g,\circ_{\frak g})$ and $(\frak g^*,\circ_{\frak g^*})$ are
$(\mathfrak{g},[-,-]_{\mathfrak{g}})$ and
$(\mathfrak{g}^{*},[-,-]_{\mathfrak{g}^{*}})$ respectively.
\end{proof}

\delete{\begin{defi}\label{defi:quadratic anti-pre-Lie algebras}
    An anti-pre-Lie algebra $(A,\circ)$ is called \textbf{quadratic} if there exists a nondegenerate symmetric invariant bilinear form $\mathcal{B}$ on it. We denote it by $(A,\circ,\mathcal{B})$.
\end{defi}

\begin{defi}\label{defi:Manin triples of anti-pre-Lie algebras}
    Let $(A,\circ_{A})$ be an anti-pre-Lie algebra. Suppose that there is an anti-pre-Lie algebra $(A^{*},\circ_{A})$ on the dual space $A^{*}$.
    A \textbf{Manin triple of anti-pre-Lie algebras}  associated to $(A,\circ_{A})$ and $(A^{*},\circ_{A^{*}})$ is a collection $((A\oplus A^{*},\circ,\mathcal{B}_{d}),(A,\circ_{A})$,
    $(A^{*},\circ_{A^{*}}))$, such that $(A\oplus A^{*},\circ,\mathcal{B}_{d})$ is a quadratic anti-pre-Lie algebra, where $\mathcal{B}_{d}$ is given by
    \begin{equation}\label{eq:defi:Manin triples of anti-pre-Lie algebras}
        \mathcal{B}_{d}(x+a^{*},y+b^{*})=\langle x,b^{*}\rangle+\langle a^{*},y\rangle, \forall x,y\in A,a^{*},b^{*}\in A^{*},
    \end{equation}
    and $(A,\circ_{A})$ and $(A^{*},\circ_{A^{*}})$ are anti-pre-Lie subalgebras of $(A\oplus A^{*},\circ)$.
\end{defi}}

\begin{thm}\label{thm:Manin triples and matched pairs}
    Let $(A,\circ_{A})$ and $(A^{*},\circ_{A^{*}})$ be two anti-pre-Lie algebras and their sub-adjacent Lie algebras be $\big(\mathfrak{g}(A),[-,-]_{A}\big)$ and $\big(\mathfrak{g}(A^{*}),[-,-]_{A^{*}}\big)$ respectively. Then the following conditions are equivalent:
    \begin{enumerate}
        \item \label{it:APL1}  There is a Manin triple $\Big(\big(\mathfrak{g}(A)\oplus\mathfrak{g}(A^{*}),[-,-],\mathcal{B}_{d}\big),\mathfrak{g}(A),\mathfrak{g}(A^{*})\Big)$ of Lie algebras  with respect to the commutative 2-cocycle such that
         the compatible anti-pre-Lie algebra $(A\oplus A^{*},\circ)$ defined by
Eq.~(\ref{eq:thm:commutative 2-cocycles and anti-pre-Lie
algebras}) through $\mathcal B_d$ contains $(A,\circ_{A})$ and
$(A^{*},\circ_{A^{*}})$ as anti-pre-Lie subalgebras.
        \item \label{it:APL2} $(A,A^{*},-\mathrm{ad}^{*}_{A},\mathcal{R}^{*}_{\circ_{A}},-\mathrm{ad}^{*}_{A^{*}},\mathcal{R}^{*}_{\circ_{A^{*}}})$ is a matched pair of anti-pre-Lie algebras.
        \item \label{it:APL3} $\big(\mathfrak{g}(A),\mathfrak{g}(A^{*}),-\mathcal{L}^{*}_{\circ_{A}},-\mathcal{L}^{*}_{\circ_{A^{*}}}\big)$ is a matched pair of Lie algebras.
    \end{enumerate}
Moreover, every Manin triple of Lie algebras  with respect to the
commutative 2-cocycle can be obtained this way.
\end{thm}

\begin{proof}
$(\ref{it:APL1})\Longrightarrow (\ref{it:APL2})$. By assumption
and Theorem~\ref{thm:matched pairs of anti-pre-Lie algebras},
there are linear maps
$l_{\circ_{A}},r_{\circ_{A}}:A\rightarrow\mathrm{End}(A^{*})$ and
$l_{\circ_{A^{*}}},r_{\circ_{A^{*}}}:A^{*}\rightarrow\mathrm{End}(A)$
such that
$(A,A^{*},l_{\circ_{A}},r_{\circ_{A}},l_{\circ_{A^{*}}},r_{\circ_{A^{*}}})$
is a matched pair of anti-pre-Lie algebras and
    $$x\circ a^{*}=l_{\circ_{A}}(x)a^{*}+r_{\circ_{A^{*}}}(a^{*})x, \;\;a^{*}\circ x=l_{\circ_{A^{*}}}(a^{*})x+r_{\circ_{A}}(x)a^{*},\;\;\forall x\in A, a^*\in A^*.$$
    Then we have
$$\langle l_{\circ_{A}}(x)a^{*}, y\rangle=\mathcal{B}_{d}(x\circ a^{*},
y)=\mathcal{B}_{d}([x,y]_{A},a^{*})=\langle[x,y]_{A},
a^{*}\rangle=\langle -{\rm ad}_A^*(x)a^*, y\rangle,\;\;\forall
x,y\in A,a^{*}\in A^{*}.$$
    Thus $l_{\circ_{A}}=-\mathrm{ad}^*_{A}$, and similarly $l_{\circ_{A^{*}}}=-\mathrm{ad}_{A^{*}}^*$. Moreover, for all $x,y\in A,a^{*}\in A^{*}$, we have
    \begin{eqnarray*}
        \langle a^{*}, x\circ_{A}y\rangle&=& \mathcal{B}_{d}(a^{*},x\circ_{A}y)=\mathcal{B}_{d}([x,a^{*}],y)=\mathcal{B}_{d}(x\circ a^{*}-a^{*}\circ x,y)\\
        &=&\langle(l_{\circ_{A}}-r_{\circ_{A}})(x)a^{*},y\rangle=-\langle a^{*},(l^*_{\circ_{A}}-r^*_{\circ_{A}})(x)y\rangle.
    \end{eqnarray*}
    Since $l_{\circ_{A}}=-\mathrm{ad}^*_{A}$, we have $r^*_{\circ_{A}}=\mathcal{R}_{\circ_{A}}$
and thus $r_{\circ_{A}}=\mathcal{R}^*_{\circ_{A}}$. Similarly we
have $r_{\circ_{A^{*}}}=\mathcal{R}^*_{\circ_{A^{*}}}$. Thus
$(A,A^{*},-\mathrm{ad}^{*}_{A},\mathcal{R}^{*}_{\circ_{A}},-\mathrm{ad}^{*}_{A^{*}}$,
    $\mathcal{R}^{*}_{\circ_{A^{*}}})$ is a matched pair of anti-pre-Lie algebras.

   $(\ref{it:APL2})\Longrightarrow (\ref{it:APL3})$.  It follows from Proposition \ref{pro:from matched pairs of anti-pre-Lie algebras to matched pairs of Lie algebras}.

    $(\ref{it:APL3})\Longrightarrow (\ref{it:APL1})$. Suppose that $\big(\mathfrak{g}(A),\mathfrak{g}(A^{*}),-\mathcal{L}^{*}_{\circ_{A}},-\mathcal{L}^{*}_{\circ_{A^{*}}}\big)$ is a matched pair of Lie algebras.
    Then there is a Lie algebra
    $\mathfrak{g}(A)\bowtie_{-\mathcal{L}^{*}_{\circ_{A^{*}}}}^{-\mathcal{L}^{*}_{\circ_{A}}}\mathfrak{g}(A^{*})$
    with
$$[x,a^{*}]=-\mathcal{L}^{*}_{\circ_{A}}(x)a^{*}+\mathcal{L}^{*}_{\circ_{A^{*}}}(a^{*})x,\;\;\forall x\in A,a^{*}\in A^{*}.$$
It is straightforward to show that  $\mathcal{B}_{d}$ is a
commutative 2-cocycle on this Lie algebra. Therefore
$\Big(\big(\mathfrak{g}(A)\oplus\mathfrak{g}(A^{*}),[-,-],\mathcal{B}_{d}\big),\mathfrak{g}(A),\mathfrak{g}(A^{*})\Big)$
is a Manin triple of Lie algebras with respect to the commutative
2-cocycle. Moreover, by Lemma~\ref{lem:111}, with the compatible
anti-pre-Lie algebra structure  $(A\oplus A^{*},\circ)$ defined by
Eq.~(\ref{eq:thm:commutative 2-cocycles and anti-pre-Lie
algebras}) through $\mathcal B_d$, the two anti-pre-Lie subalgebra
structures on $A$ and $A^*$ are exactly $(A,\circ_{A})$ and
$(A^{*},\circ_{A^{*}})$ respectively, that is,
Item~(\ref{it:APL1}) holds.

Finally, the last conclusion also follows from
Lemma~\ref{lem:111}.
\end{proof}


\begin{defi}\label{defi:anti-pre-Lie coalgebras}
    An \textbf{anti-pre-Lie coalgebra} is a pair $(A,\delta)$, such that $A$ is a vector space and $\delta:A\rightarrow A\otimes A$ is a linear map satisfying
    \begin{equation}\label{eq:defi:anti-pre-Lie coalgebras1}
        (\mathrm{id}^{\otimes 3}-\tau\otimes \mathrm{id})(\mathrm{id}\otimes\delta)\delta=(\tau\otimes \mathrm{id}-\mathrm{id}^{\otimes 3})(\delta\otimes \mathrm{id})\delta,
    \end{equation}
 where $\tau:A\otimes A\rightarrow A\otimes A$ is the flip map defined by $\tau(x\otimes y)=y\otimes x$, for all $x,y\in A$, and
    \begin{equation}\label{eq:defi:anti-pre-Lie coalgebras2}
        (\mathrm{id}^{\otimes 3}+\xi+\xi^{2})(\tau\otimes \mathrm{id}-\mathrm{id}^{\otimes 3})(\delta\otimes \mathrm{id})\delta=0,
    \end{equation}
    where $\xi(x\otimes y\otimes z)=y\otimes z\otimes x$, for all $x,y,z\in A$.
\end{defi}

\begin{pro}\label{pro:anti-pre-Lie coalgebras and anti-pre-Lie algebras}
    Let $A$ be a vector space and $\delta:A\rightarrow A\otimes A$ be a linear map.
    Let $\circ_{A^{*}}:A^{*}\otimes A^{*}\rightarrow A^{*}$ be the linear dual of $\delta$, that is,
        \begin{equation}\label{eq:pro:anti-pre-Lie coalgebras and anti-pre-Lie algebras}
        \langle a^{*}\circ_{A^{*}}b^{*},x\rangle=\langle\delta^{*}(a^{*}\otimes b^{*}),x\rangle=\langle a^{*}\otimes b^{*},\delta(x)\rangle, \;\;\forall a^{*},b^{*}\in A^{*}, x\in A.
    \end{equation}
    Then $(A,\delta)$ is an anti-pre-Lie coalgebra if and only if $(A^{*},\circ_{A^{*}})$ is an anti-pre-Lie algebra.
\end{pro}
\begin{proof}
    For all $x\in A, a^{*},b^{*},c^{*}\in A^{*}$, by Eq.~(\ref{eq:pro:anti-pre-Lie coalgebras and anti-pre-Lie algebras}), we have
\begin{eqnarray*}
    \langle (\mathrm{id}^{\otimes 3}-\tau\otimes \mathrm{id})(\mathrm{id}\otimes\delta)\delta(x),a^{*}\otimes b^{*}\otimes c^{*}\rangle
    &=&\langle x,\delta^{*}(\mathrm{id}\otimes \delta^{*})(\mathrm{id}^{\otimes 3}-\tau\otimes \mathrm{id})(a^{*}\otimes b^{*}\otimes c^{*})\rangle\\
    &=&\langle x,a^{*}\circ_{A^{*}}(b^{*}\circ_{A^{*}} c^{*})-b^{*}\circ_{A^{*}}(a^{*}\circ_{A^{*}}c^{*})\rangle,
\end{eqnarray*}
\begin{eqnarray*}
    \langle (\tau\otimes \mathrm{id}-\mathrm{id}^{\otimes 3})(\delta\otimes \mathrm{id})\delta(x), a^{*}\otimes b^{*}\otimes c^{*}\rangle&=&\langle x, \delta^{*}(\delta^{*}\otimes \mathrm{id})(\tau\otimes
    \mathrm{id}-\mathrm{id}^{\otimes 3})(a^{*}\otimes b^{*}\otimes c^{*})\rangle\\
    &=&\langle x,(b^{*}\circ_{A^{*}} a^{*})\circ_{A^{*}} c^{*}-(a^{*}\circ_{A^{*}} b^{*})\circ_{A^{*}} c^{*}\rangle.
\end{eqnarray*}
    Thus Eq.~(\ref{eq:defi:anti-pre-Lie coalgebras1}) holds if and only if Eq.~(\ref{eq:defi:anti-pre-Lie algebras1}) holds on $A^{*}$. Similarly  Eq.~(\ref{eq:defi:anti-pre-Lie coalgebras2}) holds if and only if Eq.~(\ref{eq:defi:anti-pre-Lie algebras2}) holds on $A^{*}$. Hence the conclusion follows.
\end{proof}

\begin{rmk}
   By \cite{LB2022}, $(A,\circ)$ is an anti-pre-Lie algebra if and only if Eq.~(\ref{eq:defi:anti-pre-Lie algebras1}) and the following equation
 \begin{equation}\label{eq:defi:anti-pre-Lie algebras5}
        x\circ[y,z]+y\circ[z,x]+z\circ[x,y]=0, \;\;\forall x,y,z\in A,
    \end{equation}
hold. Therefore by dualization, $(A,\delta)$ is an anti-pre-Lie
coalgebra if and only if Eq.~(\ref{eq:defi:anti-pre-Lie
coalgebras1}) and the following equation
\begin{equation}\label{eq:defi:anti-pre-Lie coalgebras3}
        (\mathrm{id}^{\otimes 3}+\xi+\xi^{2})(\tau\otimes\mathrm{id}-\mathrm{id}^{\otimes 3})(\mathrm{id}\otimes\delta)\delta=0,
    \end{equation}
hold.
  \end{rmk}

\begin{defi}\label{defi:anti-pre-Lie bialgebra}
    An \textbf{anti-pre-Lie bialgebra} is a triple $(A,\circ,\delta)$, such that the pair $(A,\circ)$ is an anti-pre-Lie algebra, the pair $(A,\delta)$ is an anti-pre-Lie coalgebra, and the following equations hold:
\begin{eqnarray}
     &&(\mathrm{id}^{\otimes 2}-\tau)\Big(\delta(x\circ y)-\big(\mathcal{L}_{\circ}(x)\otimes \mathrm{id}\big)\delta(y)-\big(\mathrm{id}\otimes\mathcal{L}_{\circ}(x)\big)\delta(y)+\big(\mathrm{id}\otimes\mathcal{R}_{\circ}(y)\big)\delta(x)\Big)=0,\label{eq:defi:anti-pre-Lie bialgebra1}\\
     &&\delta([x,y])=\big(\mathrm{id}\otimes\mathrm{ad}(x)-\mathcal{L}_{\circ}(x)\otimes \mathrm{id}\big)\delta(y)-\big(\mathrm{id}\otimes\mathrm{ad}(y)-\mathcal{L}_{\circ}(y)\otimes
        \mathrm{id}\big)\delta(x),\label{eq:defi:anti-pre-Lie bialgebra2}
\end{eqnarray}
\delete{
    \begin{equation}\label{eq:defi:anti-pre-Lie bialgebra1}
        (\mathrm{id}^{\otimes 2}-\tau)(\delta(x\circ y)-(\mathcal{L}_{\circ}(x)\otimes \mathrm{id})\delta(y)-(\mathrm{id}\otimes\mathcal{L}_{\circ}(x))\delta(y)+(\mathrm{id}\otimes\mathcal{R}_{\circ}(y))\delta(x))=0,
    \end{equation}
    \begin{equation}\label{eq:defi:anti-pre-Lie bialgebra2}
        \delta([x,y])=(\mathrm{id}\otimes\mathrm{ad}(x)-\mathcal{L}_{\circ}(x)\otimes \mathrm{id})\delta(y)-(\mathrm{id}\otimes\mathrm{ad}(y)-\mathcal{L}_{\circ}(y)\otimes
        \mathrm{id})\delta(x),
    \end{equation}}
 for all $x,y\in A$.
\end{defi}

For a representation $(\rho,V)$ of a Lie algebra
$(\mathfrak{g},[-,-])$, recall that a \textbf{1-cocycle} of
$(\mathfrak{g},[-,-])$ associated to $(\rho,V)$ is a linear map
$f:\mathfrak{g}\rightarrow V$ such that
\begin{equation}
    f([x,y])=\rho(x)f(y)-\rho(y)f(x),\;\;\forall x,y\in\mathfrak{g}.
\end{equation}
In particular, if there exists $u\in V$ such that $f(x)=\rho(x) u$
for all $x\in \frak g$, then $f$ is called a {\bf 1-coboundary}.
Any 1-coboundary is a 1-cocycle.

For two representations $(\rho,V)$ and $(\phi,W)$ of a Lie algebra
$(\mathfrak{g},[-,-])$, it is known that
$(\rho\otimes\mathrm{id}+\mathrm{id}\otimes\phi, V\otimes W)$ is
also a representation of $(\mathfrak{g},[-,-])$, where
\begin{equation*}
(\rho\otimes\mathrm{id}+\mathrm{id}\otimes\phi)(x)(v\otimes w)
=\rho(x)v\otimes w+v\otimes\phi(x)w, \;\;\forall x\in\mathfrak{g}, v\in V,w\in W.
\end{equation*}


\begin{thm}\label{thm:equivalence matched pairs of Lie algebras and anti-pre-Lie bialgebras}
    Let $(A,\circ_{A})$ be an anti-pre-Lie algebra. Suppose that there is an anti-pre-Lie algebra structure $(A^{*},\circ_{A^{*}})$ on the dual space $A^{*}$.
Let $\big(\mathfrak{g}(A),[-,-]_{A}\big)$ and
$\big(\mathfrak{g}(A^{*}),[-,-]_{A^{*}}\big)$ be the sub-adjacent Lie
algebras of $(A,\circ_{A})$ and $(A^{*},\circ_{A^{*}})$
respectively. Let $\delta:A\rightarrow A\otimes A$ and
$\beta:A^{*}\rightarrow A^{*}\otimes A^{*}$ be linear duals of
$\circ_{A^{*}}$ and $\circ_{A}$ respectively. Then the following
conditions are equivalent:
    \begin{enumerate}
            \item\label{it:aa1} $(A,\circ_{A},\delta)$ is an anti-pre-Lie bialgebra.
        \item\label{it:aa2} $\delta$ is a 1-cocycle of $\big(\mathfrak{g}(A),[-,-]_{A}\big)$ associated to $(-\mathcal{L}_{\circ_{A}}\otimes\mathrm{id}+\mathrm{id}\otimes\mathrm{ad}_{A})$,
        and $\beta$ is a 1-cocycle of $\big(\mathfrak{g}(A^{*}),[-,-]_{A^{*}}\big)$ associated to $(-\mathcal{L}_{\circ_{A^{*}}}\otimes\mathrm{id}+\mathrm{id}\otimes\mathrm{ad}_{A^{*}})$.
        \item\label{it:aa3} $\big(\mathfrak{g}(A),\mathfrak{g}(A^{*}),-\mathcal{L}^{*}_{\circ_{A}},-\mathcal{L}^{*}_{\circ_{A^{*}}}\big)$ is a matched pair of Lie algebras.
    \end{enumerate}
\end{thm}
\begin{proof}
$(\ref{it:aa1})\Longleftrightarrow (\ref{it:aa2})$. Obviously
Eq.~(\ref{eq:defi:anti-pre-Lie bialgebra2}) holds if and only if
$\delta$ is a 1-cocycle of $(\mathfrak{g}(A),[-,-]_{A})$
associated to
$(-\mathcal{L}_{\circ_{A}}\otimes\mathrm{id}+\mathrm{id}\otimes\mathrm{ad}_{A})$.
Let $x,y\in A, a^{*},b^{*}\in A^{*}$. Then we have {\small
\begin{eqnarray*} \langle\beta([a^{*},b^{*}]_{A^{*}}),x\otimes
y\rangle&=&
\langle a^{*}\circ_{A^{*}}b^{*}, x\circ_{A}y\rangle-\langle
b^{*}\circ_{A^{*}}a^{*}, x\circ_{A}y\rangle\\
&=&\langle a^*\otimes b^*, (\mathrm{id}^{\otimes 2}-\tau)\delta(x\circ_{A}y)\rangle,\\
\langle
\big(\mathrm{id}\otimes\mathrm{ad}_{A^{*}}(b^{*})-\mathcal{L}_{\circ_{A^{*}}}(b^{*})\otimes\mathrm{id}\big)\beta(a^{*}),
x\otimes y\rangle
&=&-\langle\beta(a^{*}),x\otimes\mathrm{ad}^{*}_{A^{*}}(b^{*})y\rangle+\langle\beta(a^{*}),\mathcal{L}^{*}_{\circ_{A^{*}}}(b^{*})x\otimes y\rangle\\
&=&-\langle a^{*},x\circ_{A}\mathrm{ad}^{*}_{A^{*}}(b^{*})y\rangle+\langle a^{*},\mathcal{L}^{*}_{\circ_{A^{*}}}(b^{*})x\circ_{A}y\rangle\\
&=&\langle[\mathcal{L}^{*}_{\circ_{A}}(x)a^{*},b^{*}]_{A^{*}},y\rangle+\langle b^{*}\circ_{A^{*}}\mathcal{R}^{*}_{\circ_{A}}(y)a^{*},x\rangle\\
&=&\langle a^{*}\otimes b^{*} ,\tau\big(\mathrm{id}\otimes\mathcal{L}_{\circ_{A}}(x)\big)\delta(y)-\tau\big(\mathrm{id}\otimes\mathcal{R}_{\circ_{A}}(y)\big)\delta(x)\\
&&\hspace{1.6cm}-\big(\mathcal{L}_{\circ_{A}}(x)\otimes\mathrm{id}\big)\delta(y)\rangle,\\
-\langle
\big(\mathrm{id}\otimes\mathrm{ad}_{A^{*}}(a^{*})-\mathcal{L}_{\circ_{A^{*}}}(a^{*})\otimes\mathrm{id}\big)\beta(b^{*}),
x\otimes y\rangle&=& \langle a^{*}\otimes b^{*},
-\big(\mathrm{id}\otimes\mathcal{L}_{\circ_{A}}(x)\big)\delta(y)+\big(\mathrm{id}\otimes\mathcal{R}_{\circ_{A}}(y)\big)\delta(x)\\
&&\hspace{1.6cm}+\tau\big(\mathcal{L}_{\circ_{A}}(x)\otimes\mathrm{id}\big)\delta(y)\rangle.
\end{eqnarray*}}
Hence $\beta$ is a 1-cocycle of $\big(\mathfrak{g}(A^{*}),[-,-]_{A^{*}}\big)$ associated to $(-\mathcal{L}_{\circ_{A^{*}}}\otimes\mathrm{id}+\mathrm{id}\otimes\mathrm{ad}_{A^{*}})$
if and only if Eq.~(\ref{eq:defi:anti-pre-Lie bialgebra1}) holds.

$(\ref{it:aa2})\Longleftrightarrow (\ref{it:aa3})$.
    For all $x,y\in A, a^{*},b^{*}\in A^{*}$, we have
    \begin{eqnarray*}
        -\langle\mathcal{L}^{*}_{\circ_A^{*}}(a^{*})([x,y]_{A}),b^{*}\rangle&=&\langle[x,y]_{A},a^{*}\circ_{A^{*}}b^{*}\rangle=\langle\delta([x,y]_{A}),a^{*}\otimes b^{*}\rangle,\\
        \langle[\mathcal{L}^{*}_{\circ_{A^{*}}}(a^{*})x,y]_{A},b^{*}\rangle&=&\langle\mathcal{L}^{*}_{\circ_{A^{*}}}(a^{*})x,\mathrm{ad}^{*}_{A}(y)b^{*}\rangle=-\langle x,a^{*}\circ_{A^{*}}\mathrm{ad}^{*}_{A}(y)b^{*}\rangle\\
        &=&-\langle x,\delta^{*}\big(\mathrm{id}\otimes\mathrm{ad}^{*}_{A}(y)\big)(a^{*}\otimes b^{*})\rangle=\langle\big(\mathrm{id}\otimes\mathrm{ad}_{A}(y)\big)\delta(x),a^{*}\otimes b^{*}\rangle,\\
        \langle[x,\mathcal{L}^{*}_{\circ_{a^{*}}}(a^{*})y]_{A},b^{*}\rangle&=&-\langle\big(\mathrm{id}\otimes\mathrm{ad}_{A}(x)\big)\delta(y),a^{*}\otimes b^{*}\rangle,\\
        \langle\mathcal{L}^{*}_{\circ_{A^{*}}}\big(\mathcal{L}^{*}_{\circ_{A}}(x)a^{*}\big)y,b^{*}\rangle&=&-\langle y,\mathcal{L}^{*}_{\circ_{A}}(x)a^{*}\circ_{A^{*}}b^{*}\rangle=-\langle y,\delta^{*}\big(\mathcal{L}^{*}_{\circ_{A}}(x)\otimes \mathrm{id}\big)(a^{*}\otimes b^{*})\rangle\\
        &=&\langle\big(\mathcal{L}_{\circ_{A}}(x)\otimes \mathrm{id}\big)\delta(y),a^{*}\otimes b^{*}\rangle,\\
        -\langle\mathcal{L}^{*}_{\circ_{A^{*}}}\big(\mathcal{L}^{*}_{\circ_{A}}(y)a^{*}\big)x,b^{*}\rangle&=&-\langle\big(\mathcal{L}_{\circ_{A}}(y)\otimes \mathrm{id}\big)\delta(x),a^{*}\otimes b^{*}\rangle.
    \end{eqnarray*}
 Thus $\delta$ is a 1-cocycle of $\big(\mathfrak{g}(A),[-,-]_{A}\big)$ associated to $(-\mathcal{L}_{\circ_{A}}\otimes\mathrm{id}+\mathrm{id}\otimes\mathrm{ad}_{A})$  if and only if Eq.~(\ref{eq:matched pair of Lie algebras2}) holds for
 $\rho_{\frak g(A)}=-\mathcal{L}^{*}_{\circ_{A}}, \rho_{\frak g (A^{*})}=-\mathcal{L}^{*}_{\circ_{A^{*}}}$. By symmetry,
 $\beta$ is a 1-cocycle of $\big(\mathfrak{g}(A^{*}),[-,-]_{A^{*}}\big)$ associated to $(-\mathcal{L}_{\circ_{A^{*}}}\otimes\mathrm{id}+\mathrm{id}\otimes\mathrm{ad}_{A^{*}})$ if and only
 if Eq.~(\ref{eq:matched pair of Lie algebras1}) holds for $\rho_{\frak g(A)}=-\mathcal{L}^{*}_{\circ_{A}}, \rho_{\frak g (A^{*})}=-\mathcal{L}^{*}_{\circ_{A^{*}}}$.
\end{proof}

Combining Theorems \ref{thm:Manin triples and matched pairs} and
\ref{thm:equivalence matched pairs of Lie algebras and
anti-pre-Lie bialgebras} together, we have

\begin{cor}\label{cor:matched pairs of anti-pre-Lie algebras, Manin triples of anti-pre-Lie algebras, anti-pre-Lie bialgebras and their equivalence}
    Let $(A,\circ_{A})$ be an anti-pre-Lie algebra. Suppose that there is an anti-pre-Lie algebra structure $(A^{*},\circ_{A^{*}})$ on the dual space $A^{*}$ and $\delta:A\rightarrow A\otimes A$ is the linear dual of $\circ_{A^{*}}$.
Let $\big(\mathfrak{g}(A),[-,-]_{A}\big)$ and
$\big(\mathfrak{g}(A^{*}),[-,-]_{A^{*}}\big)$ be the sub-adjacent Lie
algebras of $(A,\circ_{A})$ and $(A^{*},\circ_{A^{*}})$
respectively.
    Then the following  conditions are equivalent:
    \begin{enumerate}
        \item There is a Manin triple $\Big(\big(\mathfrak{g}(A)\oplus\mathfrak{g}(A^{*}),[-,-],\mathcal{B}_{d}\big),\mathfrak{g}(A),\mathfrak{g}(A^{*})\Big)$ of Lie algebras  with respect to the commutative 2-cocycle such that
         the compatible anti-pre-Lie algebra $(A\oplus A^{*},\circ)$ defined by
Eq.~(\ref{eq:thm:commutative 2-cocycles and anti-pre-Lie
algebras}) through $\mathcal B_d$ contains $(A,\circ_{A})$ and
$(A^{*},\circ_{A^{*}})$ as anti-pre-Lie subalgebras.
    \item $(A,A^{*},-\mathrm{ad}^{*}_{A},\mathcal{R}^{*}_{\circ_{A}},-\mathrm{ad}^{*}_{A^{*}},\mathcal{R}^{*}_{\circ_{A^{*}}})$ is a matched pair of anti-pre-Lie algebras.
    \item $\big(\mathfrak{g}(A),\mathfrak{g}(A^{*}),-\mathcal{L}^{*}_{\circ_{A}},-\mathcal{L}^{*}_{\circ_{A^{*}}}\big)$ is a matched pair of Lie algebras.
    \item $(A,\circ_{A},\delta)$ is an anti-pre-Lie bialgebra.
    \end{enumerate}
\end{cor}

\subsection{Coboundary anti-pre-Lie bialgebras and the anti-pre-Lie Yang-Baxter equation}

\begin{defi}\label{defi:coboundary anti-pre-Lie bialgebras}
    An anti-pre-Lie bialgebra $(A,\circ,\delta)$
    is called \textbf{coboundary} if $\delta:A\rightarrow A\otimes A$
    is a 1-coboundary of $\big(\mathfrak{g}(A),[-,-]\big)$ associated to $(-\mathcal{L}_{\circ}\otimes\mathrm{id}+\mathrm{id}\otimes\mathrm{ad})$, that is,  there exists an $r\in A\otimes A$ such that
    \begin{equation}\label{eq:defi:coboundary anti-pre-Lie bialgebras}
        \delta(x):=\delta_{r}(x):=\big(-\mathcal{L}_{\circ}(x)\otimes \mathrm{id}+\mathrm{id}\otimes\mathrm{ad}(x)\big)r,\;\; \forall x\in A.
    \end{equation}
\end{defi}


\begin{pro}\label{pro:cob coalg}
Let $(A,\circ)$ be an anti-pre-Lie algebra and
$r=\sum\limits_{i}a_{i}\otimes b_{i}\in A\otimes A$. Let
$\delta=\delta_{r}:A\rightarrow A\otimes A$ be a linear map
defined by Eq.~(\ref{eq:defi:coboundary anti-pre-Lie bialgebras}).
Denote
\begin{equation}
    \textbf{T}(r)=r_{12}\circ r_{13}+r_{12}\circ r_{23}-[r_{13},r_{23}],
\end{equation}
where
\begin{equation}
r_{12}\circ r_{13}=\sum_{i,j}a_{i}\circ a_{j}\otimes b_{i}\otimes b_{j},  r_{12}\circ r_{23}=\sum_{i,j}a_{i}\otimes b_{i}\circ a_{j}\otimes b_{j}, [r_{13},r_{23}]=\sum_{i,j}a_{i}\otimes a_{j}\otimes [b_{i},b_{j}].
\end{equation}
\begin{enumerate}
    \item \label{it:bb1}Eq.~(\ref{eq:defi:anti-pre-Lie coalgebras1}) holds if and only if for all $x\in
    A$, the following equation holds:
    \begin{equation}\label{eq:pro:cob coalg1}
        \begin{split}
            &\Big(\mathcal{L}_{\circ }(x)\otimes\mathrm{id}\otimes\mathrm{id}-(\tau\otimes\mathrm{id})\big(\mathcal{L}_{\circ }(x)\otimes\mathrm{id}\otimes\mathrm{id}\big)-\mathrm{id}\otimes\mathrm{id}\otimes\mathrm{ad}(x)\Big)\textbf{T}(r)\\
            &+\sum_{j}\big(\mathrm{id}\otimes\mathcal{L}_{\circ}(a_{j})\otimes\mathrm{ad}(x)-\mathrm{ad}(a_{j})\otimes\mathrm{id}\otimes\mathrm{ad}(x)\big)\Big(\big(r+\tau(r)\big)\otimes b_{j}\Big)\\
            &+(\mathrm{id}^{\otimes 3}-\tau\otimes\mathrm{id})\sum_{j}\big(\mathcal{L}_{\circ}(x\circ a_{j})\otimes\mathrm{id}\otimes\mathrm{id}-\mathcal{L}_{\circ}(x)\mathcal{R}_{\circ}(a_{j})\otimes\mathrm{id}\otimes\mathrm{id}\big)\Big(\big(r+\tau(r)\big)\otimes b_{j}\Big)=0.
        \end{split}
    \end{equation}
    \item \label{it:bb2} Eq.~(\ref{eq:defi:anti-pre-Lie coalgebras2}) holds if and only if for all $x\in
    A$, the following equation holds:
    \begin{small}
\begin{equation}\label{eq:pro:cob coalg2}
    \begin{split}
        &(\mathrm{id}^{\otimes 3}+\xi+\xi^{2})\bigg(-\big(\mathrm{id}\otimes\mathrm{id}\otimes\mathcal{L}_{\circ}(x)\big)\textbf{T}(r)+\sum_{j}\big(\mathrm{id}\otimes\mathrm{id}\otimes\mathcal{L}_{\circ}([x,b_{j}])\big)(\mathrm{id}^{\otimes 3}-\tau\otimes\mathrm{id})\Big(a_{j}\otimes \big(r+\tau(r)\big)\Big)\\
        &+\sum_{j}\big(\mathrm{id}\otimes\mathrm{ad}(b_{j})\otimes \mathcal{L}_{\circ}(x)\big)\Big(a_{j}\otimes \big(r+\tau(r)\big)\Big)+\sum_{j}\big(\mathrm{id}\otimes\mathrm{ad}(a_{j})\otimes \mathcal{L}_{\circ}(x)\big)\Big(b_{j}\otimes \big(r+\tau(r)\big)\Big)\\
        &+\sum_{j}\big(\mathcal{L}_{\circ}(a_{j})\otimes\mathrm{id}\otimes \mathcal{L}_{\circ}(x)\big)(\tau\otimes\mathrm{id})\Big(b_{j}\otimes \big(r+\tau(r)\big)\Big)+\sum_{j}\big(\mathrm{ad}(b_{j})\otimes \mathrm{id}\otimes\mathcal{L}_{\circ}(x)\big)
        \Big(\big(r+\tau(r)\big)\otimes a_{j}\Big)\bigg)=0.
    \end{split}
\end{equation}
\end{small}
    \item \label{it:bb3} Eq.~(\ref{eq:defi:anti-pre-Lie bialgebra1}) holds if and only if for all $x,y\in A$, the following equation holds:
\begin{small}
    \begin{equation}\label{eq:pro:coboundary anti-pre-Lie bialgebras1}
        \big(\mathrm{id}\otimes\mathcal{L}_{\circ}(x\circ y)-\mathrm{id}\otimes\mathcal{L}_{\circ}(x)\mathcal{L}_{\circ}(y)+\mathcal{L}_{\circ}(x)\mathcal{L}_{\circ}(y)\otimes \mathrm{id}-\mathcal{L}_{\circ}(x\circ y)\otimes \mathrm{id}+\mathcal{L}_{\circ}(y)\otimes\mathcal{L}_{\circ}(x)-\mathcal{L}_{\circ}(x)\otimes\mathcal{L}_{\circ}(y)\big)\big(r+\tau(r)\big)=0.
    \end{equation}
\end{small}
\end{enumerate}
\end{pro}

\begin{proof}
The proof follows from the careful interpretation. It seems a
little lengthy and hence we put it into the Appendix.
\end{proof}

\begin{thm}\label{thm:cob bialg}
Let $(A,\circ)$ be an anti-pre-Lie algebra and
$r=\sum\limits_{i}a_{i}\otimes b_{i}\in A\otimes A$. Let
$\delta=\delta_{r}$ be a linear map defined by
Eq.~(\ref{eq:defi:coboundary anti-pre-Lie bialgebras}). Then
$(A,\circ,\delta)$ is an anti-pre-Lie bialgebra  if and only if
Eqs.~(\ref{eq:pro:cob coalg1})-(\ref{eq:pro:coboundary
anti-pre-Lie bialgebras1}) hold.
\end{thm}

\begin{proof}
Obviously Eq.~(\ref{eq:defi:anti-pre-Lie bialgebra2}) holds
automatically. By Proposition~\ref{pro:cob coalg}, the conclusion
holds.
\end{proof}

The simplest solutions to satisfy Eqs.~(\ref{eq:pro:cob
coalg1})-(\ref{eq:pro:coboundary anti-pre-Lie bialgebras1}) are to
assume that $r$ is skew-symmetric (that is, $\tau(r)=-r$) and
$\textbf{T}(r)=0$, that is,

\begin{cor}\label{cor:cob bialg}
Let $(A,\circ)$ be an anti-pre-Lie algebra and $r\in A\otimes A$
be skew-symmetric. If $\textbf{T}(r)=0$, then $(A,\circ,\delta)$
is an anti-pre-Lie bialgebra, where $\delta=\delta_{r}$ is defined
by Eq.~(\ref{eq:defi:coboundary anti-pre-Lie bialgebras}).
\end{cor}

\begin{defi}
Let $(A,\circ)$ be an anti-pre-Lie algebra and $r\in A\otimes A$.
We say $r$ is a solution of the \textbf{anti-pre-Lie Yang-Baxter
equation} or \textbf{APL-YBE} in short, in $(A,\circ)$  if
$\textbf{T}(r)=0$.
\end{defi}

\begin{ex}\label{ex:AYBE}
Let $(A,\cdot)$ be a commutative associative algebra with a
derivation $P$.  Then according to \cite{LB2022}, there is an
anti-pre-Lie algebra $(A,\circ)$ given by
\begin{equation}\label{eq:diff anti-pre-Lie}
    x\circ y=P(x\cdot y)+P(x)\cdot y,\;\;\forall x,y\in A.
\end{equation}
Let $r\in A\otimes A$ be a solution of the {\bf associative
Yang-Baxter equation (AYBE)} in $(A,\cdot)$, that is, $r$
satisfies
\begin{equation}\label{eq:AYBE}{\bf A}(r):=r_{12}\cdot r_{13}-r_{23}\cdot r_{12}+r_{13}\cdot
r_{23}=0,\end{equation} where for $r=\sum\limits_{i} a_{i}\otimes
b_{i}$,
\begin{equation}
r_{12}\cdot r_{13}=\sum_{i,j} a_{i}\cdot a_{j}\otimes b_{i}\otimes b_{j},\  r_{23}\cdot r_{12}=\sum_{i,j}a_{j}\otimes a_{i}\cdot b_{j}\otimes b_{i},\  r_{13}\cdot r_{23}=\sum_{i,j} a_{i}\otimes a_{j}\otimes b_{i}\cdot b_{j}.
\end{equation}
If in addition, $r$ satisfies
\begin{equation}\label{eq:-P}
        (\mathrm{id}\otimes P+P\otimes\mathrm{id})r=0,
    \end{equation}
then $r$ is also a solution of the APL-YBE in $(A,\circ)$. In fact, we have
 \begin{eqnarray*}
    \textbf{T}(r)&=&r_{12}\circ r_{13}+r_{12}\circ r_{23}-[r_{13}, r_{23}]\\
    &=&\sum_{i,j}\big(P(a_{i}\cdot a_{j})\otimes b_{i}\otimes b_{j}+P(a_{i})\cdot a_{j}\otimes b_{i}\otimes b_{j}+a_{i}\otimes P(b_{i}\cdot a_{j})\otimes b_{j}\\
    &&\ +a_{i}\otimes P(b_{i})\cdot a_{j}\otimes b_{j}+a_{i}\otimes a_{j}\otimes b_{i}\cdot P(b_{j})-a_{i}\otimes a_{j}\otimes P(b_{i})\cdot b_{j}\big)\\
    &{\overset{(\ref{eq:-P})}{=}}&\sum_{i,j}\big(P(a_{i}\cdot a_{j})\otimes b_{i}\otimes b_{j}-a_{i}\cdot a_{j}\otimes P(b_{i})\otimes b_{j}+a_{i}\otimes P(b_{i}\cdot a_{j})\otimes b_{j}\\
    &&\ -P(a_{i})\otimes b_{i}\cdot a_{j}\otimes b_{j}-a_{i}\otimes P(a_{j})\otimes b_{i}\cdot b_{j}+P(a_{i})\otimes a_{j}\otimes b_{i}\cdot b_{j}\big)\\
    &=&(P\otimes\mathrm{id}\otimes\mathrm{id}-\mathrm{id}\otimes P\otimes\mathrm{id})(r_{12}\cdot r_{13}-r_{23}\cdot r_{12}+r_{13}\cdot r_{23})\\
    &=&0.
    \end{eqnarray*}
\end{ex}

\delete{
\begin{defi}
    Let $(A,\circ)$ be an anti-pre-Lie algebra. A bilinear form $\mathcal{B}$ on $(A,\circ)$ is called an \textbf{anti-2-cocycle} if
    \begin{equation}
        \mathcal{B}(x,y\circ z)-    \mathcal{B}(y,x\circ z)=\mathcal{B}([y,x],z),\;\;\forall x,y,z\in A.
    \end{equation}
\end{defi}

Then we have the following theorem.

For a vector space $A$, the isomorphism $A\otimes A\cong {\rm
Hom}(A^*,A)$ identifies an $r\in A\otimes A$ with a map from $A^*$
to $A$ which we still denote by $r$. Explicitly, writing
$r=\sum\limits_{i}a_{i}\otimes b_{i}\in A\otimes A$, we have
\begin{equation}\label{eq:identify}
    r:A^{*}\rightarrow A,\;\; r(a^{*})=\sum_{i}\langle a^{*}, a_{i}\rangle b_{i},\;\;\forall a^{*}\in A^{*}.
\end{equation}

\begin{thm}\label{thm:anti-2-cocycle}
    Let $(A,\circ)$ be an anti-pre-Lie algebra and $r\in A\otimes A$. Suppose that $r$ is skew-symmetric and nondegenerate. Then $r$ is a solution of the APL-YBE if and only if the bilinear form $\mathcal{B}_{r}$ given by
    \begin{equation}\label{eq:Br}
        \mathcal{B}_{r}(x,y)=\langle r^{-1}(x), y\rangle,\;\;\forall x,y\in A
    \end{equation}
is an anti-2-cocycle on $(A,\circ)$.
\end{thm}
\begin{proof}
    Let $r=\sum\limits_{i}a_{i}\otimes b_{i}$. Since $r$ is skew-symmetric, we have
    $\sum\limits_{i}a_{i}\otimes b_{i}=-\sum\limits_{i}b_{i}\otimes a_{i}$.
    Therefore $r(v^{*})=\sum\limits_{i}\langle v^{*}, a_{i}\rangle b_{i}=-\sum\limits_{i}\langle v^{*}, b_{i}\rangle a_{i},\;\;\forall v^{*}\in A^{*}$.
    Since $r$ is nondegenerate, for all $x,y,z\in A$, there exist $u^{*}, v^{*}, w^{*}\in A^{*}$ such that $x=r(u^{*}), y=r(v^{*}), z=r(w^{*})$.
    By Eq.~(\ref{eq:Br}), $\mathcal{B}_{r}$ is nondegenerate and skew-symmetric, and we have
    \begin{eqnarray*}
        \mathcal{B}_{r}([x,y],z)&=&-\langle [r(u^{*}), r(v^{*})],w^{*}\rangle=-\sum_{i,j}\langle u^{*}, a_{i}\rangle\langle v^{*}, a_{j}\rangle\langle w^{*},[b_{i},b_{j}]\rangle\\
        &=&-\sum_{i,j}\langle u^{*}\otimes v^{*}\otimes w^{*}, a_{i}\otimes a_{j}\otimes[b_{i},b_{j}]\rangle,\\
        \mathcal{B}_{r}(x,y\circ z)&=&\langle u^{*}, r(v^{*})\circ r(w^{*})\rangle=\sum_{i,j}\langle v^{*}, b_{i}\rangle\langle w^{*},b_{j}\rangle\langle u^{*}, a_{i}\circ a_{j}\rangle\\
        &=&\sum_{i,j}\langle u^{*}\otimes v^{*}\otimes w^{*}, a_{i}\circ a_{j}\otimes b_{i}\otimes b_{j}\rangle,\\
        -\mathcal{B}_{r}(y,x\circ z)&=&-\langle v^{*}, r(u^{*})\circ r(w^{*})\rangle=\sum_{i,j}\langle u^{*}, a_{i}\rangle\langle w^{*}, b_{j}\rangle\langle v^{*}, b_{i}\circ a_{j}\rangle\\
        &=&\sum_{i,j}\langle u^{*}\otimes v^{*}\otimes w^{*}, a_{i}\otimes b_{i}\circ a_{j}\otimes b_{j}\rangle.
    \end{eqnarray*}
Thus
$$\mathcal{B}_{r}([x,y],z)+\mathcal{B}_{r}(x,y\circ z)-\mathcal{B}_{r}(y,x\circ z)=\langle u^{*}\otimes v^{*}\otimes w^{*}, \textbf{T}(r)\rangle.$$
Hence the conclusion follows.
    \end{proof}
}

For a vector space $A$, the isomorphism $A\otimes A\cong {\rm
Hom}(A^*,A)$ identifies an $r\in A\otimes A$ with a map from $A^*$
to $A$ which we still denote by $r$. Explicitly, writing
$r=\sum\limits_{i}a_{i}\otimes b_{i}\in A\otimes A$, we have
\begin{equation}\label{eq:identify}
    r:A^{*}\rightarrow A,\;\; r(a^{*})=\sum_{i}\langle a^{*}, a_{i}\rangle b_{i},\;\;\forall a^{*}\in A^{*}.
\end{equation}

\begin{thm}\label{thm:O-operator and T-equation}
    Let $(A,\circ)$ be an anti-pre-Lie algebra and $r\in A\otimes A$ be skew-symmetric. Then the following conditions are equivalent.
\begin{enumerate}
\item\label{it:ss1} $r$ is a solution of the APL-YBE in
$(A,\circ)$. \item\label{it:ss2} The following equation holds.
\begin{equation}\label{eq:APL-O}
    r(a^{*})\circ r(b^{*})+r\Big(\mathrm{ad}^{*}\big(r(a^{*})\big)b^{*}\Big)-r\Big(\mathcal{R}^{*}_{\circ}\big(r(b^{*})\big)a^{*}\Big)=0,\;\;\forall a^{*},b^{*}\in A^{*}.
\end{equation}
\item\label{it:ss3} The following equation holds.
\begin{equation}\label{eq:Lie-O}
        [r(a^{*}),r(b^{*})]+r\Big(\mathcal{L}^{*}_{\circ}\big(r(a^{*})\big)b^{*}\Big)-r\Big(\mathcal{L}^{*}_{\circ}\big(r(b^{*})\big)a^{*}\Big)=0,\;\;\forall a^{*},b^{*}\in A^{*}.
    \end{equation}
\end{enumerate}
\end{thm}

\begin{proof}
Let $r=\sum\limits_{i}a_{i}\otimes b_{i}\in A\otimes A$ and $a^{*},b^{*},c^{*}\in A^{*}$.

(\ref{it:ss1}) $\Longleftrightarrow$ (\ref{it:ss2}).  By
Eq.~(\ref{eq:identify}), we have
\begin{small}
\begin{eqnarray*}
    \langle r(a^{*})\circ r(b^{*}), c^{*}\rangle&=&\sum_{i,j}\langle a^{*},a_{i}\rangle\langle b^{*},a_{j}\rangle\langle b_{i}\circ b_{j}, c^{*}\rangle=\sum_{i,j}\langle c^{*}\otimes a^{*}\otimes b^{*}, b_{i}\circ b_{j}\otimes a_{i}\otimes a_{j}\rangle\\
    &=&\sum_{i,j}\langle c^{*}\otimes a^{*}\otimes b^{*}, a_{i}\circ a_{j}\otimes b_{i}\otimes b_{j}\rangle,\\
    \langle r\Big(\mathrm{ad}^{*}\big(r(a^{*})\big)b^{*}\Big), c^{*}\rangle&=&\sum_{i}\langle\mathrm{ad}^{*}\big(r(a^{*})\big)b^{*}, a_{i}\rangle\langle b_{i},c^{*}\rangle=-\sum_{i}\langle b^{*},[r(a^{*}),a_{i}]\rangle\langle b_{i},c^{*}\rangle\\
    &=&-\sum_{i,j}\langle a^{*}, a_{j}\rangle\langle b^{*},[b_{j},a_{i}]\rangle\langle b_{i},c^{*}\rangle=-\sum_{i,j}\langle c^{*}\otimes a^{*}\otimes b^{*}, b_{i}\otimes a_{j}\otimes [b_{j},a_{i}]\rangle\\
    &=&-\sum_{i,j}\langle c^{*}\otimes a^{*}\otimes b^{*}, a_{i}\otimes a_{j}\otimes [b_{i},b_{j}]\rangle,\\
    -\langle r\Big(\mathcal{R}^{*}_{\circ}\big(r(b^{*})\big)a^{*}\Big),c^{*}\rangle&=&-\sum_{i}\langle \mathcal{R}^{*}_{\circ}\big(r(b^{*})\big)a^{*},a_{i}\rangle\langle b_{i},c^{*}\rangle=\sum_{i}\langle a^{*},a_{i}\circ r(b^{*})\rangle\langle b_{i},c^{*}\rangle\\
    &=&\sum_{i,j}\langle b^{*},a_{j}\rangle\langle a^{*}, a_{i}\circ b_{j}\rangle\langle b_{i},c^{*}\rangle=\sum_{i,j}\langle c^{*}\otimes a^{*}\otimes b^{*}, b_{i}\otimes a_{i}\circ b_{j}\otimes a_{j}\rangle\\
    &=&\sum_{i,j}\langle c^{*}\otimes a^{*}\otimes b^{*}, a_{i}\otimes b_{i}\circ a_{j}\otimes b_{j}\rangle.
\end{eqnarray*}
\end{small}
Hence Item~(\ref{it:ss1}) holds if and only if
Eq.~(\ref{eq:APL-O}) holds.

(\ref{it:ss1}) $\Longleftrightarrow$ (\ref{it:ss3}). By a similar
study as above, we have
\begin{eqnarray*}
               \langle [r(a^{*}), r(b^{*})], c^{*}\rangle
                &=&\sum_{i,j}\langle a^{*}\otimes b^{*}\otimes c^{*}, a_{i}\otimes a_{j}\otimes [b_{i},b_{j}]\rangle,\\
               \langle r\Big(\mathcal{L}^{*}_{\circ}\big(r(a^{*})\big)b^{*}\Big),c^{*}\rangle
                &=&-\sum_{i,j}\langle a^{*}\otimes b^{*}\otimes c^{*}, a_{i}\otimes b_{i}\circ a_{j}\otimes b_{j}\rangle,\\
                -\langle r\Big(\mathcal{L}^{*}_{\circ}\big(r(b^{*})\big)a^{*}\Big),c^{*}\rangle
                &=&-\sum_{i,j}\langle a^{*}\otimes b^{*}\otimes c^{*}, a_{i}\circ a_{j}\otimes b_{i}\otimes b_{j}\rangle.
            \end{eqnarray*}
Hence Item~(\ref{it:ss1}) holds if and only if
Eq.~(\ref{eq:Lie-O}) holds.
\end{proof}

\subsection{$\mathcal{O}$-operators of anti-pre-Lie algebras and \papl  algebras}\

Theorem~\ref{thm:O-operator and T-equation} motivates to give the following notion.

\begin{defi}\label{defi:O-operators}
    Let $(A,\circ)$ be an anti-pre-Lie algebra and $(l_{\circ},r_{\circ},V)$ be a representation of $(A,\circ)$. A linear map $T:V\rightarrow A$ is called \textbf{an $\mathcal{O}$-operator of $(A,\circ)$ associated to} $(l_{\circ},r_{\circ},V)$ if
    \begin{equation}\label{eq:defi:O-operators}
        T(u)\circ T(v)=T\Big(l_{\circ}\big(T(u)\big)v+r_{\circ}\big(T(v)\big)u\Big),\;\;\forall u,v\in V.
    \end{equation}
\end{defi}

Recall (\cite{Kup}) that an {\bf $\mathcal O$-operator $T$ of a
Lie algebra $(\frak g,[-,-])$ associated to a representation
$(\rho, V)$} is a linear map $T:V\rightarrow \frak g$ satisfying
\begin{equation}
[T(u),T(v)]=T\Big(\rho\big(T(u)\big)v-\rho\big(T(v)\big)u\Big),\;\;\forall u,v\in V.
\end{equation}
Obviously, if $T$ is an $\mathcal O$-operator of an anti-pre-Lie
algebra $(A,\circ)$ associated to a representation
$(l_{\circ},r_{\circ},V)$, then $T$ is an $\mathcal O$-operator of
the sub-adjacent Lie algebra $\big(\frak g(A),[-,-]\big)$ associated to
$(l_\circ-r_\circ, V)$.

\begin{ex}
Theorem~\ref{thm:O-operator and T-equation} is rewritten as
follows. Let $(A,\circ)$ be an anti-pre-Lie algebra and $r\in
A\otimes A$ be skew-symmetric. Then $r$ is a solution of the
APL-YBE in $(A,\circ)$ if and only if $r$ is an
$\mathcal{O}$-operator of $(A,\circ)$ associated to
$(-\mathrm{ad}^{*},\mathcal{R}^{*}_{\circ},A^{*})$, or
equivalently, $r$ is an $\mathcal O$-operator of the sub-adjacent
Lie algebra $\big(\frak g(A),[-,-]\big)$ associated to $(-\mathcal
L^*_\circ, A^{*})$.
\end{ex}

\begin{thm}\label{thm:O-operator and T-equation:semi-direct product version}
    Let $(A,\circ)$ be an anti-pre-Lie algebra and $(l_{\circ},r_{\circ},V)$ be a representation of $(A,\circ)$.
    Set $\hat{A}=A\ltimes_{r^{*}_{\circ}-l^{*}_{\circ},r^{*}_{\circ}}V^{*}$.
    Let $T:V\rightarrow A$ be a linear map which is identified as an element in the vector space
    $\hat{A}\otimes\hat{A}$ (through $\mathrm{Hom}(V,A)\cong A\otimes V^{*} \subseteq  \hat{A}\otimes\hat{A}$).
     Then $r=T-\tau(T)$ is a skew-symmetric solution of the APL-YBE in the anti-pre-Lie algebra $\hat{A}$ if and only if $T$ is an $\mathcal{O}$-operator of $(A,\circ)$ associated to $(l_{\circ},r_{\circ},V)$.
\end{thm}
\begin{proof}
    Let $\lbrace v_{1},\cdots ,v_{n}\rbrace$ be a basis of $V$ and $\lbrace v^{*}_{1},\cdots ,v^{*}_{n}\rbrace$ be the dual basis. Then
    $$T=\sum_{i=1}^{n}T(v_{i})\otimes v^{*}_{i}\in T(V)\otimes V^{*}\subset \hat{A}\otimes \hat{A}.$$
    Note that
    $$l^{*}_{\circ}\big(T(v_{i})\big)v^{*}_{j}=\sum_{k=1}^{n}-\langle v^{*}_{j},l_{\circ}\big(T(v_{i})\big)v_{k}\rangle v^{*}_{k},\ \  r^{*}_{\circ}\big(T(v_{i})\big)v^{*}_{j}=\sum_{k=1}^{n}-\langle v^{*}_{j},r_{\circ}\big(T(v_{i})\big)v_{k}\rangle v^{*}_{k}.$$
   Then we have
    \begin{small}
    \begin{eqnarray*}
        r_{12}\circ r_{13}
        &=&\sum_{i,j=1}^{n}T(v_{i})\circ T(v_{j})\otimes v^{*}_{i}\otimes v^{*}_{j}-T(v_{i})\circ v^{*}_{j}\otimes v^{*}_{i}\otimes T(v_{j})-v^{*}_{i}\circ T(v_{j})\otimes T(v_{i})\otimes v^{*}_{j}\\
        &=&\sum_{i,j=1}^{n}T(v_{i})\circ T(v_{j})\otimes v^{*}_{i}\otimes v^{*}_{j}
        -(r^{*}_{\circ}-l^{*}_{\circ})\big(T(v_{i})\big)v^{*}_{j}\otimes v^{*}_{i}\otimes T(v_{j})
        -r^{*}_{\circ}\big(T(v_{j})\big)v^{*}_{i}\otimes T(v_{i})\otimes v^{*}_{j}\\
        &=&\sum_{i,j=1}^{n}T(v_{i})\circ T(v_{j})\otimes v^{*}_{i}\otimes v^{*}_{j}
        +v^{*}_{j}\otimes v^{*}_{i}\otimes T\Big((r_{\circ}-l_{\circ})\big(T(v_{i})\big)v_{j}\Big)
        +v^{*}_{i}\otimes T\Big(r_{\circ}\big(T(v_{j})\big)v_{i}\Big)\otimes v^{*}_{j}.
    \end{eqnarray*}
    \end{small}
    Similarly, we have
    \begin{small}
    \begin{eqnarray*}
        r_{12}\circ r_{23}&=&\sum_{i,j=1}^{n}-T\Big(r_{\circ}\big(T(v_{j})\big)v_{i}\Big)\otimes v^{*}_{i}\otimes v^{*}_{j}-v^{*}_{i}\otimes T(v_{i})\circ T(v_{j})\otimes v^{*}_{j}\\
        &\mbox{}&\hspace{4cm} +v^{*}_{i}\otimes v^{*}_{j}\otimes T\Big((r_{\circ}-l_{\circ})\big(T(v_{i})\big)v_{j}\Big),\\
        -[r_{13},r_{23}]&=&\sum_{i,j=1}^{n}v^{*}_{i}\otimes T\Big(l_{\circ}\big(T(v_{i})\big)v_{j}\Big)\otimes v^{*}_{j}-T\Big(l_{\circ}\big(T(v_{j})\big)v_{i}\Big)\otimes v^{*}_{j}\otimes v^{*}_{i}+v^{*}_{i}\otimes v^{*}_{j}\otimes[T(v_{i}),T(v_{j})].
    \end{eqnarray*}
\end{small}
    Hence the conclusion follows.
\end{proof}

\delete{ According to Theorems \ref{thm:O-operator and T-equation}
and \ref{thm:O-operator and T-equation:semi-direct product
version}, we have the following conclusion.

\begin{cor}\label{cor:O-operator and T-equation:equivalence}
    Let $(A,\circ)$ be an anti-pre-Lie algebra and $(l_{\circ},r_{\circ},V)$ be a representation of $(A,\circ)$. Set $\hat{A}=A\ltimes_{r^{*}_{\circ}-l^{*}_{\circ},r^{*}_{\circ}}V^{*}$.
    Then the following conditions are equivalent:
    \begin{enumerate}
        \item $T$ is an $\mathcal{O}$-operator of $(A,\circ)$ associated to $(l_{\circ},r_{\circ},V)$;
        \item $r=T-\tau(T)$ is a skew-symmetric solution of the APL-YBE in the anti-pre-Lie algebra $\hat{A}$;
        \item $r=T-\tau(T)$ is an $\mathcal{O}$-operator of $\hat{A}$ associated to
        $(-\mathrm{ad}^{*}_{\hat{A}},\mathcal{R}^{*}_{\circ_{\hat{A}}},\hat{A}^{*})$.
    \end{enumerate}
\end{cor}
}

\begin{defi}\label{defi:quasi anti-pre-Lie algebras}
    A \textbf{\papl algebra} or a {\bf pre-APL algebra} in short, is a triple $(A,\succ$,
    $\prec)$, where $A$ is a vector space and $\succ,\prec:A\otimes A\rightarrow A$ are bilinear operations, such that the following equations hold:
    \begin{small}
\begin{eqnarray}
    (y\succ x+y\prec x)\succ z-(x\succ y+x\prec y)\succ z&=&x\succ(y\succ z)-y\succ(x\succ z),\ \ \ \ \ \ \label{eq:defi:quasi anti-pre-Lie algebras1}\\
    z\prec(x\succ y+ x\prec y)&=&x\succ(z\prec y)+(x\succ z)\prec y-(z\prec x)\prec y,\ \ \ \ \ \ \label{eq:defi:quasi anti-pre-Lie algebras2}\\
    (y\succ x+y\prec x)\succ z-(x\succ y+x\prec y)\succ z
       &=&(y\succ z)\prec x-(x\succ z)\prec y-(z\prec y)\prec x+(z\prec x)\prec y,\ \ \ \ \ \ \label{eq:defi:quasi anti-pre-Lie algebras3}
\end{eqnarray}
\end{small}
    \delete{
    \begin{equation}\label{eq:defi:quasi anti-pre-Lie algebras1}
        (y\succ x+x\triangleleft y)\succ z-(x\succ y+y\triangleleft x)\succ z=x\succ(y\succ z)-y\succ(x\succ z),
    \end{equation}
    \begin{equation}\label{eq:defi:quasi anti-pre-Lie algebras2}
        (x\succ y+y\triangleleft x)\triangleleft z=x\succ(y\triangleleft z)+y\triangleleft(x\succ z)-y\triangleleft(x\triangleleft z),
    \end{equation}
    \begin{equation}\label{eq:defi:quasi anti-pre-Lie algebras3}
        (y\succ x+x\triangleleft y)\succ z-(x\succ y+y\triangleleft x)\succ z
        =x\triangleleft(y\succ z)-y\triangleleft(x\succ z)-x\triangleleft(y\triangleleft z)+y\triangleleft(x\triangleleft
        z),
    \end{equation}}
for all $x,y,z\in A$.
\end{defi}

\begin{pro}\label{defi:quasi anti-pre-Lie algebras and anti-pre-Lie algebras}
    Let $(A,\succ,\prec)$ be a pre-APL algebra. Define a bilinear operation $\circ:A\otimes A\rightarrow A$ by
    \begin{equation}\label{eq:defi:quasi anti-pre-Lie algebras and anti-pre-Lie algebras}
        x\circ y=x\succ y+ x\prec y,\;\;\forall x,y\in A.
    \end{equation}
    Then $(A,\circ)$ is an anti-pre-Lie algebra, called the \textbf{sub-adjacent anti-pre-Lie algebra} of $(A,\succ,\prec)$, and $(A,\succ,\prec)$ is
    called a \textbf{compatible pre-APL  algebra} of $(A,\circ)$. Moreover, $(\mathcal{L}_{\succ},\mathcal{R}_{\prec},A)$ is a representation of $(A,\circ)$,
    where $\mathcal{L}_{\succ},\mathcal{R}_{\prec}:A\rightarrow\mathrm{End}(A)$ are linear maps defined by
    $$\mathcal{L}_{\succ}(x)y=x\succ y,\; \mathcal{R}_{\prec}(x)y=y\prec x,\;\forall x,y\in A.$$
\end{pro}
\begin{proof}
    Let $x,y,z\in A$. Then by Eqs.~(\ref{eq:defi:quasi anti-pre-Lie algebras1}) and (\ref{eq:defi:quasi anti-pre-Lie algebras2}), we have
    \begin{eqnarray*}
    &&x\circ(y\circ z)-y\circ(x\circ z)+(x\circ y)\circ z-(y\circ x)\circ z\\
    &&=x\prec(y\succ z) +x\prec(  y\prec z) - y\succ(x\prec z)+ (  x\prec y)\prec z- (y\succ x)\prec z\\
    &&\ \ -y\prec (x\succ z) -y\prec(  x\prec z) + x\succ(  y\prec z)- (  y\prec x)\prec z+ (x\succ y)\prec z\\
    &&\ \ +x\succ(y\succ z)-y\succ(x\succ z)+(x\succ y)\succ z+( x\prec y)\succ z-(y\succ x)\succ z-( y\prec x)\succ z\\
    &&=0.
    \end{eqnarray*}
Hence  Eq.~(\ref{eq:defi:anti-pre-Lie algebras1}) holds.
Similarly, Eq.~(\ref{eq:defi:anti-pre-Lie algebras2}) holds by
Eq.~(\ref{eq:defi:quasi anti-pre-Lie algebras3}),
    and thus $(A,\circ)$ is an anti-pre-Lie algebra. Moreover, by Eqs.~(\ref{eq:defi:rep anti-pre-Lie
algebra1})-(\ref{eq:defi:rep anti-pre-Lie algebra3}),
$(\mathcal{L}_{\succ},\mathcal{R}_{\prec},A)$ is
a representation $(A,\circ)$.
\end{proof}

 Recall the definition of Zinbiel algebras.
\begin{defi}(\cite{Lod})
    A \textbf{Zinbiel algebra} is a pair $(A,\star)$, where $A$ is a vector space and $\star:A\otimes A\rightarrow A$ is a bilinear operation such that the following equation holds:
    \begin{equation}\label{eq:Zinbiel}
        x\star(y\star z)=(y\star x)\star z+(x\star y)\star z, \;\;\forall x,y,z\in A.
    \end{equation}
\end{defi}

Let $(A,\star)$ be a Zinbiel algebra. Define a new bilinear
operation $\cdot$ on $A$ by
\begin{equation}\label{eq:ZintoAss}
    x\cdot y=x\star y+y\star x,\;\;\forall x,y\in A.
\end{equation}
Then $(A,\cdot)$ is a commutative associative algebra. Moreover,
$(\mathcal{L}_{\star},A)$ is a representation of the commutative
associative algebra $(A,\cdot)$ (see Eq.~\eqref{eq:54}), where
$\mathcal{L}_{\star}:A\rightarrow\mathrm{End}(A)$ is given by
$\mathcal{L}_{\star}(x)y=x\star y$, for all $x,y\in A$.

Similar to the fact that a commutative associative algebra with a
derivation gives an anti-pre-Lie algebra in Example \ref{ex:AYBE},
there is the following construction of pre-APL algebras from
Zinbiel algebras with derivations, which can be regarded as the
``splitting viewpoint" of the former fact.

\begin{ex}\label{ex:Zinbiel derivation}
    Let $(A,\star)$ be a Zinbiel algebra with a derivation $P$. Define two bilinear operations $\succ,\prec:A\otimes A\rightarrow A$ respectively by
    \begin{equation}\label{eq:ex:Zinbiel derivation}
        x\succ y=P(x\star y)+P(x)\star y,\; x\prec y=P(y\star x)+y\star P(x),\;\;\forall x,y\in A.
    \end{equation}
Then $(A,\succ,\prec)$ is a pre-APL algebra.
Moreover, the sub-adjacent anti-pre-Lie algebra $(A,\circ)$ of
$(A,\succ,\prec)$ satisfies Eq.~(\ref{eq:diff
anti-pre-Lie}).
\end{ex}

\delete{ \textcolor{blue}{GL: The "pre pre-Lie algebras" are
L-dendriform algebras, but they correspond to pre-Lie algebras in
two different ways, called "horizontal" and "vertical" pre-Lie
algebras. I am not sure how it comes and how to do it for
anti-pre-Lie algebras.}

\textcolor{blue}{GL:An anti L-dendriform algebra $(A,\succ,\triangleleft)$ is characterized as a compatible algebra of a pre-Lie algebra $(A,\circ)$ by $x\circ y=x\succ y+x\triangleleft y$ such that $(-\mathcal{L}_{\succ},-\mathcal{L}_{\triangleleft},A)$ is a representation of $(A,\circ)$.
Explicitly, the identities of an anti L-dendriform algebra $(A,\succ,\triangleleft)$ are
\begin{equation}
    \begin{array}{ll}
    &(x\succ y)\succ z+(x\triangleleft y)\succ z+(x\succ y)\triangleleft z+(x\triangleleft y)\triangleleft z\\
    &-x\succ (y\succ z)-x\succ(y\triangleleft z)-x\triangleleft(y\succ z)-x\triangleleft(y\triangleleft z)\\
    &=(y\succ x)\succ z+(y\triangleleft x)\succ z+(y\succ x)\triangleleft z+(y\triangleleft x)\triangleleft z\\
    &-y\succ (x\succ z)-y\succ(x\triangleleft z)-y\triangleleft(x\succ z)-y\triangleleft(x\triangleleft z),\\
    \end{array}
\end{equation}
\begin{equation}
x\succ (y\succ z)+(x\succ y+x\triangleleft y)\succ z=y\succ    (x\succ z)+(y\succ x+y\triangleleft x)\succ z,
\end{equation}
\begin{equation}
x\succ(y\triangleleft z)-y\triangleleft(x\succ z)+(x\succ y)\triangleleft z+(x\triangleleft y)\triangleleft z+y\triangleleft(x\triangleleft z)=0.
\end{equation}
} }

\begin{pro}\label{thm:O-operator and quasi anti-pre-Lie algebras}
    Let $(A,\circ)$ be an anti-pre-Lie algebra and $(l_{\circ},r_{\circ},V)$ be a representation of $(A,\circ)$. Let $T:V\rightarrow A$ be an $\mathcal{O}$-operator of $(A,\circ)$ associated to
    $(l_{\circ},r_{\circ},V)$. Then there exists a  pre-APL  algebra structure $(V,\succ,\prec)$ on $V$, where
    \begin{equation}\label{eq:thm:O-operator and quasi anti-pre-Lie algebras1}
        u\succ v=l_{\circ}\big(T(u)\big)v,\;u\prec v=r_{\circ}\big(T(v)\big)u,\;\;\forall u,v\in V.
    \end{equation}
Conversely, let $(A,\succ,\prec)$ be a pre-APL
algebra and $(A,\circ)$ be the sub-adjacent anti-pre-Lie algebra.
Then the identity map $\mathrm{id}:A\rightarrow A$ is an
$\mathcal{O}$-operator of $(A,\circ)$ associated to
$(\mathcal{L}_{\succ},\mathcal{R}_{\prec},A)$.
\end{pro}
\begin{proof}
It is straightforward.
\end{proof}

\delete{

\begin{cor}\label{cor:O-operator and quasi anti-pre-Lie algebras}
    Let $(A,\circ)$ be an anti-pre-Lie algebra. Then there is a compatible \papl   algebra structure on $(A,\circ)$ if and only if there is an invertible $\mathcal{O}$-operator of $(A,\circ)$.
\end{cor}
\begin{proof}
    If $T:V\rightarrow A$ is an invertible $\mathcal{O}$-operator of $(A,\circ)$ associated to $(l_{\circ},r_{\circ},V)$, then by Theorem \ref{thm:O-operator and quasi anti-pre-Lie algebras}, there is a compatible \papl   algebra structure on $A$ given by
    $$x\succ y=T(l_{\circ}(x)T^{-1}(y)),\;x\triangleleft y=T(r_{\circ}(x)T^{-1}(y)),\;\;\forall x,y\in A.$$
    Conversely, let $(A,\succ,\triangleleft)$ be a \papl   algebra, and $(A,\circ)$ be the sub-adjacent anti-pre-Lie algebra. Then the identity map $\mathrm{id}:A\rightarrow A$ is an invertible $\mathcal{O}$-operator of $(A,\circ)$ associated to $(\mathcal{L}_{\succ},\mathcal{L}_{\triangleleft},A)$.
\end{proof}
}


\begin{pro}\label{pro:O-operator and T-equation}
    Let $(A,\succ,\prec)$ be a pre-APL   algebra and $(A,\circ)$ be the sub-adjacent anti-pre-Lie algebra of $(A,\succ,\prec)$.
    Let $\lbrace e_{1},\cdots ,e_{n}\rbrace$ be a basis of $A$ and $\lbrace e^{*}_{1},\cdots ,e^{*}_{n}\rbrace$ be the dual basis. Then
    \begin{equation}\label{eq:pro:O-operator and T-equation}
        r=\sum_{i=1}^{n}(e_{i}\otimes e^{*}_{i}-e^{*}_{i}\otimes e_{i})
    \end{equation}
    is a skew-symmetric solution of the APL-YBE in the anti-pre-Lie algebra $A\ltimes_{\mathcal{R}^{*}_{\prec}-\mathcal{L}^{*}_{\succ},\mathcal{R}^{*}_{\prec}}A^{*}$.
\end{pro}
\begin{proof}
By  Proposition~\ref{thm:O-operator and quasi anti-pre-Lie
algebras}, the identity map $\mathrm{id}$ is an
$\mathcal{O}$-operator of $(A,\circ)$ associated to
$(\mathcal{L}_{\succ},\mathcal{R}_{\prec}$, $A)$.
Note that  $\mathrm{id}=\sum\limits_{i=1}^{n}e_{i}\otimes
e^{*}_{i}$. 
Thus by  Theorem \ref{thm:O-operator and T-equation:semi-direct
product version}, the conclusion follows.
    \end{proof}

\section{Anti-pre-Lie Poisson bialgebras}\label{S4}
We give the notion of representations of transposed
 Poisson algebras. Thus anti-pre-Lie Poisson algebras are characterized
in terms of representations of the sub-adjacent transposed Poisson
algebras on the dual spaces of themselves. We also introduce the
notion of representations of anti-pre-Lie Poisson algebras as well
as the notions of matched pairs of both transposed Poisson and
anti-pre-Lie Poisson algebras and study the relationships between
them. We introduce the notion of Manin triples of transposed
Poisson algebras with respect to the invariant bilinear forms on
the commutative associative algebras and the commutative
2-cocycles on the Lie algebras, which is a combination of a double
construction of commutative Frobenius algebras and a Manin triple
of Lie algebras with respect to the commutative 2-cocycle,
characterized by certain matched pairs of anti-pre-Lie Poisson
algebras and transposed Poisson algebras. Consequently, we
introduce the notion of anti-pre-Lie Poisson bialgebras as their
equivalent structures. Finally, we study the coboundary
anti-pre-Lie Poisson bialgebras, which lead to the introduction of
the anti-pre-Lie Poisson Yang-Baxter equation (APLP-YBE). In
particular, a skew-symmetric solution of the APLP-YBE in an
anti-pre-Lie Poisson algebra gives a coboundary anti-pre-Lie
Poisson bialgebra. We also introduce the notions of
$\mathcal{O}$-operators of anti-pre-Lie Poisson algebras and
pre-(anti-pre-Lie Poisson) algebras to construct skew-symmetric
solutions of the APLP-YBE in anti-pre-Lie Poisson algebras.

\subsection{Representations of transposed Poisson algebras and anti-pre-Lie Poisson
algebras}\


Recall that a \textbf{representation} of a commutative associative
algebra $(A,\cdot)$ is a pair $(\mu,V)$, where $V$ is a vector
space and $\mu:A\rightarrow\mathrm{End}(V)$ is a linear map
satisfying
\begin{equation}\label{eq:54}
    \mu(x\cdot y)=\mu(x)\mu(y),\;\;\forall x,y\in A.
\end{equation}
For a commutative associative algebra $(A,\cdot)$, let
$\mathcal{L}_{\cdot}:A\rightarrow\mathrm{End}(A)$ be a linear map
given by $\mathcal{L}_{\cdot}(x)y=x\cdot y$, for all $x,y\in A$.
Then $(\mathcal{L}_{\cdot},A)$ is a representation of $(A,\cdot)$,
called the \textbf{adjoint representation} of $(A,\cdot)$.

In fact, $(\mu,V)$ is a representation of a commutative
associative algebra $(A,\cdot)$  if and only if the direct sum
$A\oplus V$ of vector spaces is a ({\bf
    semi-direct product}) commutative associative algebra  by defining the multiplication $\cdot_{A\oplus V}$ (often still denoted by $\cdot$) on $A\oplus
V$ by
\begin{equation}\label{eq:SDASSO}
    (x+u)\cdot_{A\oplus V}(y+v)=x\cdot y+\mu(x)v+\mu(y)u,\;\;\forall x,y\in A, u,v\in V.
\end{equation}
We denote it by $A\ltimes_{\mu}V$.

If $(\mu,V)$ is a representation of a commutative associative
algebra $(A,\cdot)$, then $(-\mu^{*},V^{*})$ is also a
representation of $(A,\cdot)$. In particular,
$(-\mathcal{L}^{*}_{\cdot},A^{*})$ is a representation of
$(A,\cdot)$.

Now we introduce the notion of  representations of  transposed
Poisson algebras.
\begin{defi}
 A \textbf{representation} of a transposed Poisson algebra $(A,\cdot,[-,-])$ is a triple $(\mu,\rho,V)$, such that the pair $(\mu,V)$ is a representation of the commutative associative algebra $(A,\cdot)$, the pair $(\rho,V)$ is a
 representation of the Lie algebra $(A,[-,-])$, and the following conditions hold:
\begin{eqnarray}
    &&2\mu(x)\rho(y)=\rho(x\cdot y)+\rho(y)\mu(x),\label{eq:defi:TPA REP2}\\
    &&2\mu([x,y])=\rho(x)\mu(y)-\rho(y)\mu(x),\label{eq:defi:TPA REP1}
\end{eqnarray}
for all $x,y\in A$.
\end{defi}

\begin{ex}
    Let $(A,\cdot,[-,-])$ be a transposed Poisson algebra.  Then $(\mathcal{L}_{\cdot},\mathrm{ad},A)$ is a representation of $(A,\cdot,[-,-])$, called the \textbf{adjoint representation} of $(A,\cdot,[-,-])$.
\end{ex}

\begin{pro}\label{pro:repandsemidirectproduct2}
    Let $(A,\cdot,[-,-])$ be a transposed Poisson algebra and $V$ be a vector space. Let $\mu,\rho:A\rightarrow \mathrm{End}(V)$ be linear maps. Then $(\mu,\rho,V)$ is a representation of $(A,\cdot,[-,-])$ if and only if the direct sum $A\oplus V$ of vector spaces is a (\textbf{semi-direct product}) transposed Poisson algebra by defining the
    bilinear operations on $A\oplus V$ by Eqs.~(\ref{eq:SDASSO}) and (\ref{eq:SDLie}) respectively.
    We denote it by $A\ltimes_{\mu,\rho}V$.
\end{pro}
\begin{proof}
    It is a special case of the matched pair of transposed Poisson algebras in Theorem \ref{thm:mp TPA} when $B=V$ is equipped with the zero multiplications.
\end{proof}

\begin{pro}\label{pro:dual rep TPA}
Let $(A,\cdot,[-,-])$ be a transposed Poisson algebra and
$(\mu,\rho,V)$ be a representation of $(A,\cdot,[-,-])$. Then
$(-\mu^{*},\rho^{*},V^{*})$ is a representation of
$(A,\cdot,[-,-])$ if and only if
\begin{equation}\label{eq:d0}
    \mu([x,y])=0,\ \rho(x\cdot y)=\mu(x)\rho(y),\;\;\forall x,y\in A.
\end{equation}
In particular, $(-\mathcal{L}^{*}_{\cdot},\mathrm{ad}^{*},A^{*})$
is a representation of $(A,\cdot,[-,-])$ if and only if
Eq.~(\ref{eq:coherent TPA}) holds, that is, the following equation
holds:
\begin{equation}\label{eq:coherent TPA1}
        [x,y]\cdot z=[x\cdot y,z]=0,\;\;\forall x,y,z\in A.
    \end{equation}
\end{pro}

\begin{proof}
    Let $x,y\in A, u^{*}\in V^{*}, v\in V$. Then we have
        \begin{equation}\label{eq:d1}
        \langle \big(\rho^{*}(x\cdot y)-\rho^{*}(y)\mu^{*}(x)+2\mu^{*}(x)\rho^{*}(y)\big)u^{*}, v\rangle=\langle u^{*}, \big(-\rho(x\cdot y)-\mu(x)\rho(y)+2\rho(y)\mu(x)\big)v\rangle,
        \end{equation}
    \begin{equation}\label{eq:d2}
            \langle \big(-\rho^{*}(x)\mu^{*}(y)+\rho^{*}(y)\mu^{*}(x)+2\mu^{*}([x,y])\big)u^{*},v\rangle=\langle u^{*},\big(-\mu(y)\rho(x)+\mu(x)\rho(y)-2\mu([x,y])\big)v\rangle.
    \end{equation}
Thus $(-\mu^{*},\rho^{*},V^{*})$ is a representation if and only
if the following equations hold:
\begin{equation}\label{eq:d3}
\rho(x\cdot y)+\mu(x)\rho(y)-2\rho(y)\mu(x)=0,
\end{equation}
\begin{equation}\label{eq:d4}
\mu(y)\rho(x)-\mu(x)\rho(y)+2\mu([x,y])=0.
\end{equation}
By Eqs.~(\ref{eq:defi:TPA REP2}) and (\ref{eq:defi:TPA REP1}),
Eqs.~(\ref{eq:d3}) and (\ref{eq:d4}) hold if and only if
Eq.~(\ref{eq:d0}) holds. The particular case follows immediately.
\end{proof}

\begin{rmk} In \cite{Bai2020},
a transposed Poisson algebra $(A,\cdot,[-,-])$ is also a Poisson
algebra if and only if Eq.~(\ref{eq:coherent TPA}) holds. On the
other hand, by \cite{NB}, if $(\mu,\rho,V)$ is a representation of
a Poisson algebra $(A,\cdot,[-,-])$, then
$(-\mu^{*},\rho^{*},V^{*})$ is automatically a representation of
$(A,\cdot,[-,-])$. Hence here there is an obvious difference
between Poisson algebras and transposed Poisson algebras, that is,
for the latter, there is not a ``natural" dual representation for
any representation of a transposed Poisson algebra.
\end{rmk}

Next we interpret the relationship between transposed Poisson
algebras and anti-pre-Lie Poisson algebras in terms of
representations of transposed Poisson algebras.

\begin{pro}{\rm (\cite{LB2022})}
Let $(A,\cdot,\circ)$ be an anti-pre-Lie Poisson algebra and
$(A,[-,-])$ be the sub-adjacent Lie algebra of $(A,\circ)$. Then
$(A,\cdot,[-,-])$ is a transposed Poisson algebra.
\end{pro}

Hence we give the following notion.

\begin{defi} Let $(A,\cdot,\circ)$ be an anti-pre-Lie Poisson algebra and
$(A,[-,-])$ be the sub-adjacent Lie algebra of $(A,\circ)$. We
call $(A,\cdot,[-,-])$ the \textbf{sub-adjacent transposed Poisson
algebra} of $(A,\cdot,\circ)$, and $(A,\cdot,\circ)$  a
\textbf{compatible anti-pre-Lie Poisson algebra} of
$(A,\cdot,[-,-])$.
\end{defi}

On the other hand, anti-pre-Lie Poisson algebras are characterized
in terms of representations of the sub-adjacent transposed Poisson
algebras on the dual spaces of themselves.

\begin{pro}\label{pro:anti-pre-Lie Poisson}
Let $(A,\cdot,[-,-])$ be a transposed Poisson algebra. Suppose
that $(A,[-,-])$ is the sub-adjacent Lie algebra of an
anti-pre-Lie algebra $(A,\circ)$. Then
$(-\mathcal{L}^{*}_{\cdot},-\mathcal{L}^{*}_{\circ},A^{*})$ is a
representation of $(A,\cdot,[-,-])$ if and only if
$(A,\cdot,\circ)$ is an anti-pre-Lie Poisson algebra.
\end{pro}
\begin{proof}
    It is obvious that $(-\mathcal{L}^{*}_{\cdot}, A^{*})$ is a representation of $(A,\cdot)$, and $(-\mathcal{L}^{*}_{\circ}, A^{*})$ is a representation of $(A,\circ)$. Moreover, for all $x,y,z\in A, a^{*}\in A^{*}$, we have
    \begin{equation*}
        \langle \big(2\mathcal{L}^{*}_{\cdot}(x)\mathcal{L}^{*}_{\circ}(y)+\mathcal{L}^{*}_{\circ}(x\cdot y)-\mathcal{L}^{*}_{\circ}(y)\mathcal{L}^{*}_{\cdot}(x)\big)a^{*},z\rangle=\langle a^{*}, 2y\circ(x\cdot z)-(x\cdot y)\circ z-x\cdot(y\circ z)\rangle,
    \end{equation*}
    \begin{equation*}
        \langle \big(2\mathcal{L}^{*}_{\cdot}([x,y])-\mathcal{L}^{*}_{\circ}(x)\mathcal{L}^{*}_{\cdot}(y)+\mathcal{L}^{*}_{\circ}(y)\mathcal{L}^{*}_{\cdot}(x)\big)a^{*},z\rangle=\langle a^{*}, -2[x,y]\cdot z-y\cdot (x\circ z)+x\cdot(y\circ z)\rangle.
    \end{equation*}
Hence the conclusion follows.
\end{proof}

\begin{cor}
Let $(A,\cdot,\circ)$ be an anti-pre-Lie Poisson algebra and
$(A,\cdot,[-,-])$ be the sub-adjacent transposed Poisson algebra.
Then $(\mathcal{L}_{\cdot},-\mathcal{L}_{\circ},A)$ is a
representation of $(A,\cdot,[-,-])$ if and only if the following
equation holds:
\begin{equation}\label{eq:cor}
[x,y]\cdot z=0, (x\cdot y)\circ z=y\circ(x\cdot z),\;\;\forall x,y,z\in A.
\end{equation}
\end{cor}
\begin{proof}
By Proposition \ref{pro:anti-pre-Lie Poisson}, $(-\mathcal{L}^{*}_{\cdot},-\mathcal{L}^{*}_{\circ},A^{*})$ is a representation of $(A,\cdot,[-,-])$. Then by Proposition \ref{pro:dual rep TPA}, $\big(-(-\mathcal{L}^{*}_{\cdot})^{*},(-\mathcal{L}^{*}_{\circ})^{*},(A^{*})^{*}\big)=(\mathcal{L}_{\cdot},-\mathcal{L}_{\circ}, A)$ is a representation of $(A,\cdot,[-,-])$ if and only if
\begin{equation}
    -\mathcal{L}^{*}_{\cdot}([x,y])=0, -\mathcal{L}^{*}_{\circ}(x\cdot y)=\mathcal{L}^{*}_{\cdot}(x)\mathcal{L}^{*}_{\circ}(y),\;\;\forall x,y\in A,
\end{equation}
or equivalently, Eq.~(\ref{eq:cor}) holds.
\end{proof}

\begin{rmk} Unlike anti-pre-Lie algebras (\cite{LB2022}) and pre-Lie algebras
(\cite{Bur}) being characterized in terms of representations of
the sub-adjacent Lie algebras on both the spaces of themselves and
equivalently, the dual spaces, only the representations of the
sub-adjacent transposed Poisson algebras of anti-pre-Lie Poisson
algebras on the dual spaces of themselves are involved in
characterizing anti-pre-Lie Poisson algebras.
\end{rmk}

\begin{defi}\label{defi:anti-pre-Lie Poisson rep}
    Let $(A,\cdot,\circ)$ be an anti-pre-Lie Poisson algebra. A \textbf{representation} of $(A,\cdot,\circ)$ is a
    quadruple $(\mu,l_{\circ}, r_{\circ},V)$, such that the pair $(\mu,V)$ is
a representation of $(A,\cdot)$, the triple $(l_{\circ},
r_{\circ},V)$ is a representation of $(A,\circ)$, and the
following conditions hold:
\begin{eqnarray}
 2\mu(y)l_{\circ}(x)- 2\mu(y)r_{\circ}(x)&=&\mu(x\circ y)-\mu(x)r_{\circ}(y),   \label{eq:defi:anti-pre-Lie Poisson rep1}\\
 2\mu(x\circ y)- 2\mu(y\circ x)&=&\mu(y)l_{\circ}(x)-\mu(x)l_{\circ}(y), \label{eq:defi:anti-pre-Lie Poisson rep2}\\
  2r_{\circ}(x\cdot y)&=&r_{\circ}(y)\mu(x)+\mu(x)r_{\circ}(y), \label{eq:defi:anti-pre-Lie Poisson rep3}\\
   2l_{\circ}(x)\mu(y)&=&l_{\circ}(x\cdot y)+\mu(y)l_{\circ}(x), \label{eq:defi:anti-pre-Lie Poisson rep4}\\
   2l_{\circ}(x)\mu(y)&=&r_{\circ}(y)\mu(x)+\mu(x\circ y), \label{eq:defi:anti-pre-Lie Poisson rep5}
\end{eqnarray}
\delete{
    \begin{equation}\label{eq:defi:anti-pre-Lie Poisson rep1}
        2\mu(y)l_{\circ}(x)- 2\mu(y)r_{\circ}(x)=\mu(x\circ y)-\mu(x)r_{\circ}(y),
    \end{equation}
    \begin{equation}\label{eq:defi:anti-pre-Lie Poisson rep2}
        2\mu(x\circ y)- 2\mu(y\circ x)=\mu(y)l_{\circ}(x)-\mu(x)l_{\circ}(y),
    \end{equation}
    \begin{equation}\label{eq:defi:anti-pre-Lie Poisson rep3}
        2r_{\circ}(x\cdot y)=r_{\circ}(y)\mu(x)+\mu(x)r_{\circ}(y),
    \end{equation}
    \begin{equation}\label{eq:defi:anti-pre-Lie Poisson rep4}
        2l_{\circ}(x)\mu(y)=l_{\circ}(x\cdot y)+\mu(y)l_{\circ}(x),
    \end{equation}
    \begin{equation}\label{eq:defi:anti-pre-Lie Poisson rep5}
        2l_{\circ}(x)\mu(y)=r_{\circ}(y)\mu(x)+\mu(x\circ y),
    \end{equation}}
for all $x,y\in A$.
\end{defi}

\begin{ex}
    Let $(A,\cdot,\circ)$ be an anti-pre-Lie Poisson algebra. Then $(\mathcal{L}_{\cdot},\mathcal{L}_{\circ},\mathcal{R}_{\circ},A)$ is a representation of $(A,\cdot,\circ)$, called the \textbf{adjoint representation} of $(A,\cdot,\circ)$.
\end{ex}

\begin{pro}\label{pro:repandsemidirectproduct3}
    Let $(A,\cdot,\circ)$ be  an anti-pre-Lie Poisson algebra and
$(A,\cdot,[-,-])$ be the sub-adjacent transposed Poisson algebra.
Let $V$ be a vector space and
$\mu,l_{\circ},r_{\circ}:A\rightarrow \mathrm{End}(V)$ be linear
maps.
    \begin{enumerate}
        \item \label{it:hh1} $(\mu,l_{\circ},r_{\circ},V)$ is a representation of $(A,\cdot,\circ)$ if and only if the direct sum $A\oplus V$ of vector spaces is an (\textbf{semi-direct product})
        anti-pre-Lie Poisson algebra by defining the  bilinear operations on $A\oplus V$ by Eqs.~(\ref{eq:SDASSO}) and (\ref{eq:pro:repandsemidirectproduct1}) respectively.
        We denote it by $A\ltimes_{\mu,l_{\circ},r_{\circ}}V$.
        \item\label{rep32} If $(\mu,l_{\circ},r_{\circ},V)$ is a representation of $(A,\cdot,\circ)$, then $(\mu,l_{\circ}-r_{\circ},V)$ is a representation of $(A,\cdot,[-,-])$.
    \end{enumerate}
\end{pro}
\begin{proof}
(\ref{it:hh1}). It is a special case of the matched pair of
anti-pre-Lie Poisson algebras in Theorem
        \ref{thm:mp AP}
         when $B=V$ is equipped with the zero multiplications.

(\ref{rep32}) It is a special case of Proposition~
\ref{pro:TPA-APL} when $B=V$ is equipped
with the zero multiplications.
\end{proof}

\begin{pro}\label{pro:rep of anti-pre-Lie Poisson}
    Let $(\mu,l_{\circ}, r_{\circ},V)$ be a representation of an anti-pre-Lie Poisson algebra $(A,\cdot,\circ)$.
Then $(-\mu^{*}, r^{*}_{\circ}-l^{*}_{\circ},r^{*}_{\circ},V^{*})$
is a representation of $(A,\cdot,\circ)$. In particular,
$(-\mathcal{L}^{*}_{\cdot},-\mathrm{ad}^{*},\mathcal{R}^{*}_{\circ},A^{*})$
is a representation of $(A,\cdot,\circ)$.
\end{pro}
\begin{proof}
    For all $x,y\in A, u^{*}\in V^{*}, v\in V$, we have
    \begin{eqnarray*}
        &&\langle \Big(-2\mu^{*}(y)\big(r^{*}_{\circ}(x)-l^{*}_{\circ}(x)\big)+2\mu^{*}(y)r^{*}_{\circ}(x)+\mu^{*}(x\circ y)-\mu^{*}(x)r^{*}_{\circ}(y)\Big)u^{*},v\rangle\\
        &&=\langle u^{*}, \big(2l_{\circ}(x)\mu(y)-\mu(x\circ y)-r_{\circ}(y)\mu(x)\big)v\rangle
        \overset{(\ref{eq:defi:anti-pre-Lie Poisson rep5})}{=}0.
    \end{eqnarray*}
Thus Eq.~(\ref{eq:defi:anti-pre-Lie Poisson rep1})  holds for the
quadruple $(-\mu^{*},
r^{*}_{\circ}-l^{*}_{\circ},r^{*}_{\circ},V^{*})$. Similarly
Eqs.~(\ref{eq:defi:anti-pre-Lie Poisson
rep2})-(\ref{eq:defi:anti-pre-Lie Poisson rep5})   hold for
$(-\mu^{*}, r^{*}_{\circ}-l^{*}_{\circ},r^{*}_{\circ},V^{*})$.
Thus $(-\mu^{*}, r^{*}_{\circ}-l^{*}_{\circ},r^{*}_{\circ},V^{*})$
is a representation of $(A,\cdot,\circ)$.
    \end{proof}

\begin{cor}\label{cor:dual rep TPA}
Let $(A,\cdot,\circ)$ be an anti-pre-Lie Poisson algebra and
$(A,\cdot,[-,-])$ be the sub-adjacent transposed Poisson algebra.
If $(\mu,l_{\circ},r_{ \circ},V)$ is a representation of
$(A,\cdot,\circ)$, then $(-\mu^{*}$, $-l^{*}_{\circ}$, $V^{*})$ is
a representation of $(A,\cdot,[-,-])$.
\end{cor}
\begin{proof}
    It follows from Propositions \ref{pro:repandsemidirectproduct3}  (\ref{rep32})  and  \ref{pro:rep of anti-pre-Lie Poisson}.
    \end{proof}

\begin{rmk}
As a direct consequence of Corollary \ref{cor:dual rep TPA}, for
an anti-pre-Lie Poisson algebra $(A,\cdot,\circ)$, the triple
$(-\mathcal{L}^{*}_{\cdot},-\mathcal{L}^{*}_{\circ},A^{*})$ is a
representation of the sub-adjacent transposed Poisson algebra
$(A$, $\cdot$, $[-,-])$, which  recovers the ``if" part of the
result in Proposition \ref{pro:anti-pre-Lie Poisson}.

\end{rmk}

\subsection{Matched pairs of transposed Poisson algebras and anti-pre-Lie Poisson algebras}\

Recall matched pairs of commutative associative algebras
(\cite{Bai2010}).
 Let $(A,\cdot_{A})$ and $(B,\cdot_{B})$ be two
commutative associative algebras, and $(\mu_{A},B)$ and
$(\mu_{B},A)$ be representations of $(A,\cdot_{A})$ and
$(B,\cdot_{B})$ respectively. If the following equations are
satisfied:
\begin{eqnarray}
    &&\mu_{A}(x)(a\cdot_{B} b)=\big(\mu_{A}(x)a\big)\cdot_{B} b+\mu_{A}\big(\mu_{B}(a)x\big)b,\\
    &&\mu_{B}(a)(x\cdot_{A} y)=\big(\mu_{B}(a)x\big)\cdot_{A} y+\mu_{B}\big(\mu_{A}(x)a\big)y,
\end{eqnarray}
for all $x,y\in A,a,b\in B$, then $(A,B,\mu_{A},\mu_{B})$ is
called a \textbf{matched
    pair of commutative associative algebras}.
In fact, for commutative associative  algebras $(A,\cdot_{A})$ ,
$(B,\cdot_{B})$ and linear maps
$\mu_{A}:A\rightarrow\mathrm{End}(B)$,
$\mu_{B}:B\rightarrow\mathrm{End}(A)$, there is a commutative
associative algebra structure on the vector space $A\oplus B$
given by
\begin{equation}\label{eq:Asso} (x+a)\cdot (y+b)=x\cdot_{A}
    y+\mu_{B}(a)y+\mu_{B}(b)x+a\cdot_{B}
    b+\mu_{A}(x)b+\mu_{A}(y)a,\;\;\forall x,y\in A, a,b\in B,
\end{equation}
if and only if $(A,B,\mu_{A},\mu_{B})$ is a matched pair of
commutative associative algebras. In this case, we denote the
commutative associative algebra structure on $A\oplus B$ by
$A\bowtie^{\mu_{A}}_{\mu_{B}}B$. Conversely, every commutative
associative algebra which is the direct sum of the underlying
vector spaces of two subalgebras can be obtained from a matched
pair of commutative associative algebras by this construction.

\begin{defi}\label{defi:MP TPA}
    Let $(A,\cdot_{A},[-,-]_{A})$ and $(B,\cdot_{B},[-,-]_{B})$ be two transposed Poisson algebras. Let $\mu_{A},\rho_{A}:A\rightarrow\mathrm{End}(B)$ and $\mu_{B},\rho_{B}:B\rightarrow\mathrm{End}(A)$
    be linear maps. Suppose that the following conditions are satisfied:
    \begin{enumerate}
        \item $(A,B,\mu_{A},\mu_{B})$ is a matched pair of commutative associative algebras and $(A,B,\rho_{A},\rho_{B})$ is a matched pair of Lie algebras.
        \item 
        $(\mu_{A},\rho_{A},B)$ is a representation of $(A,\cdot_{A},[-,-]_{A})$
        and $(\mu_{B},\rho_{B},A)$ is a representation of $(B$, $\cdot_{B}$, $[-,-]_{B})$.
        \item For all $x,y\in A, a,b\in B$, the following equations hold:
        \begin{eqnarray}
            &&2\mu_{B}\big(\rho_{A}(y)a\big)x-2x\cdot_{A}\rho_{B}(a)y=-\rho_{B}(a)(x\cdot_{A} y)-\rho_{B}\big(\mu_{A}(x)a\big)y+[y,\mu_{B}(a)x]_{A},\label{eq:defi:MP TPA a}\\
            && 2\mu_{B}(a)([x,y]_{A})=[\mu_{B}(a)x,y]_{A}+\rho_{B}\big(\mu_{A}(x)a\big)y+[x,\mu_{B}(a)y]_{A}-\rho_{B}\big(\mu_{A}(y)a\big)x,\label{eq:defi:MP TPA b}\\
            &&2\mu_{A}\big(\rho_{B}(b)x\big)a-2a\cdot_{B}\rho_{A}(x)b=-\rho_{A}(x)(a\cdot_{B} b)-\rho_{A}\big(\mu_{B}(a)x\big)b+[b,\mu_{A}(x)a]_{B},\label{eq:defi:MP TPA c}\\
            &&2\mu_{A}(x)([a,b]_{B})=[\mu_{A}(x)a,b]_{B}+\rho_{A}\big(\mu_{B}(a)x\big)b+[a,\mu_{A}(y)b]_{B}-\rho_{A}\big(\mu_{B}(b)x\big)a.\label{eq:defi:MP TPA d}
        \end{eqnarray}
    \end{enumerate}
    Such a structure is called a  \textbf{matched pair of transposed Poisson algebras} $(A,\cdot_{A},[-,-]_{A})$ and $(B,\cdot_{B},[-,-]_{B})$. We denote it by $(A,B,\mu_{A},\rho_{A},\mu_{B},\rho_{B})$.
\end{defi}

\begin{thm}\label{thm:mp TPA}
    Let $(A,\cdot_{A},[-,-]_{A})$ and $(B,\cdot_{B},[-,-]_{B})$ be two transposed Poisson algebras. Suppose that $\mu_{A},\rho_{A}:A\rightarrow\mathrm{End}(B)$ and $\mu_{B},\rho_{B}:B\rightarrow\mathrm{End}(A)$ are linear maps. Define two bilinear
    operations $\cdot,[-,-]$ on $A\oplus B$ by Eqs.~ (\ref{eq:Asso}) and (\ref{eq:Lie}) respectively. Then
    $(A\oplus B,\cdot,[-,-])$ is a transposed Poisson algebra if and only if $(A,B,\mu_{A},\rho_{A},\mu_{B},\rho_{B})$ is a matched pair of transposed Poisson algebras.
    In this case, we denote this transposed Poisson algebra structure on $A\oplus B$ by $A\bowtie^{\mu_{A},\rho_{A}}_{\mu_{B},\rho_{B}}B$.
  Conversely, every transposed Poisson  algebra which is the
    direct sum of the underlying vector spaces of two subalgebras can
    be obtained from a matched pair of transposed Poisson
    algebras by this construction.
\end{thm}
\begin{proof}
\delete{    As aforementioned, $(A\oplus B,\cdot)$ is a
commutative associative algebra if and only if
$(A,B,\mu_{A},\mu_{B})$ is a matched pair of commutative
associative algebras, and $(A\oplus B,[-,-])$ is a Lie algebra if
and only if $(A,B,\rho_{A},\rho_{B})$ is a matched pair of Lie
algebras.
    Let $x,y,z\in A, a,b,c\in B$. The transposed Leibniz rule on $A\oplus B$, that is,
    \begin{equation}
        2(z+c)\cdot[x+a,y+b]=[(z+c)\cdot(x+a),y+b]+[x+a,(z+c)\cdot(y+b)]
    \end{equation}
    holds if and only if the following equations hold:
    \begin{eqnarray}
        &&2x\cdot[y,a]=[x\cdot_{A}y,a]+[y,x\cdot a],\label{eq:Pf mp tpa a}\\
        &&2a\cdot[x,y]_{A}=[a\cdot x,y]+[x,a\cdot y],\label{eq:Pf mp tpa b}\\
        &&2a\cdot[b,x]=[a\cdot_{B}b,x]+[b,a\cdot x],\label{eq:Pf mp tpa c}\\
        &&2x\cdot[a,b]_{B}=[x\cdot a,b]+[a,x\cdot b].\label{eq:Pf mp tpa d}
    \end{eqnarray}
    By Eqs.~ (\ref{eq:Asso}) and (\ref{eq:Lie}), we have
    \begin{eqnarray*}
        &&2x\cdot[y,a]=2x\cdot(\rho_{A}(y)a-\rho_{B}(a)y)=2\mu_{A}(x)\rho_{A}(y)a+2\mu_{B}(\rho_{A}(y)a)x-2x\cdot_{A}\rho_{B}(a)y,\\
        &&[x\cdot_{A}y,a]=\rho_{A}(x\cdot_{A}y)a-\rho_{B}(a)(x\cdot_{A}y),\\
        &&[y,x\cdot a]=[y,\mu_{A}(x)a+\mu_{B}(a)x]=\rho_{A}(y)\mu_{A}(x)a-\rho_{B}(\mu_{A}(x)a)y+[y,\mu_{B}(a)x]_{A}.
    \end{eqnarray*}
    Thus Eq.~(\ref{eq:Pf mp tpa a}) holds if and only if Eqs.~(\ref{eq:rep TPA A1}) and (\ref{eq:defi:MP TPA a}) hold, and similarly Eq.~(\ref{eq:Pf mp tpa b}) holds if and only if Eqs.~(\ref{eq:rep TPA A2}) and (\ref{eq:defi:MP TPA b}) hold.
    By symmetry, Eq.~(\ref{eq:Pf mp tpa c}) holds if and only if Eqs.~ (\ref{eq:defi:MP TPA c}) and (\ref{eq:TPA b1}) hold, and
    Eq.~(\ref{eq:Pf mp tpa d}) holds if and only if Eqs.~ (\ref{eq:defi:MP TPA d}) and (\ref{eq:TPA b2}) hold.
    Hence the conclusion follows.}
It is straightforward.
\end{proof}

\delete{
    Noticing for representations $(\mu_{A},B)$ of $(A,\cdot_{A})$ and $(\rho_{A},B)$ of $(A,[-,-]_{A})$, Eqs.~(\ref{eq:defi:TPA REP2}) and (\ref{eq:defi:TPA REP1}) makes $(\mu_{A},\rho_{A},B)$ a representation of the transposed Poisson algebra $(A,\cdot_{A},[-,-]_{A})$,
    and similarly
    for representations $(\mu_{B},A)$ of $(B,\cdot_{B})$ and $(\rho_{B},A)$ of $(B,[-,-]_{B})$, Eqs.~(\ref{eq:TPA b1}) and (\ref{eq:TPA b2}) makes $(\mu_{B},\rho_{B},A)$ a representation of the transposed Poisson algebra $(B,\cdot_{B},[-,-]_{B})$,
    we finish the proof.

}

\begin{defi}
Let $(A,\cdot_{A},\circ_{A})$ and $(B,\cdot_{B},\circ_{B})$ be two
anti-pre-Lie Poisson algebras. Let $\mu_{A},l_{\circ_{A}},$
$r_{\circ_{A}}:A\rightarrow\mathrm{End}(B)$ and
$\mu_{B},l_{\circ_{B}},r_{\circ_{B}}:B\rightarrow\mathrm{End}(A)$
be linear maps. Suppose that the following conditions are
satisfied:
\begin{enumerate}
    \item $(A,B,\mu_{A},\mu_{B})$ is a matched pair of commutative associative algebras and $(A,B,l_{\circ_{A}},r_{\circ_{A}},l_{\circ_{B}},$
    $r_{\circ_{B}})$ is a matched pair of anti-pre-Lie algebras.
    \item $(\mu_{A},l_{\circ_{A}},r_{\circ_{A}},B)$ is a representation of $(A,\cdot_{A},\circ_{A})$ and $(\mu_{B},l_{\circ_{B}},r_{\circ_{B}},A)$ is a representation of $(B,\cdot_{B},\circ_{B})$.
    \item For all $x,y\in A, a,b\in B$, the following equations hold:
    \begin{small}
    \begin{eqnarray*}
    &&2\mu_{B}(a)(x\circ_{A}y)-2\mu_{B}(a)(y\circ_{A}x)=\mu_{B}\big(l_{\circ_{A}}(x)a\big)y+y\cdot_{A}r_{\circ_{B}}(a)x-\mu_{B}\big(l_{\circ_{A}}(y)a\big)x-x\cdot_{A}r_{\circ_{B}}(a)y,\\
    &&  2\mu_{B}\big((l_{\circ_{A}}-r_{\circ_{A}})(x)a\big)y-2(l_{\circ_{B}}-r_{\circ_{B}})(a)x\cdot_{A}y =\mu_{B}(a)(x\circ_{A}y)-x\cdot_{A}l_{\circ_{B}}(a)y-\mu_{B}\big(r_{\circ_{A}}(y)a\big)x,\\
    &&2r_{\circ_{B}}\big(\mu_{A}(y)a\big)x+2x\circ_{A}\mu_{B}(a)y=\mu_{B}(a)x\circ_{A}y+l_{\circ_{B}}\big(\mu_{A}(x)a\big)y+\mu_{B}(a)(x\circ_{A}y),\\
    &&2r_{\circ_{B}}\big(\mu_{A}(y)a\big)x+2x\circ_{A}\mu_{B}(a)y=r_{\circ_{B}}(a)(x\cdot_{A}y)+\mu\big(l_{\circ_{A}}(x)a\big)y+y\cdot_{A}r_{\circ_{B}}(a)x,\\
    &&2l_{\circ_{B}}(a)(x\cdot_{A}y)=l_{\circ_{B}}\big(\mu_{A}(y)a\big)x+\mu_{B}(a)y\circ_{A}x+y\cdot_{A}l_{\circ_{B}}(a)x+\mu_{B}\big(r_{\circ_{A}}(x)a\big)y,\\
    &&2\mu_{A}(x)(a\circ_{B}b)-2\mu_{A}(x)(b\circ_{B}a)=\mu_{A}\big(l_{\circ_{B}}(a)x\big)b+b\cdot_{B}r_{\circ_{A}}(x)a-\mu_{A}\big(l_{\circ_{B}}(b)x\big)a-a\cdot_{B}r_{\circ_{A}}(x)b,\\
    &&  2\mu_{A}\big((l_{\circ_{B}}-r_{\circ_{B}})(a)x\big)b-2(l_{\circ_{A}}-r_{\circ_{A}})(x)a\cdot_{B}b =\mu_{A}(x)(a\circ_{B}b)-a\cdot_{B}l_{\circ_{A}}(x)b-\mu_{A}\big(r_{\circ_{B}}(b)x\big)a,\\
     &&2r_{\circ_{A}}\big(\mu_{B}(b)x\big)a+2a\circ_{B}\mu_{A}(x)b=\mu_{A}(x)a\circ_{B}b+l_{\circ_{A}}\big(\mu_{B}(a)x\big)b+\mu_{A}(x)(a\circ_{B}b),\\
     &&2r_{\circ_{A}}\big(\mu_{B}(b)x\big)a+2a\circ_{B}\mu_{A}(x)b=r_{\circ_{A}}(x)(a\cdot_{B}b)+\mu\big(l_{\circ_{B}}(a)x\big)b+b\cdot_{B}r_{\circ_{A}}(x)a,\\
     &&2l_{\circ_{A}}(x)(a\cdot_{B}b)=l_{\circ_{A}}\big(\mu_{B}(b)x\big)a+\mu_{A}(x)b\circ_{B}a+b\cdot_{B}l_{\circ_{A}}(x)a+\mu_{A}\big(r_{\circ_{B}}(a)x\big)b.
\end{eqnarray*}
\end{small}
\end{enumerate}
Such a structure is called a \textbf{matched pair of anti-pre-Lie
Poisson algebras} $(A,\cdot_{A},\circ_{A})$ and
$(B,\cdot_{B},\circ_{B})$. We denote it by
$(A,B,\mu_{A},l_{\circ_{A}},r_{\circ_{A}},\mu_{B},l_{\circ_{B}},r_{\circ_{B}})$.
\end{defi}

\begin{thm}\label{thm:mp AP}
    Let $(A,\cdot_{A},\circ_{A})$ and $(B,\cdot_{B},\circ_{B})$ be two anti-pre-Lie Poisson algebras. Suppose that $\mu_{A},l_{\circ_{A}},$
    $r_{\circ_{A}}:A\rightarrow\mathrm{End}(B)$ and  $\mu_{B},l_{\circ_{B}},r_{\circ_{B}}:B\rightarrow\mathrm{End}(A)$ are linear maps. Define two bilinear
    operations $\cdot,\circ$ on $A\oplus B$ by Eqs.~(\ref{eq:Asso}) and (\ref{thm:matched pairs of anti-pre-Lie algebras1}) respectively.
    Then $(A\oplus B,\cdot,\circ)$ is an anti-pre-Lie Poisson algebra if and only if $(A,B,\mu_{A},l_{\circ_{A}},r_{\circ_{A}},\mu_{B},l_{\circ_{B}},r_{\circ_{B}})$
    is a matched pair of anti-pre-Lie Poisson algebras. In this case, we denote this anti-pre-Lie Poisson algebra structure on
    $A\oplus B$ by $A\bowtie^{\mu_{A},l_{\circ_{A}},r_{\circ_{A}}}_{\mu_{B},l_{\circ_{B}},r_{\circ_{B}}}B$. Conversely,
    every anti-pre-Lie Poisson algebra which is the direct sum of the underlying vector spaces of two subalgebras can be obtained from a matched pair of anti-pre-Lie Poisson algebras by this construction.
    \end{thm}
\begin{proof}
It is straightforward.
\end{proof}

\begin{pro}\label{pro:TPA-APL}
Let $(A,\cdot_{A},\circ_{A})$ and $(B,\cdot_{B},\circ_{B})$ be two
anti-pre-Lie Poisson algebras and their sub-adjacent transposed
Poisson algebras be $(A,\cdot_{A},[-,-]_{A})$ and
$(B,\cdot_{B},[-,-]_{B})$ respectively. If
$(A,B,\mu_{A},l_{\circ_{A}},r_{\circ_{A}},\mu_{B},l_{\circ_{B}},r_{\circ_{B}})$
is a matched pair of anti-pre-Lie Poisson algebras, then
$(A,B,\mu_{A},$
$l_{\circ_{A}}-r_{\circ_{A}},\mu_{B},l_{\circ_{B}}-r_{\circ_{B}})$
is a matched pair of transposed Poisson algebras.
\end{pro}
\begin{proof}
    It is similar to the proof of Proposition \ref{pro:from matched pairs of anti-pre-Lie algebras to matched pairs of Lie algebras}.
    \end{proof}

\subsection{Manin triples of transposed Poisson algebras and anti-pre-Lie Poisson bialgebras}\

As mentioned in Introduction, transposed Poisson algebras are
``inconsistent" with the nondegenerate (symmetric) bilinear forms
which are invariant on both the commutative associative algebras
and the Lie algebras in the following sense.

\begin{pro}\label{pro:tpa bilinear form}
    Let $(A,\cdot,[-,-])$ be a transposed Poisson algebra. If there is a 
    nondegenerate bilinear form $\mathcal{B}$ such that it is invariant on both $(A,\cdot)$ and $(A,[-,-])$,
    then Eq.~(\ref{eq:coherent TPA}) holds.

\end{pro}
\begin{proof}
    Let $x,y,z,w\in A$. We have
    \begin{eqnarray*}
        &&\mathcal{B}(2z\cdot [x,y]-[z\cdot x,y]-[x,z\cdot y],w)=\mathcal{B}(z,2[x,y]\cdot w-x\cdot[y,w]+y\cdot[x,w]),\\
        &&\mathcal{B}(2z\cdot [x,y]-[z\cdot x,y]-[x,z\cdot y],w)=\mathcal{B}(x,2[y,z\cdot w]-z\cdot[y,w]-[z\cdot y,w]).
    \end{eqnarray*}
    Then by the nondegeneracy of $\mathcal{B}$, we have
    \begin{equation}\label{eq:invariant tpa}
        2[x,y]\cdot z-x\cdot [y,z]+y\cdot[x,z]=0,\  2[y,z\cdot x]-z\cdot[y,x]-[z\cdot y,x]=0.
    \end{equation}
    Combining Eq.~(\ref{eq:invariant tpa}) with Eq.~(\ref{eq:defi:transposed Poisson algebra}), we get Eq.~(\ref{eq:coherent TPA}).
\end{proof}

\delete{ However, this compatibility condition is too specialized
since it asserts the interaction of the commutative associative
product and the Lie bracket to be trivial. Thus by Proposition
\ref{pro:tpa bilinear form}, defining a Manin triple of transposed
Poisson algebras to be a Manin triple of Lie algebras with respect
to the invariant bilinear form and a double construction of
commutative Frobenius algebras not only loses generality, but also
recovers the study of Manin triples of Poisson algebras. On the
other hand, by Proposition \ref{pro:comm 2-cocycle}, it is natural
to consider a symmetric nondegenerate bilinear form $\mathcal{B}$
on a transposed Poisson algebra $(A,\cdot,[-,-])$ such that it is
invariant on $(A,\cdot)$ and a commutative 2-cocycle on
$(A,[-,-])$. Moreover, a transposed Poisson algebra with such a
bilinear form leads to the notion of anti-pre-Lie Poisson
algebras.}

On the other hand, the relationship between anti-pre-Lie Poisson
algebras and transposed Poisson algebras with the nondegenerate
symmetric bilinear forms which are invariant on the commutative
associative algebras and commutative 2-cocycles on the Lie
algebras is given as follows.
\begin{pro} {\rm (\cite{LB2022})}
    Let $(A,\cdot,[-,-])$ be a transposed Poisson algebra. Suppose that there is a nondegenerate symmetric bilinear form $\mathcal{B}$ on  $A$ such that it is  invariant on $(A,\cdot)$ and a commutative 2-cocycle on $(A,[-,-])$.
    Then there is an anti-pre-Lie algebra $(A,\circ)$ defined by
Eq.~(\ref{eq:thm:commutative 2-cocycles and anti-pre-Lie
algebras}) through $\mathcal B$ such that $(A,\cdot,\circ)$ is a
compatible anti-pre-Lie Poisson algebra of $(A,\cdot,[-,-])$.
\end{pro}

Recall that a \textbf{Frobenius algebra} is a triple
$(A,\cdot,\mathcal B)$, where the pair $(A,\cdot)$ is an
associative algebra and $\mathcal{B}$ is a nondegenerate invariant
bilinear form on $(A,\cdot)$.  Furthermore, we recall the
definition of double constructions of (commutative) Frobenius
algebras.

\begin{defi}(\cite{Bai2010})
    Let $(A,\cdot_{A})$ be a commutative associative algebra. Suppose that there is a commutative associative algebra structure $\cdot_{A^{*}}$ on the dual space $A^{*}$.
    A \textbf{double construction of commutative Frobenius algebras}  associated to $(A,\cdot_{A})$ and $(A^{*},\cdot_{A^{*}})$  is a collection $\big((A\oplus A^{*},\cdot,\mathcal{B}_{d}),A,A^{*}\big)$,
    such that $(A\oplus A^{*},\cdot)$ is a commutative associative algebra which contains $(A,\cdot_{A})$ and $(A^{*},\cdot_{A^{*}})$ as subalgebras, and the nondegenerate symmetric bilinear form $\mathcal{B}_{d}$ defined
     by Eq.~(\ref{eq:defi:Manin triples of Lie algebras})
    makes $(A\oplus A^{*},\cdot,\mathcal{B}_{d})$ a Frobenius algebra.
    \delete{
    \begin{enumerate}
        \item $(A\oplus A^{*},\cdot)$ is a commutative associative algebra which contains $(A,\cdot_{A})$ and $(A^{*},\cdot_{A^{*}})$ as subalgebras.
        \item The nondegenerate symmetric bilinear form $\mathcal{B}_{d}$ given by Eq.~(\ref{eq:defi:Manin triples of Lie algebras}) is invariant on $(A\oplus A^{*},\cdot)$.
    \end{enumerate}
    We denote it by $(A\bowtie A^{*},\mathcal{B}_{d})$.}
\end{defi}

\delete{
The notation $(A\bowtie A^{*},\mathcal{B}_{d})$ is justified since we have the following result.

\begin{pro}(\cite{Bai2010})
    Let $(A,\cdot_{A})$ be a commutative associative algebra. Suppose that there is a commutative associative algebra structure $\cdot_{A^{*}}$ on the dual space $A^{*}$. Then there is a double
     construction of commutative Frobenius algebras $((A\oplus A^{*},\cdot,\mathcal{B}_{d}),A,A^{*})$ if and only if $(A,A^{*},-\mathcal{L}^{*}_{\cdot_{A}},-\mathcal{L}^{*}_{\cdot_{A^{*}}})$ is a matched pair
     of commutative associative algebras.
\end{pro}}



\begin{defi}
    Let $(A,\cdot_{A},[-,-]_{A})$ be a transposed Poisson algebra. Assume  that there is a transposed Poisson algebra structure $(A^{*},\cdot_{A^{*}},[-,-]_{A^{*}})$ on the dual space $A^{*}$. Suppose that there is a transposed Poisson algebra structure $(A\oplus A^{*},\cdot,[-,-])$ on the direct sum $A\oplus A^{*}$ of vector spaces, such that $\big((A\oplus A^{*},\cdot,\mathcal{B}_{d}),A,A^{*}\big)$ is a double construction of commutative Frobenius
    algebras, $\big((A\oplus A^{*},[-,-],\mathcal{B}_{d}), A,
A^{*}\big)$ is a Manin triple of Lie algebras with respect to the
commutative 2-cocycle,
and $(A,\cdot_{A},[-,-]_{A})$ and
$(A^{*},\cdot_{A^{*}},[-,-]_{A^{*}})$ are transposed Poisson
subalgebras of $(A\oplus A^{*},\cdot,[-,-])$. Such a structure is
called a \textbf{Manin triple of transposed Poisson algebras with
respect to the invariant bilinear form on $(A\oplus A^{*},\cdot)$
and the commutative 2-cocycle on $(A\oplus A^{*},[-,-])$}, or
simply a \textbf{Manin triple of transposed Poisson algebras}, and
is denoted by $\big((A\oplus
A^{*},\cdot,[-,-],\mathcal{B}_{d}),A,A^{*}\big)$.
\end{defi}


\begin{lem}\label{lem:111-Poisson}
Let $\big((A\oplus A^{*},\cdot,[-,-],\mathcal{B}_{d}),A,A^{*}\big)$ be a
Manin triple of transposed Poisson algebras. Then there exists a
compatible anti-pre-Lie algebra structure $\circ$ on $A\oplus A^*$
defined by Eq.~(\ref{eq:thm:commutative 2-cocycles and
anti-pre-Lie algebras}) through $\mathcal B_d$ such that $(A\oplus
A^*, \cdot, \circ)$ is a compatible anti-pre-Lie Poisson algebra.
Moreover, $(A,\cdot_A,\circ_A=
\circ|_{A\otimes A})$ and
$(A^*,\cdot_{A^*},\circ_{A^*}=
\circ|_{A^*\otimes A^*})$ are
anti-pre-Lie Poisson subalgebras,
whose sub-adjacent transposed Poisson algebras are
$(A,\cdot_{A},[-.-]_{A})$ and $(A^{*},\cdot_{A^{*}},
[-,-]_{A^{*}})$ respectively.
\end{lem}

\delete{
\begin{thm}
Let $(A,\cdot_{A},\circ_{A})$ and $(A^{*},\cdot_{A^{*}},\circ_{A^{*}})$ be two anti-pre-Lie Poisson algebras, and their sub-adjacent transposed Poisson algebras be $(A,\cdot_{A},[-,-]_{A})$ and $(A^{*},\cdot_{A^{*}},[-,-]_{A^{*}})$ respectively. Then the following conditions are equivalent:
\begin{enumerate}
    \item There is a Manin triple of transposed Poisson algebras $((A\oplus A^{*},\cdot,\circ,\mathcal{B}_{d}),A,A^{*})$ such that the compatible anti-pre-Lie Poisson algebra $(A\oplus A^{*},\cdot,\circ)$
in which $\circ$ defined by Eq.~(\ref{eq:thm:commutative
2-cocycles and anti-pre-Lie algebras}) through $\mathcal B_d$
contains $(A,\cdot_{A},\circ_{A})$ and
$(A^{*},\cdot_{A^{*}},\circ_{A^{*}})$ as anti-pre-Lie Poisson
subalgebras.
    \item
    $(A,A^{*},-\mathcal{L}^{*}_{\cdot_{A}},-\mathrm{ad}^{*}_{A},\mathcal{R}^{*}_{\circ_{A}},-\mathcal{L}^{*}_{\cdot_{A^{*}}},-\mathrm{ad}^{*}_{A^{*}},\mathcal{R}^{*}_{A^{*}})$ is a matched pair of anti-pre-Lie Poisson algebras;
    \item $(A,A^{*},-\mathcal{L}^{*}_{\cdot_{A}},-\mathcal{L}^{*}_{\circ_{A}},-\mathcal{L}^{*}_{\cdot_{A^{*}}}-\mathcal{L}^{*}_{\circ_{A^{*}}})$ is a matched pair of transposed Poisson algebras.
\end{enumerate}
\end{thm}
\begin{proof}
It is similar to the proof of Proposition \ref{thm:Manin triples and matched pairs}.
\end{proof}}

\begin{proof}
    It is similar to the proof of Lemma~\ref{lem:111}.
    \end{proof}

\begin{thm}\label{thm:ddd}
    Let $(A,\cdot_{A},\circ_{A})$ and $(A^{*},\cdot_{A^{*}},\circ_{A^{*}})$ be two anti-pre-Lie Poisson algebras and their sub-adjacent transposed Poisson algebras be $(A,\cdot_{A},[-,-]_{A})$ and $(A^{*},\cdot_{A^{*}},[-,-]_{A^{*}})$ respectively. Then the following conditions are equivalent:
    \begin{enumerate}
        \item There is a Manin triple of transposed Poisson algebras $\big((A\oplus A^{*},\cdot,[-,-],\mathcal{B}_{d}),A,A^{*}\big)$ such that the
        compatible anti-pre-Lie Poisson algebra $(A\oplus A^{*},\cdot,\circ)$ in which $\circ$ is defined by Eq.~(\ref{eq:thm:commutative
2-cocycles and anti-pre-Lie algebras}) through $\mathcal B_d$
contains $(A,\cdot_{A},\circ_{A})$ and
$(A^{*},\cdot_{A^{*}},\circ_{A^{*}})$ as anti-pre-Lie Poisson
subalgebras.
        \item
        $(A,A^{*},-\mathcal{L}^{*}_{\cdot_{A}},-\mathrm{ad}^{*}_{A},\mathcal{R}^{*}_{\circ_{A}},-\mathcal{L}^{*}_{\cdot_{A^{*}}},-\mathrm{ad}^{*}_{A^{*}},\mathcal{R}^{*}_{\circ_{A^{*}}})$ is a matched pair of anti-pre-Lie Poisson algebras.
        \item $(A,A^{*},-\mathcal{L}^{*}_{\cdot_{A}},-\mathcal{L}^{*}_{\circ_{A}},-\mathcal{L}^{*}_{\cdot_{A^{*}}}-\mathcal{L}^{*}_{\circ_{A^{*}}})$ is a matched pair of transposed Poisson algebras.
    \end{enumerate}
\end{thm}
\begin{proof}
Note that  by \cite{Bai2010}, there is a double construction of
commutative Frobenius algebras $\big((A\oplus
A^{*},\cdot,\mathcal{B}_{d}),A,A^{*}\big)$ if and only if
$(A,A^{*},-\mathcal{L}^{*}_{\cdot_{A}},-\mathcal{L}^{*}_{\cdot_{A^{*}}})$
is a matched pair of commutative associative algebras. Hence the
conclusion follows from a proof similar to the one of
Proposition~\ref{thm:Manin triples and matched pairs}.
\end{proof}


Recall (\cite{Bai2010}) that a  \textbf{cocommutative
coassociative coalgebra} is a pair $(A,\Delta)$, such that $A$ is
a vector space and $\Delta:A\rightarrow A\otimes A$ is a linear
map satisfying
\begin{eqnarray}
    \tau\Delta&=&\Delta,\label{eq:symmetric}\\
    (\mathrm{id}\otimes \Delta)\Delta&=&(\Delta\otimes\mathrm{id})\Delta.\label{AssoCo}
\end{eqnarray}
\delete{
\begin{equation}\label{eq:symmetric}
    \tau\Delta=\Delta,
\end{equation}
\begin{equation}\label{AssoCo}
    (\mathrm{id}\otimes \Delta)\Delta=(\Delta\otimes\mathrm{id})\Delta.
\end{equation}}
A \textbf{commutative and cocommutative infinitesimal bialgebra}
is a triple $(A,\cdot,\Delta)$ such that the pair $(A,\cdot)$ is a
commutative associative algebra, the pair $(A,\Delta)$ is a
cocommutative coassociative coalgebra, and the following equation
holds:
\begin{equation}\label{AssoBia}
    \Delta(x\cdot y)=\big(\mathcal{L}_{\cdot}(x)\otimes \mathrm{id}\big)\Delta(y)+\big(\mathrm{id}\otimes\mathcal{L}_{\cdot}(y)\big)\Delta(x),\;\;\forall x,y\in A.
\end{equation}

\begin{defi}\label{defi:anti-pre-Lie Poisson coalg}
An \textbf{anti-pre-Lie Poisson coalgebra} is a triple
$(A,\Delta,\delta)$, such that $(A,\Delta)$ is a cocommutative
coassociative coalgebra, $(A,\delta)$ is an anti-pre-Lie
coalgebra, and the following conditions are satisfied:
\begin{eqnarray}
    2(\delta\otimes\mathrm{id})\Delta-2(\tau\otimes\mathrm{id})(\delta\otimes\mathrm{id})\Delta&=&(\tau\otimes\mathrm{id})(\mathrm{id}\otimes\delta)\Delta-(\mathrm{id}\otimes\delta)\Delta,\label{eq:defi:anti-pre-Lie Poisson coalg1}\\
    2(\mathrm{id}\otimes\Delta)\delta&=&(\mathrm{id}\otimes\tau)(\delta\otimes\mathrm{id})\delta+(\delta\otimes\mathrm{id})\Delta.\label{eq:defi:anti-pre-Lie Poisson coalg2}
\end{eqnarray}
\delete{
\begin{equation}\label{eq:defi:anti-pre-Lie Poisson coalg1}
2(\delta\otimes\mathrm{id})\Delta-2(\tau\otimes\mathrm{id})(\delta\otimes\mathrm{id})\Delta=(\tau\otimes\mathrm{id})(\mathrm{id}\otimes\delta)\Delta-(\mathrm{id}\otimes\delta)\Delta,
\end{equation}
\begin{equation}\label{eq:defi:anti-pre-Lie Poisson coalg2}
2(\mathrm{id}\otimes\Delta)\delta=(\mathrm{id}\otimes\tau)(\delta\otimes\mathrm{id})\delta+(\delta\otimes\mathrm{id})\Delta.
\end{equation} }
\end{defi}

\begin{pro}\label{pro:Poisson algs and Poisson coalgs}
Let $A$ be a vector space and $\Delta,\delta:A\rightarrow A\otimes
A$ be linear maps. Let $\cdot_{A^{*}},\circ_{A^{*}}:A^{*}\otimes
A^{*}\rightarrow A^{*} $ be linear duals of $\Delta$ and $\delta$
respectively. Then $(A,\Delta,\delta)$ is an anti-pre-Lie Poisson
coalgebra if and only if $(A^{*},\cdot_{A^{*}},\circ_{A^{*}})$ is
an anti-pre-Lie Poisson algebra.
\end{pro}

\begin{proof}
    By \cite{Bai2010}, $(A,\Delta)$ is a cocommutative coassociative coalgebra if and only if $(A^{*},\cdot_{A^{*}})$ is a commutative associative algebra. Moreover, by a proof similar to the one of Proposition \ref{pro:anti-pre-Lie
        coalgebras and anti-pre-Lie algebras},
    Eqs.~(\ref{eq:defi:anti-pre-Lie Poisson1})-(\ref{eq:defi:anti-pre-Lie Poisson2}) hold  on $A^*$ if and only if Eqs.~(\ref{eq:defi:anti-pre-Lie Poisson coalg1})-(\ref{eq:defi:anti-pre-Lie Poisson coalg2}) hold respectively. Hence the conclusion follows.
\end{proof}


\delete{
\begin{defi}
An \textbf{anti-pre-Lie Poisson bialgebra} is a collection
$(A,\cdot,\circ,\Delta,\delta)$, such that $(A,\cdot,\Delta)$ is a
commutative and cocommutative infinitesimal bialgebra,
$(A,\circ,\delta)$ is an anti-pre-Lie Poisson bialgebra, and the
following compatible conditions are satisfied for all $x,y\in A$:
\begin{small}
    \begin{equation}\label{eq:Poisson bialg 1}
        2(\mathcal{L}_{\circ}(x)\otimes\mathrm{id})\Delta(y)-2(\mathrm{id}\otimes\mathcal{L}_{\cdot}(y))\delta(x)+\delta(x\cdot y)+(\mathcal{L}_{\cdot}(y)\otimes\mathrm{id})\delta(x)-(\mathrm{id}\otimes\mathrm{ad}(x))\Delta(y)=0,
    \end{equation}
\begin{equation}\label{eq:Poisson bialg 2}
2\Delta([x,y])+(\mathrm{id}\otimes\mathrm{ad}(y))\delta(x)-(\mathcal{L}_{\cdot}(x)\otimes\mathrm{id})\delta(y)-(\mathrm{id}\otimes\mathrm{ad}(x))\delta(y)+(\mathcal{L}_{\cdot}(y)\otimes\mathrm{id})\delta(x)=0,
\end{equation}
    \begin{equation}\label{eq:Poisson bialg 3}
    2(\mathrm{id}\otimes\mathcal{L}_{\cdot}(y))\delta(x)-2(\mathcal{L}_{\circ}(x)\otimes\mathrm{id})\Delta(y)+\Delta(x\circ y)+(\mathcal{R}_{\circ}(y)\otimes\mathrm{id})\Delta(x)+\tau(\mathcal{L}_{\cdot}(x)\otimes\mathrm{id})\delta(y)-(\mathrm{id}\otimes\mathcal{L}_{\cdot}(x))\delta(y)=0,
    \end{equation}
\begin{equation}\label{eq:Poisson bialg 4}
(\tau-\mathrm{id})(2\delta(x\cdot
y)-(\mathcal{L}_{\cdot}(x)\otimes\mathrm{id})\delta(y)-(\mathrm{id}\otimes\mathcal{L}_{\cdot}(x))\delta(y)-(\mathrm{id}\otimes\mathcal{R}_{\circ}(y))\Delta(x))=0.
\end{equation}
\end{small}
\end{defi}
}

\begin{defi}
Let $(A,\cdot_{A},\circ_{A})$ be an anti-pre-Lie Poisson algebra
and $(A,\Delta,\delta)$ be an anti-pre-Lie Poisson coalgebra.
Suppose that the following conditions are satisfied:
\begin{enumerate}
    \item $(A,\cdot_{A},\Delta)$ is a commutative and cocommutative infinitesimal bialgebra.
    \item $(A,\circ_{A},\delta)$ is an anti-pre-Lie bialgebra.
    \item The following equations hold:
    \begin{small}
        \begin{equation}\label{eq:Poisson bialg 1}
            2\big(\mathcal{L}_{\circ_{A}}(x)\otimes\mathrm{id}\big)\Delta(y)-2\big(\mathrm{id}\otimes\mathcal{L}_{\cdot_{A}}(y)\big)\delta(x)+\delta(x\cdot_{A} y)+\big(\mathcal{L}_{\cdot_{A}}(y)\otimes\mathrm{id}\big)\delta(x)-\big(\mathrm{id}\otimes\mathrm{ad}_{A}(x)\big)\Delta(y)=0,
        \end{equation}
        \begin{equation}\label{eq:Poisson bialg 2}
            2\Delta([x,y]_{A})+\big(\mathrm{id}\otimes\mathrm{ad}_{A}(y)\big)\delta(x)-\big(\mathcal{L}_{\cdot_{A}}(x)\otimes\mathrm{id}\big)\delta(y)-\big(\mathrm{id}\otimes\mathrm{ad}_{A}(x)\big)\delta(y)+\big(\mathcal{L}_{\cdot_{A}}(y)\otimes\mathrm{id}\big)\delta(x)=0,
        \end{equation}
        \begin{equation}\label{eq:Poisson bialg 3}
            \begin{array}{ll}
            &2\big(\mathrm{id}\otimes\mathcal{L}_{\cdot_{A}}(y)\big)\delta(x)-2\big(\mathcal{L}_{\circ_{A}}(x)\otimes\mathrm{id}\big)\Delta(y)+\Delta(x\circ_{A} y)+\big(\mathcal{R}_{\circ_{A}}(y)\otimes\mathrm{id}\big)\Delta(x)\\
            &+\tau\big(\mathcal{L}_{\cdot_{A}}(x)\otimes\mathrm{id}\big)\delta(y)-\big(\mathrm{id}\otimes\mathcal{L}_{\cdot_{A}}(x)\big)\delta(y)=0,
            \end{array}
        \end{equation}
        \begin{equation}\label{eq:Poisson bialg 4}
            (\tau-\mathrm{id}^{\otimes 2})\Big(2\delta(x\cdot_{A}
            y)-\big(\mathcal{L}_{\cdot_{A}}(x)\otimes\mathrm{id}\big)\delta(y)-\big(\mathrm{id}\otimes\mathcal{L}_{\cdot_{A}}(x)\big)\delta(y)-\big(\mathrm{id}\otimes\mathcal{R}_{\circ_{A}}(y)\big)\Delta(x)\Big)=0,
        \end{equation}
        for all $x,y\in A$.
    \end{small}
\end{enumerate}
Such a structure is called an \textbf{anti-pre-Lie Poisson
bialgebra}. We denote it by
$(A,\cdot_{A},\circ_{A},\Delta,\delta)$.
\end{defi}

\begin{thm}\label{thm:eee}
Let $(A,\cdot_{A},\circ_{A})$ be an anti-pre-Lie Poisson algebra.
Suppose that there is an anti-pre-Lie Poisson algebra structure
$(A^{*},\cdot_{A^{*}},\circ_{A^{*}})$ on the dual space $A^{*}$.
Let $(A,\cdot_A,[-,-]_A)$ and $(A^*,\cdot_{A^*}$, $[-,-]_{A^*})$
be the sub-adjacent transposed Poisson algebras respectively. Let
$\Delta,\delta:A\rightarrow A\otimes A$ be the linear duals of
$\cdot_{A^{*}}$ and $\circ_{A^{*}}$ respectively. Then
$(A,\cdot_{A},\circ_{A},\Delta,\delta)$ is an anti-pre-Lie Poisson
bialgebra if and only if
$(A,A^{*},-\mathcal{L}^{*}_{\cdot_{A}},-\mathcal{L}^{*}_{\circ_{A}},-\mathcal{L}^{*}_{\cdot_{A^{*}}},-\mathcal{L}^{*}_{\circ_{A^{*}}})$
is a matched pair of transposed Poisson algebras.
    \end{thm}
\begin{proof}
    By \cite{Bai2010}, $(A,A^{*}, -\mathcal{L}^{*}_{\cdot_{A}},-\mathcal{L}^{*}_{\cdot_{A^{*}}})$ is a matched pair of commutative associative algebras
    if and only if $(A,\cdot_{A},\Delta)$ is a commutative and cocommutative infinitesimal bialgebra,
    and by Theorem \ref{thm:equivalence matched pairs of Lie algebras and anti-pre-Lie bialgebras},
     $(A,A^{*}, -\mathcal{L}^{*}_{\circ_{A}},-\mathcal{L}^{*}_{\circ_{A^{*}}})$ is a matched pair of Lie algebras if and only if $(A,\circ_{A},\delta)$ is an anti-pre-Lie bialgebra.
By Proposition~\ref{pro:anti-pre-Lie Poisson},
$(-\mathcal{L}^{*}_{\cdot_{A}},-\mathcal{L}^{*}_{\circ_{A}}, A^*)$
and
$(-\mathcal{L}^{*}_{\cdot_{A^{*}}},-\mathcal{L}^{*}_{\circ_{A^{*}}},A)$
are representations of the transposed Poisson algebras
$(A,\cdot_A,[-,-]_A)$ and $(A^*,\cdot_{A^*},[-,-]_{A^*})$
respectively. By Proposition \ref{pro:Poisson algs and Poisson
coalgs}, $(A,\Delta,\delta)$ is an anti-pre-Lie Poisson coalgebra
if and only if $(A^{*},\cdot_{A^{*}},\circ_{A^{*}})$ is an
anti-pre-Lie Poisson algebra. Moreover, for all $x,y\in A, a^{*},
b^{*}\in A^{*}$, we have
\begin{eqnarray*}
\langle 2\mathcal{L}^{*}_{\cdot_{A^{*}}}\big(\mathcal{L}^{*}_{\circ_{A}}(x)a^{*}\big)y,b^{*}\rangle&=&-\langle 2y, \mathcal{L}^{*}_{\circ_{A}}(x)a^{*}\cdot_{A^{*}}b^{*}\rangle=\langle 2\big(\mathcal{L}_{\circ_{A}}(x)\otimes\mathrm{id}\big)\Delta(y),a^{*}\otimes b^{*}\rangle,\\
\langle 2y\cdot_{A}\mathcal{L}^{*}_{\circ_{A^{*}}}(a^{*})x, b^{*}\rangle&=&\langle 2x, a^{*}\circ_{A^{*}}\mathcal{L}^{*}_{\cdot_{A}}(y)b^{*}\rangle=-\langle 2\big(\mathrm{id}\otimes\mathcal{L}_{\cdot_{A}}(y)\big)\delta(x), a^{*}\otimes b^{*}\rangle,\\
\langle\mathcal{L}^{*}_{\circ_{A^{*}}}(a^{*})(x\cdot_{A}y), b^{*}\rangle&=&\-\langle \delta(x\cdot_{A}y), a^{*}\circ_{A^{*}}b^{*}\rangle,\\
-\langle\mathcal{L}^{*}_{\circ_{A^{*}}}\big(\mathcal{L}^{*}_{\cdot_{A}}(y)a^{*}\big)x,b^{*}\rangle&=&\langle x, \mathcal{L}^{*}_{\cdot_{A}}(y)a^{*}\circ_{A^{*}}b^{*}\rangle=-\langle\big(\mathcal{L}_{\cdot_{A}}(y)\otimes\mathrm{id}\big)\delta(x), a^{*}\otimes b^{*}\rangle,\\
-\langle[x,\mathcal{L}^{*}_{\cdot_{A^{*}}}(a^{*})y]_{A},b^{*}\rangle&=&-\langle
y, a^{*}\cdot_{A^{*}}
\mathrm{ad}^{*}_{A}(x)b^{*}\rangle=\langle\big(\mathrm{id}\otimes
\mathrm{ad}_{A}(x)\big)\Delta(y), a^{*}\otimes b^{*}\rangle.
\end{eqnarray*}
Thus Eq.~(\ref{eq:defi:MP TPA a}) holds if and only if Eq.~(\ref{eq:Poisson bialg 1}) holds for $\mu_{A}=-\mathcal{L}^{*}_{\cdot_{A}}$, $\mu_{B}=-\mathcal{L}^{*}_{\cdot_{A^{*}}}$,
$\rho_{A}=-\mathcal{L}^{*}_{\circ_{A}}$, $\rho_{B}=-\mathcal{L}^{*}_{\circ_{A^{*}}}$.
Similarly Eqs.~(\ref{eq:defi:MP TPA b})-(\ref{eq:defi:MP TPA d}) hold if and only if Eqs.~(\ref{eq:Poisson bialg 2})-(\ref{eq:Poisson bialg 4}) hold for $\mu_{A}=-\mathcal{L}^{*}_{\cdot_{A}}$, $\mu_{B}=-\mathcal{L}^{*}_{\cdot_{A^{*}}}$,
$\rho_{A}=-\mathcal{L}^{*}_{\circ_{A}}$, $\rho_{B}=-\mathcal{L}^{*}_{\circ_{A^{*}}}$ respectively.  Hence the conclusion follows.
\end{proof}

Combining Theorems \ref{thm:ddd} and \ref{thm:eee} together, we
have

\begin{cor}
Let $(A,\cdot_{A},\circ_{A})$ be an anti-pre-Lie Poisson algebra.
Suppose that there is an anti-pre-Lie Poisson algebra structure
$(A^{*},\cdot_{A^{*}},\circ_{A^{*}})$ on the dual space $A^{*}$
and $\Delta,\delta:A\rightarrow A\otimes A$ are the linear duals
of $\cdot_{A^{*}}$ and $\circ_{A^{*}}$ respectively. Let
$(A,\cdot_{A},[-,-]_{A})$ and
$(A^{*},\cdot_{A^{*}},[-,-]_{A^{*}})$ be the sub-adjacent
transposed Poisson algebras of $(A,\cdot_{A},\circ_{A})$ and
$(A^{*},\cdot_{A^{*}},\circ_{A^*})$ respectively. Then the
following conditions are equivalent:
\begin{enumerate}
\item There is a Manin triple of transposed Poisson algebras
$\big((A\oplus A^{*},\cdot,[-,-],\mathcal{B}_{d}),A,A^{*}\big)$ such that
the
        compatible anti-pre-Lie Poisson algebra $(A\oplus A^{*},\cdot,\circ)$ in which $\circ$ is defined by Eq.~(\ref{eq:thm:commutative
2-cocycles and anti-pre-Lie algebras}) through $\mathcal B_d$
contains $(A,\cdot_{A},\circ_{A})$ and
$(A^{*},\cdot_{A^{*}},\circ_{A^{*}})$ as anti-pre-Lie Poisson
subalgebras. \item
$(A,A^{*},-\mathcal{L}^{*}_{\cdot_{A}},-\mathrm{ad}^{*}_{A},\mathcal{R}^{*}_{\circ_{A}},-\mathcal{L}^{*}_{\cdot_{A^{*}}},-\mathrm{ad}^{*}_{A^{*}},{\mathcal{R}^{*}_{\circ_{A^{*}}}})$
is a matched pair of anti-pre-Lie Poisson algebras. \item
$(A,A^{*},-\mathcal{L}^{*}_{\cdot_{A}},-\mathcal{L}^{*}_{\circ_{A}},-\mathcal{L}^{*}_{\cdot_{A^{*}}}-\mathcal{L}^{*}_{\circ_{A^{*}}})$
is a matched pair of transposed Poisson algebras. \item
$(A,\cdot_{A},\circ_{A},\Delta,\delta)$ is an anti-pre-Lie Poisson
bialgebra.
\end{enumerate}
\end{cor}


\subsection{Coboundary anti-pre-Lie Poisson bialgebras and the anti-pre-Lie Poisson Yang-Baxter equation}\

Recall (\cite{Bai2010}) that a commutative and cocommutative
infinitesimal bialgebra $(A,\cdot,\Delta)$ is called
\textbf{coboundary} if there exists an $r\in A\otimes A$ such that
\begin{equation}\label{AssoCob}
    \Delta(x):=\Delta_{r}(x):=\big(\mathrm{id}\otimes\mathcal{L}_{\cdot}(x)-\mathcal{L}_{\cdot}(x)\otimes \mathrm{id}\big)r,\ \ \forall x\in A.
\end{equation}

Let $(A,\cdot)$ be a commutative associative algebra and $r\in
A\otimes A$. Let $\Delta:A\rightarrow A\otimes A$ be a linear map
defined by Eq.~(\ref{AssoCob}). Then by \cite{Bai2010},
$(A,\cdot,\Delta)$ is a commutative and cocommutative
infinitesimal bialgebra if and only if for all $x\in A,$
\begin{eqnarray}
    \big(\mathrm{id}\otimes\mathcal{L}_{\cdot}(x)-\mathcal{L}_{\cdot}(x)\otimes \mathrm{id}\big)\big(r+\tau(r)\big)&=&0,\label{eq:AYBE1}\\
    \big(\mathrm{id}\otimes \mathrm{id}\otimes\mathcal{L}_{\cdot}(x)-\mathcal{L}_{\cdot}(x)\otimes \mathrm{id}\otimes
    \mathrm{id}\big)\textbf{A}(r)&=&0,\label{eq:AYBE2}
\end{eqnarray}
\delete{
\begin{equation}\label{eq:AYBE1}
    (\mathrm{id}\otimes\mathcal{L}_{\cdot}(x)-\mathcal{L}_{\cdot}(x)\otimes \mathrm{id})(r+\tau(r))=0,
\end{equation}
\begin{equation}\label{eq:AYBE2}
    (\mathrm{id}\otimes \mathrm{id}\otimes\mathcal{L}_{\cdot}(x)-\mathcal{L}_{\cdot}(x)\otimes \mathrm{id}\otimes
    \mathrm{id})\textbf{A}(r)=0,
\end{equation}}
where $\textbf{A}(r)$ is defined by Eq.~(\ref{eq:AYBE}).

 \delete{ where
\begin{equation}\label{eq:AYBE}
    \textbf{A}(r)=r_{12}\cdot r_{13}-r_{12}\cdot r_{23}+r_{13}\cdot
    r_{23},
\end{equation}
and
$$r_{12}\cdot r_{13}=\sum_{i,j}a_{i}\cdot a_{j}\otimes b_{i}\otimes b_{j}, r_{12}\cdot r_{23}=\sum_{i,j}a_{i}\otimes b_{i}\cdot a_{j}\otimes b_{j}, r_{13}\cdot r_{23}=\sum_{i,j}a_{i}\otimes  a_{j}\otimes b_{i}\cdot b_{j}.$$
The equation $\textbf{A}(r)=0$ is called the \textbf{associative
    Yang-Baxter equation} (AYBE) in $(A,\cdot)$.
}


\begin{defi}
An anti-pre-Lie Poisson bialgebra  $(A,\cdot,\circ,\Delta,\delta)$
is called \textbf{coboundary} if there exists an
$r
\in A\otimes A$ such that
Eqs.~(\ref{AssoCob}) and (\ref{eq:defi:coboundary anti-pre-Lie
bialgebras}) hold.
\end{defi}

A coboundary anti-pre-Lie Poisson bialgebra is clearly both a coboundary commutative and cocommutative infinitesimal bialgebra and a coboundary anti-pre-Lie bialgebra.

\begin{pro}\label{pro:fff2}
Let $(A,\cdot,\circ)$ be an anti-pre-Lie Poisson algebra  and
$r=\sum\limits_{i}a_{i}\otimes b_{i}\in A\otimes A$. Let
$\Delta=\Delta_{r}$ and $\delta=\delta_{r}$ be two linear maps
defined by Eqs.~(\ref{AssoCob}) and (\ref{eq:defi:coboundary
anti-pre-Lie bialgebras}) respectively.
\begin{enumerate}
    \item\label{it:1} Eq.~(\ref{eq:defi:anti-pre-Lie Poisson coalg1}) holds if and only if for all $x\in
    A$, the following {equation} holds:
    \begin{small}
    \begin{equation}\label{eq:TPBA1}
        \begin{split}
        &\big(2\mathrm{id}\otimes\mathrm{id}\otimes\mathcal{L}_{\cdot}(x)-\mathcal{L}_{\cdot}(x)\otimes\mathrm{id}\otimes\mathrm{id}-\mathrm{id}\otimes\mathcal{L}_{\cdot}(x)\otimes\mathrm{id}\big)\textbf{T}(r)\\
        &\ \ +\sum\limits_{j}\Big(2\big(\mathrm{ad}(a_{j})\otimes\mathrm{id}\otimes \mathcal{L}_{\cdot}(x)-\mathrm{id}\otimes\mathcal{L}_{\circ}(a_{j})\otimes\mathcal{L}_{\cdot}(x)-\mathrm{ad}(x\cdot a_{j})\otimes\mathrm{id}\otimes\mathrm{id}+\mathrm{id}\otimes\mathcal{L}_{\circ}(x\cdot a_{j})\otimes\mathrm{id}\big)\\
        &\ \ +\mathcal{R}_{\circ}(a_{j})\otimes \mathcal{L}_{\cdot}(x)\otimes\mathrm{id}-\mathcal{R}_{\circ}(a_{j})\mathcal{L}_{\cdot}(x)\otimes\mathrm{id}\otimes\mathrm{id}\Big)\Big(\big(r+\tau(r)\big)\otimes b_{j}\Big)=0.
        \end{split}
    \end{equation}
\end{small}
\item\label{it:2} Eq.~(\ref{eq:defi:anti-pre-Lie Poisson coalg2})
holds if and only if for all $x\in A$, the following {equation}
holds:
\begin{small}
    \begin{equation}\label{eq:TPBA2}
        \begin{split}
            &\big(2\mathcal{L}_{\circ}(x)\otimes\mathrm{id}\otimes\mathrm{id}-\mathrm{id}\otimes\mathrm{ad}(x)\otimes\mathrm{id}\big)\textbf{A}(r)-\big(\mathrm{id}\otimes\mathrm{id}\otimes\mathcal{L}_{\cdot}(x)\big)(\mathrm{id}\otimes \tau)\textbf{T}(r)\\
            &\ \ +\sum\limits_{j} \big(\mathrm{id}\otimes\mathrm{ad}(x)\otimes\mathcal{L}_{\cdot}(b_{j})+\mathrm{id}\otimes\mathrm{id}\otimes\mathcal{L}_{\cdot}(x)\mathcal{L}_{\circ}(b_{j})-\mathrm{id}\otimes\mathrm{id}\otimes\mathcal{L}_{\cdot}(b_{j})\mathcal{L}_{\circ}(x)\\
            &\ \ -\mathrm{id}\otimes\mathrm{ad}(b_{j})\otimes\mathcal{L}_{\cdot}(x)\big)\Big(a_{j}\otimes\big(r+\tau(r)\big)\Big)=0.
        \end{split}
    \end{equation}
\end{small}
\item \label{it:3} Eq.(\ref{eq:Poisson bialg 1}) holds
automatically. \item\label{it:4} Eq.(\ref{eq:Poisson bialg 2})
holds automatically. \item\label{it:5}Eq.(\ref{eq:Poisson bialg
3}) holds if and only if for all $x,y\in A$, the following
{equation} holds:
    \begin{equation}\label{eq:TPBA3}
        \big(\mathrm{ad}(y)\otimes\mathcal{L}_{\cdot}(x)-\mathrm{id}\otimes\mathcal{L}_{\cdot}(x)\mathcal{L}_{\circ}(y)\big)\big(r+\tau(r)\big)=0.
    \end{equation}
\item\label{it:6} Eq.(\ref{eq:Poisson bialg 4}) holds if and only
if for all $x,y\in A$, the following {equation} holds:
\begin{small}
\begin{equation}\label{eq:TPBA4}
    \big(\mathcal{L}_{\cdot}(x)\otimes\mathcal{L}_{\circ}(y)-\mathcal{L}_{\circ}(y)\otimes\mathcal{L}_{\cdot}(x)+2\mathcal{L}_{\circ}(x\cdot y)\otimes\mathrm{id}-2\mathrm{id}\otimes\mathcal{L}_{\circ}(x\cdot y)+\mathrm{id}\otimes\mathcal{L}_{\cdot}(x)\mathcal{L}_{\circ}(y)-\mathcal{L}_{\cdot}(x)\mathcal{L}_{\circ}(y)\otimes\mathrm{id}\big)\big(r+\tau(r)\big)=0.
\end{equation}
\end{small}
\end{enumerate}
\end{pro}

\begin{proof}
It follows from the careful interpretation, which is put into the
Appendix.
\end{proof}


Therefore, with Theorem~\ref{thm:cob bialg} together, we have the
following conclusion.

\begin{thm}\label{thm:lllll}
Let $(A,\cdot,\circ)$ be an anti-pre-Lie Poisson  algebra and
$r=\sum\limits_{i}a_{i}\otimes b_{i}\in A\otimes A$. Let
$\Delta=\Delta_{r}$ and $\delta=\delta_{r}$ be two linear maps
defined by Eqs.~(\ref{AssoCob}) and (\ref{eq:defi:coboundary
anti-pre-Lie bialgebras}) respectively. Then
$(A,\cdot,\circ,\Delta,\delta)$ is an anti-pre-Lie Poisson
bialgebra if and only if Eqs.~(\ref{eq:pro:cob coalg1})-
(\ref{eq:pro:coboundary anti-pre-Lie bialgebras1}) and
(\ref{eq:AYBE1})-(\ref{eq:TPBA4}) hold.
\end{thm}

\begin{defi}
Let $(A,\cdot,\circ)$ be an anti-pre-Lie Poisson  algebra and
$r\in A\otimes A$. We say $r$ is a solution of the
\textbf{anti-pre-Lie Poisson Yang-Baxter equation} or the
\textbf{APLP-YBE} in short, in $(A,\cdot,\circ)$ if $r$ satisfies
both the AYBE and the APL-YBE, that is,
\begin{equation}
    \textbf{A}(r)=\textbf{T}(r)=0.
\end{equation}
\end{defi}

\begin{ex}
     Let $(A,\cdot)$ be a commutative associative algebra with a derivation $P$ and $(A,\circ)$ be the anti-pre-Lie algebra defined by Eq.~(\ref{eq:diff anti-pre-Lie}).
     Then by \cite{LB2022}, $(A,\cdot,\circ)$ is an anti-pre-Lie Poisson algebra. Moreover, by Example \ref{ex:AYBE}, if $r$ is a solution of the AYBE in $(A,\cdot)$ satisfying Eq.~(\ref{eq:-P}), then $r$ is also a
     solution of the APLP-YBE in $(A,\cdot,\circ)$.
\end{ex}

\begin{cor}
Let $(A,\cdot,\circ)$ be an anti-pre-Lie Poisson algebra and $r\in
A\otimes A$ be a skew-symmetric solution of the APLP-YBE in
$(A,\cdot,\circ)$. Then $(A,\cdot,\circ,\Delta,\delta)$  is an
anti-pre-Lie Poisson bialgebra, where $\Delta=\Delta_{r}$ and
$\delta=\delta_{r}$ are defined by Eqs.~(\ref{AssoCob}) and
(\ref{eq:defi:coboundary anti-pre-Lie bialgebras}) respectively.
\end{cor}

\begin{proof}
It follows from Theorem~\ref{thm:lllll}.
\end{proof}

\delete{Recall a skew-symmetric bilinear form $\mathcal{B}$ on an
(commutative) associative algebra $(A,\cdot)$ is called a
\textbf{Connes cocycle} (\cite{Bai2010}) if
\begin{equation}
\mathcal{B}(x\cdot y,z)+\mathcal{B}(y\cdot z,x)+\mathcal{B}(z\cdot x,y)=0,\ \ \forall x,y,x\in A.
\end{equation}
Then we have the following theorem.

\begin{thm}
Let $(A,\cdot,\circ)$ be an anti-pre-Lie Poisson algebra and $r\in
A\otimes A$. Suppose that $r$ is skew-symmetric and nondegenerate.
Then $r$ is a solution of the APLP-YBE if and only if the bilinear
form $\mathcal{B}_{r}$ given by Eq.~(\ref{eq:Br}) is both a Connes
cocycle on $(A,\cdot)$ and an anti-2-cocycle on $(A,\circ)$.
\end{thm}
\begin{proof}
    By \cite{Bai2010}, $r$ is a solution of the AYBE in $(A,\cdot)$ if and only if
    $\mathcal{B}_{r}$ is a Connes cocycle on $(A,\cdot)$. Hence the conclusion follows from Theorem \ref{thm:anti-2-cocycle}.
    \end{proof}}

\subsection{$\mathcal{O}$-operators of anti-pre-Lie Poisson algebras and pre-(anti-pre-Lie Poisson) algebras}\

Let $(A,\cdot)$ be a commutative associative algebra and $(\mu,V)$
be a representation of $(A,\cdot)$. Recall (\cite{Bai2010}) that a
linear map $T:V\rightarrow A$ is called an
\textbf{$\mathcal{O}$-operator of $(A,\cdot)$ associated to
}$(\mu,V)$ if
\begin{equation}\label{eq:O-op}
T(u)\cdot T(v)=T\Big(\mu\big(T(u)\big)v+\mu\big(T(v)\big)u\Big),\;\;\forall u,v\in V.
\end{equation}

\begin{defi}
Let $(A,\cdot,\circ)$ be an anti-pre-Lie Poisson algebra and
$(\mu,l_{\circ},r_{\circ},V)$ be a representation of
$(A,\cdot,\circ)$. A linear map $T:V\rightarrow A$ is called an
{\bf $\mathcal{O}$-operator of $(A,\cdot,\circ)$ associated to
$(\mu,l_{\circ},r_{\circ},V)$} if $T$ is both an
$\mathcal{O}$-operator of $(A,\cdot)$ associated to $(\mu,V)$ and
an $\mathcal{O}$-operator of $(A,\circ)$ associated to
$(l_{\circ},r_{\circ},V)$, that is, Eqs.~(\ref{eq:O-op}) and
(\ref{eq:defi:O-operators}) hold.
\end{defi}

\begin{thm}\label{thm:AP1}
    Let $(A,\cdot,\circ)$ be an anti-pre-Lie Poisson algebra and $r\in A\otimes A$ be skew-symmetric. Then $r$ is a solution of the APLP-YBE in $(A,\cdot,\circ)$ if and only if $r$ is an $\mathcal{O}$-operator of $(A,\cdot,\circ)$ associated to $(-\mathcal{L}^{*}_{\cdot},-\mathrm{ad}^{*},\mathcal{R}^{*}_{\circ},A^{*})$.
\end{thm}
\begin{proof}
    By
    \cite
    {Bai2010}, $r$ is a solution of AYBE if and only if $r$ is an $\mathcal{O}$-operator of $(A,\cdot)$ associated to $(-\mathcal{L}^{*}_{\cdot}, A^{*})$. Hence the conclusion follows from Theorem \ref{thm:O-operator and T-equation}.
    \end{proof}

\begin{thm}\label{thm:AP2}
Let $(A,\cdot,\circ)$ be an anti-pre-Lie Poisson algebra and
$(\mu,l_{\circ},r_{\circ},V)$ be a representation of
$(A,\cdot,\circ)$. Set
$\hat{A}=A\ltimes_{-\mu^{*},r^{*}_{\circ}-l^{*}_{\circ},r^{*}_{\circ}}V^{*}$.
Let $T:V\rightarrow A$ be a linear map which is identified as an
element in $\hat{A}\otimes \hat{A}$. Then $r=T-\tau(T)$ is a
skew-symmetric solution of the APLP-YBE in the anti-pre-Lie
Poisson algebra $\hat{A}$ if and only if $T$ is an
$\mathcal{O}$-operator of $(A,\cdot,\circ)$ associated to
$(\mu,l_{\circ},r_{\circ},V)$.
\end{thm}
\begin{proof}
    By \cite{Bai2012}, $r$ is a skew-symmetric solution of the AYBE in $A\ltimes_{-\mu^{*}}V^{*}$ if and only if $T$ is an $\mathcal{O}$-operator associated to $(\mu,V)$. Hence the conclusion follows from Theorem \ref{thm:O-operator and T-equation:semi-direct product version}.
    \end{proof}

\delete{According to Theorems \ref{thm:AP1} and \ref{thm:AP2}, we
have the following corollary.

\begin{cor}\label{cor:O-operator and T-equation:equivalence2}
    Let $(A,\cdot,\circ)$ be an anti-pre-Lie Poisson algebra and $(\mu,l_{\circ},r_{\circ},V)$ be a representation of $(A,\cdot,\circ)$. Set $\hat{A}=A\ltimes_{-\mu^{*},r^{*}_{\circ}-l^{*}_{\circ},r^{*}_{\circ}}V^{*}$.
    Then the following conditions are equivalent:
    \begin{enumerate}
        \item $T$ is an $\mathcal{O}$-operator of $(A,\circ)$ associated to $(\mu,l_{\circ},r_{\circ},V)$.
        \item $r=T-\tau(T)$ is a skew-symmetric solution of the APLP-YBE in the anti-pre-Lie Poisson algebra $\hat{A}$.
        \item $r=T-\tau(T)$ is an $\mathcal{O}$-operator of $\hat{A}$ associated to
        $(-\mathcal{L}^{*}_{\cdot_{\hat{A}}},-\mathrm{ad}^{*}_{\hat{A}},\mathcal{R}^{*}_{\circ_{\hat{A}}},\hat{A}^{*})$.
    \end{enumerate}
\end{cor}
}

\delete{
 Recall the definition of Zinbiel algebras.
\begin{defi}(\cite{Lod})
    A \textbf{Zinbiel algebra} is a pair $(A,\star)$, where $A$ is a vector space, and $\star:A\otimes A\rightarrow A$ is a  bilinear operation such that the following equation holds:
    \begin{equation}\label{eq:Zinbiel}
        x\star(y\star z)=(y\star x)\star z+(x\star y)\star z, \;\;\forall x,y,z\in A.
    \end{equation}
\end{defi}

Let $(A,\star)$ be a Zinbiel algebra. Define a new operation $\cdot$ on $A$ by
\begin{equation}\label{eq:ZintoAss}
    x\cdot y=x\star y+y\star x,\;\;\forall x,y\in A.
\end{equation}
Then $(A,\cdot)$ is a commutative associative algebra. Moreover,
$(\mathcal{L}_{\star},A)$ is a representation of the commutative
associative algebra $(A,\cdot)$, where
$\mathcal{L}_{\star}:A\rightarrow\mathrm{End}(A)$ is given by
$\mathcal{L}_{\star}(x)y=x\star y$, for all $x,y\in A$. 
}

\begin{defi}
    A \textbf{pre-(anti-pre-Lie Poisson) algebra} or a {\bf pre-APLP algebra} in short,  is a quadruple $(A,\star,\succ,\prec)$, such that $(A,\star)$ is a Zinbiel algebra,
     $(A,\succ,\prec)$ is a pre-APL algebra, and the following equations hold:
    \begin{eqnarray}
        2y\star(x\succ z)-2y\star(  z\prec x)&=&(x\succ y+x\prec y)\star z-x\star(    z\prec y),\label{eq:PP1}\\
        2(x\succ y+x\prec y)\star z-2(y\succ  x+y\prec x)\star z&=&y\star(x\succ z)-x\star(y\succ z),\label{eq:PP2}\\
        2z\prec (x\star y+y\star x) &=& (x\star z)\prec y+x\star(  z\prec y),\label{eq:PP3}\\
        2x\succ (y\star z)&=&{(x\star y +y\star x)\succ z}+y\star(x\succ z),\label{eq:PP4}\\
        2x\succ (y\star z)&=& (x\star z)\prec y+{(x\succ y+x\prec y)\star z},\label{eq:PP5}
    \end{eqnarray}
for all $x,y,z\in A$.
\end{defi}

\begin{rmk}
In fact, the operad of pre-APLP algebras is the successor of the
operad of anti-pre-Lie Poisson algebras in the sense of
\cite{BBGN}. Note that they are analogues of pre-Poisson algebras
(\cite{A2}) whose operad is the successor of the operad of Poisson
algebras.

\end{rmk}

\begin{pro}\label{pro:PP1}
    Let $(A,\star,\succ,\prec)$ be a pre-APLP algebra. Let $\cdot,\circ:A\otimes A\rightarrow A$ be two bilinear operations defined by
    Eqs.~(\ref{eq:ZintoAss}) and (\ref{eq:defi:quasi anti-pre-Lie algebras and anti-pre-Lie algebras}) respectively.
    Then $(A,\cdot,\circ)$ is an anti-pre-Lie Poisson algebra, called the \textbf{sub-adjacent anti-pre-Lie Poisson algebra} of $(A,\star,\succ,\prec)$, and
    $(A,\star,\succ,\prec)$ is called a \textbf{compatible pre-APLP algebra} of $(A,\cdot,\circ)$.
    Moreover, $(\mathcal{L}_{\star},\mathcal{L}_{\succ},\mathcal{R}_{\prec},A)$ is a representation of $(A,\cdot,\circ)$.
\end{pro}
\begin{proof}
    Let $x,y,z\in A$, by Eqs.~(\ref{eq:PP1}) and (\ref{eq:PP2}), we have
    \begin{eqnarray*}
    &&2(x\circ y)\cdot z-2(y\circ x)\cdot z+x\cdot(y\circ z)-y\cdot(x\circ z)\\
    &&=2\big((x\succ y)\star z+(  x\prec y)\star z+z\star(x\succ y)+z\star(  x\prec y)\big)\\
    &&\ \ -2\big((y\succ x)\star z+(  y\prec x)\star z+z\star(y\succ x)+z\star(  y\prec x)\big)\\
    &&\ \ +x\star(y\succ z)+x\star(  y\prec z)+(y\succ z)\star x+(  y\prec z)\star x\\
    &&\ \ -y\star(x\succ z)-y\star(  x\prec z)-(x\succ z)\star y-(  x\prec z)\star y\\
    &&=0.
    \end{eqnarray*}
    Hence Eq.~(\ref{eq:defi:anti-pre-Lie Poisson1}) holds.
    Similarly, Eq.~(\ref{eq:defi:anti-pre-Lie Poisson2}) holds by
    Eqs.~(\ref{eq:PP3})-(\ref{eq:PP5}). Thus $(A,\cdot,\circ)$ is an anti-pre-Lie Poisson
    algebra. By Eqs.~(\ref{eq:defi:anti-pre-Lie Poisson
    rep1})-(\ref{eq:defi:anti-pre-Lie Poisson
    rep5}), $(\mathcal{L}_{\star},\mathcal{L}_{\succ},\mathcal{R}_{\prec},A)$ is a representation of
    $(A,\cdot,\circ)$.
\end{proof}

\begin{ex}
    Let $(A,\star)$ be a Zinbiel algebra with a derivation $P$. Let $(A,\succ,\prec)$ be the corresponding pre-APL algebra defined by Eq.~(\ref{eq:ex:Zinbiel derivation}). Then
    it is straightforward to show that $(A,\star,\succ,\prec)$ is a pre-APLP algebra.
\end{ex}

\begin{pro}\label{thm:O-operator and pre anti-pre-Lie Poisson algebras}
    Let $(A,\cdot ,\circ )$ be an anti-pre-Lie Poisson algebra and $(\mu,l_{\circ},r_{\circ},V)$ be a representation of $(A,\cdot ,\circ )$.
    Let $T:V\rightarrow A$ be an $\mathcal{O}$-operator of $(A,\cdot ,\circ )$ associated to $(\mu,l_{\circ},r_{\circ},V)$. Then there exists a  pre-APLP algebra structure
    $(V,\star ,\succ ,\prec )$ on $V$ defined by
    \begin{equation}
        u\star v=\mu\big(T(u)\big)v,\; u\succ  v=l_{\circ}\big(T(u)\big)v,\;
        u\prec  v=r_{\circ}\big(T(v)\big)u,\;\;\forall u,v\in V.
    \end{equation}
 Conversely, {let} $(A,\star,\succ,\prec)$ be a pre-APLP algebra and $(A,\cdot, \circ)$ be the sub-adjacent
 anti-pre-Lie Poisson
algebra. Then the identity map $\mathrm{id}:A\rightarrow A$ is an
$\mathcal{O}$-operator of $(A,\cdot, \circ)$ associated to
$(\mathcal{L}_{\star},\mathcal{L}_{\succ},\mathcal{R}_{\prec},A)$.
\end{pro}

\begin{proof}
It is straightforward.
\end{proof}

\begin{pro}\label{pro:O-operator and APLP-YBE}
    Let $(A,\star,\succ,\prec)$ be a pre-APLP  algebra and $(A,\cdot,\circ)$ be the sub-adjacent anti-pre-Lie Poisson  algebra of $(A,\star,\succ,\prec)$. Let
    $\lbrace e_{1},\cdots ,e_{n}\rbrace$ be a basis of $A$ and $\lbrace e^{*}_{1},\cdots$, $e^{*}_{n}\rbrace$ be the dual basis. Then
    \begin{equation}
        r=\sum_{i=1}^{n}(e_{i}\otimes e^{*}_{i}-e^{*}_{i}\otimes e_{i})
    \end{equation}
    is a skew-symmetric solution of the APLP-YBE in the anti-pre-Lie Poisson algebra
    $A\ltimes_{-\mathcal{L}^{*}_{\star},\mathcal{R}^{*}_{\prec}-\mathcal{L}^{*}_{\succ}, \mathcal{R}^{*}_{\prec}}A^{*}$.
\end{pro}
\begin{proof}
    It follows from Proposition~\ref{thm:O-operator and pre anti-pre-Lie Poisson
    algebras} and Theorem~\ref{thm:AP2}.
\end{proof}

\bigskip

 \noindent
 {\bf Acknowledgements.}  This work is supported by
NSFC (11931009, 12271265, 12261131498), the Fundamental Research
Funds for the Central Universities and Nankai Zhide Foundation.
The authors thank the referee for valuable suggestions.

\section*{Appendix: Proofs of Propositions~\ref{pro:cob coalg} and~\ref{pro:fff2}}

\noindent {\it Proof of Proposition~\ref{pro:cob coalg}}.

(\ref{it:bb1}). Let $x\in A$. Then we have {\small
\begin{eqnarray*}
&&(\mathrm{id}\otimes\delta)\delta(x)-(\tau\otimes\mathrm{id})(\mathrm{id}\otimes\delta)\delta(x)+(\delta\otimes\mathrm{id})\delta(x)-(\tau\otimes\mathrm{id})(\delta\otimes\mathrm{id})\delta(x)\\
&&=\sum\limits_{i,j}\big(a_{i}\otimes a_{j}\otimes[[x,b_{i}],b_{j}]-a_{i}\otimes [x,b_{i}]\circ a_{j}\otimes b_{j}-x\circ a_{i}\otimes a_{j}\otimes [b_{i},b_{j}]+x\circ a_{i}\otimes  b_{i}\circ a_{j}\otimes b_{j}\\
&&\hspace{0.2cm}-a_{j}\otimes
a_{i}\otimes[[x,b_{i}],b_{j}]+[x,b_{i}]\circ a_{j}\otimes
a_{i}\otimes b_{j}+a_{j}\otimes x\circ a_{i}\otimes [b_{i},b_{j}]-b_{i}\circ a_{j}\otimes x\circ a_{i}\otimes b_{j}\\
&&\hspace{0.2cm}+a_{j}\otimes[a_{i},b_{j}]\otimes [x,b_{i}]-a_{i}\circ a_{j}\otimes b_{j}\otimes [x,b_{i}]-a_{j}\otimes[x\circ a_{i},b_{j}]\otimes b_{i}+(x\circ a_{i})\circ a_{j}\otimes b_{j}\otimes b_{i}\\
&&\hspace{0.2cm}-[a_{i},b_{j}]\otimes
a_{j}\otimes [x,b_{i}]+b_{j}\otimes a_{i}\circ
a_{j}\otimes[x,b_{i}]+[x\circ a_{i},b_{j}]\otimes a_{j}\otimes
b_{i}-b_{j}\otimes(x\circ
a_{i})\circ a_{j}\otimes b_{i}\big)\\
&&=A(1)+A(2)+A(3),
\end{eqnarray*}}
where
\begin{eqnarray*}
A(1)&=&\sum\limits_{i,j}(a_{i}\otimes a_{j}\otimes[[x,b_{i}],b_{j}]-a_{j}\otimes a_{i}\otimes[[x,b_{i}],b_{j}]+a_{j}\otimes[a_{i},b_{j}]\otimes[x,b_{i}]\\
&&-a_{i}\circ a_{j}\otimes b_{j}\otimes [x,b_{i}]-[a_{i},b_{j}]\otimes a_{j}\otimes[x,b_{i}]+b_{j}\otimes a_{i}\circ a_{j}\otimes[x,b_{i}])\\
&=&\sum\limits_{i,j}(a_{i}\otimes a_{j}\otimes[x,[b_{i},b_{j}]]+a_{j}\otimes a_{i}\circ b_{j}\otimes[x,b_{i}]-a_{j}\otimes b_{j}\circ a_{i}\otimes [x,b_{i}]\\
&&+b_{j}\otimes a_{i}\circ a_{j}\otimes [x,b_{i}]-a_{i}\circ a_{j}\otimes b_{j}\otimes[x,b_{i}] -[a_{i}, b_{j}]\otimes a_{j}\otimes[x,b_{i}]\\
&&-[a_{i}, a_{j}]\otimes b_{j}\otimes[x,b_{i}]+[a_{i},
a_{j}]\otimes b_{j}\otimes[x,b_{i}]
)\\
&=&-\big(\mathrm{id}\otimes\mathrm{id}\otimes\mathrm{ad}(x)\big)\textbf{T}(r)+\sum_{j}\big(\mathrm{id}\otimes\mathcal{L}_{\circ}(a_{j})\otimes\mathrm{ad}(x)\big)\Big(\big(r+\tau(r)\big)\otimes b_{j}\Big)\\
&&-\sum_{j}\big(\mathrm{ad}(a_{j})\otimes\mathrm{id}\otimes\mathrm{ad}(x)\big)\Big(\big(r+\tau(r)\big)\otimes b_{j}\Big),\\
A(2)&=&\sum\limits_{i,j}\big([x,b_{i}]\circ a_{j}\otimes a_{i}\otimes b_{j}+(x\circ a_{i})\circ a_{j}\otimes b_{j}\otimes b_{i}+[x\circ a_{i},b_{j}]\otimes a_{j}\otimes b_{i}\\
&&-x\circ a_{i}\otimes a_{j}\otimes [b_{i},b_{j}]+x\circ a_{i}\otimes b_{i}\circ a_{j}\otimes b_{j}\big)\\
&\overset{(\ref{eq:defi:anti-pre-Lie algebras1})}{=}&\sum_{i,j}\big(b_{i}\circ(x\circ a_{j})\otimes a_{i}\otimes b_{j}-x\circ(b_{i}\circ a_{j})\otimes a_{i}\otimes b_{j}+(x\circ a_{i})\circ a_{j}\otimes b_{j}\otimes b_{i}\\
&&+(x\circ a_{i})\circ b_{j}\otimes a_{j}\otimes b_{i}-b_{j}\circ(x\circ a_{i})\otimes a_{j}\otimes b_{i}-x\circ(a_{i}\circ a_{j})\otimes b_{i}\otimes b_{j}\\
&&+x\circ (a_{i}\circ a_{j})\otimes b_{i}\otimes b_{j}-x\circ a_{i}\otimes a_{j}\otimes [b_{i}, b_{j}]+x\circ a_{i}\otimes b_{i}\circ a_{j}\otimes b_{j}\big)\\
&=&\big(\mathcal{L}_{\circ}(x)\otimes\mathrm{id}\otimes\mathrm{id}\big)\textbf{T}(r)+\sum_{j}\big(\mathcal{L}_{\circ}(x\circ a_{j})\otimes\mathrm{id}\otimes\mathrm{id}\big)\Big(\big(r+\tau(r)\big)\otimes b_{j}\Big)\\
&&-\sum_{j}\big(\mathcal{L}_{\circ}(x)\mathcal{R}_{\circ}(a_{j})\otimes\mathrm{id}\otimes\mathrm{id}\big)\Big(\big(r+\tau(r)\big)\otimes b_{j}\Big),\\
A(3)&=&\sum_{i,j}(-a_{i}\otimes [x,b_{i}]\circ a_{j}\otimes b_{j}+a_{j}\otimes x\circ a_{i}\otimes[b_{i},b_{j}]-b_{i}\circ a_{j}\otimes x\circ a_{i}\otimes b_{j}\\
&&-a_{j}\otimes[x\circ a_{i},b_{j}]\otimes b_{i}-b_{j}\otimes(x\circ a_{i})\circ a_{j}\otimes b_{i})\\
&=&-(\tau\otimes\mathrm{id})A(2)\\
&=&-(\tau\otimes \mathrm{id})\big(\mathcal{L}_{\circ}(x)\otimes\mathrm{id}\otimes\mathrm{id}\big)\textbf{T}(r)-\sum_{j}(\tau\otimes\mathrm{id})\big(\mathcal{L}_{\circ}(x\circ a_{j})\otimes\mathrm{id}\otimes\mathrm{id}\big)\Big(\big(r+\tau(r)\big)\otimes b_{j}\Big)\\
&&+\sum_{j}(\tau\otimes\mathrm{id})\big(\mathcal{L}_{\circ}(x)\mathcal{R}_{\circ}(a_{j})\otimes\mathrm{id}\otimes\mathrm{id}\big)\Big(\big(r+\tau(r)\big)\otimes
b_{j}\Big).
\end{eqnarray*}
Hence Eq.~(\ref{eq:defi:anti-pre-Lie coalgebras1}) holds if and
only if Eq.~(\ref{eq:pro:cob coalg1}) holds.

(\ref{it:bb2}). Let $x\in A$. Then we have
\begin{eqnarray*}
    &&(\mathrm{id}^{\otimes 3}+\xi+\xi^{2})(\mathrm{id}^{\otimes 3}-\tau\otimes \mathrm{id})(\mathrm{id}\otimes\delta)\delta(x)\\
    &&=(\mathrm{id}^{\otimes 3}+\xi+\xi^{2})\sum_{i,j}(a_{i}\otimes a_{j}\otimes[[x,b_{i}],b_{j}]-b_{j}\otimes a_{i}\otimes[x,b_{i}]\circ a_{j}\\
    &&\ \ -a_{j}\otimes [b_{i},b_{j}]\otimes x\circ a_{i}+b_{i}\circ a_{j}\otimes b_{j}\otimes x\circ a_{i}- a_{j}\otimes a_{i}\otimes [[x,b_{i}],b_{j}]\\
    &&\ \ +a_{i}\otimes b_{j}\otimes[x,b_{i}]\circ a_{j}+[b_{i},b_{j}]\otimes a_{j}\otimes x\circ a_{i}-b_{j}\otimes b_{i}\circ a_{j}\otimes x\circ a_{i})\\
    &&=(\mathrm{id}^{\otimes 3}+\xi+\xi^{2})\big(B(1)+B(2)+B(3)\big),
\end{eqnarray*}
where {\small
\begin{eqnarray*}
B(1)&=&\sum_{i,j}(a_{i}\otimes a_{j}\otimes[[x,b_{i}],b_{j}]-b_{j}\otimes a_{i}\otimes[x,b_{i}]\circ a_{j}- a_{j}\otimes a_{i}\otimes [[x,b_{i}],b_{j}]+a_{i}\otimes b_{j}\otimes[x,b_{i}]\circ a_{j})\\
&=&\sum_{i,j}(a_{i}\otimes a_{j}\otimes[[x,b_{i}],b_{j}]-b_{j}\otimes a_{i}\otimes[x,b_{i}]\circ a_{j}-a_{j}\otimes a_{i}\otimes[x,b_{i}]\circ b_{j}+a_{j}\otimes a_{i}\otimes[x,b_{i}]\circ b_{j}\\
&&- a_{j}\otimes a_{i}\otimes [[x,b_{i}],b_{j}]+a_{i}\otimes b_{j}\otimes[x,b_{i}]\circ a_{j}+a_{i}\otimes a_{j}\otimes[x,b_{i}]\circ b_{j}-a_{i}\otimes a_{j}\otimes[x,b_{i}]\circ b_{j})\\
&=&\sum_{i,j}a_{i}\otimes a_{j}\otimes x\circ[b_{i},b_{j}]-\sum_{j}\big(\mathrm{id}\otimes\mathrm{id}\otimes\mathcal{L}_{\circ}([x,b_{j}])\big)(\tau\otimes\mathrm{id})\Big(a_{j}\otimes\big(r+\tau(r)\big)\Big)\\
&&+\sum_{j}\big(\mathrm{id}\otimes\mathrm{id}\otimes\mathcal{L}_{\circ}([x,b_{j}])\big)\Big(a_{j}\otimes\big(r+\tau(r)\big)\Big),\\
B(2)&=&\sum_{i,j}(-a_{j}\otimes [b_{i},b_{j}]\otimes x\circ a_{i}-b_{j}\otimes b_{i}\circ a_{j}\otimes x\circ a_{i})\\
&=&\sum_{i,j}(-a_{j}\otimes [b_{i},b_{j}]\otimes x\circ a_{i}-a_{j}\otimes [a_{i},b_{j}]\otimes x\circ b_{i}+a_{j}\otimes [a_{i},b_{j}]\otimes x\circ b_{i}\\
&&-b_{j}\otimes b_{i}\circ a_{j}\otimes x\circ a_{i}-b_{j}\otimes a_{i}\circ a_{j}\otimes x\circ b_{i}\\
&&+b_{j}\otimes a_{i}\circ a_{j}\otimes x\circ b_{i}+a_{j}\otimes a_{i}\circ b_{j}\otimes x\circ b_{i}-a_{j}\otimes a_{i}\circ b_{j}\otimes x\circ b_{i})\\
&=&-\sum_{i,j}a_{j}\otimes b_{j}\circ a_{i}\otimes x\circ b_{i}+\sum_{j}\big(\mathrm{id}\otimes\mathrm{ad}(b_{j})\otimes\mathcal{L}_{\circ}(x)\big)\Big(a_{j}\otimes\big(r+\tau(r)\big)\Big)\\
&&-\sum_{j}\big(\mathrm{id}\otimes\mathcal{R}_{\circ}(a_{j})\otimes\mathcal{L}_{\circ}(x)\big)\Big(b_{j}\otimes\big(r+\tau(r)\big)\Big)+\sum_{j}\big(\mathrm{id}\otimes\mathcal{L}_{\circ}(a_{j})\otimes\mathcal{L}_{\circ}(x)\big)\Big(b_{j}\otimes\big(r+\tau(r)\big)\Big),\\
B(3)&=&\sum_{i,j}(b_{i}\circ a_{j}\otimes b_{j}\otimes x\circ a_{i}+[b_{i},b_{j}]\otimes a_{j}\otimes x\circ a_{i})\\
&=&\sum_{i,j}(b_{i}\circ a_{j}\otimes b_{j}\otimes x\circ a_{i}+[b_{i},b_{j}]\otimes a_{j}\otimes x\circ a_{i}+[b_{i},a_{j}]\otimes b_{j}\otimes x\circ a_{i}-[b_{i},a_{j}]\otimes b_{j}\otimes x\circ a_{i})\\
&=&\sum_{i,j}(a_{j}\circ b_{i}\otimes b_{j}\otimes x\circ a_{i}+a_{j}\circ a_{i}\otimes b_{j}\otimes x\circ b_{i}-a_{j}\circ a_{i}\otimes b_{j}\otimes x\circ b_{i}\\
&&+[b_{i},b_{j}]\otimes a_{j}\otimes x\circ a_{i}+[b_{i},a_{j}]\otimes b_{j}\otimes x\circ a_{i})\\
&=&-\sum_{i,j}a_{i}\circ a_{j}\otimes b_{i}\otimes x\circ b_{j} +\sum_{j}\big(\mathcal{L}_{\circ}(a_{j})\otimes\mathrm{id}\otimes\mathcal{L}_{\circ}(x)\big)(\tau\otimes\mathrm{id})\Big(b_{j}\otimes\big(r+\tau(r)\big)\Big)\\
&&+\sum_{j}\big(\mathrm{ad}(b_{j})\otimes\mathrm{id}\otimes\mathcal{L}_{\circ}(x)\big)\Big(\big(r+\tau(r)\big)\otimes
a_{j}\Big).
\end{eqnarray*}}
Hence Eq.~(\ref{eq:defi:anti-pre-Lie coalgebras2}) holds if and
only if Eq.~(\ref{eq:pro:cob coalg2}) holds.

(\ref{it:bb3}). 
Let $x\in A$. Then we have
\begin{eqnarray*}
    &&(\mathrm{id}^{\otimes 2}-\tau)\Big(\delta(x\circ y)-\big(\mathcal{L}_{\circ}(x)\otimes \mathrm{id}\big)\delta(y)-\big(\mathrm{id}\otimes\mathcal{L}_{\circ}(x)\big)\delta(y)+\big(\mathrm{id}\otimes\mathcal{R}_{\circ}(y)\big)\delta(x)\Big)\\
    &&=\sum_{i}\big(a_{i}\otimes [x\circ y,b_{i}]-(x\circ y)\circ a_{i}\otimes b_{i}-[x\circ y,b_{i}]\otimes a_{i}+b_{i}\otimes(x\circ y)\circ a_{i}\\
    &&\ \ -x\circ a_{i}\otimes [y,b_{i}]+x\circ(y\circ a_{i})\otimes b_{i}+[y,b_{i}]\otimes x\circ a_{i}-b_{i}\otimes x\circ(y\circ a_{i})\\
    &&\ \ -a_{i}\otimes x\circ[y,b_{i}]+y\circ a_{i}\otimes x\circ b_{i}+x\circ[y,b_{i}]\otimes a_{i}-x\circ b_{i}\otimes y\circ a_{i}\\
    &&\ \ +a_{i}\otimes [x,b_{i}]\circ y-x\circ a_{i}\otimes b_{i}\circ y-[x,b_{i}]\circ y\otimes a_{i}+b_{i}\circ y\otimes x\circ a_{i}\big)\\
    &&=C(1)+C(2)+C(3),
\end{eqnarray*}
where
\begin{small}
    \begin{eqnarray*}
        C(1)&=&\sum_{i}\big(a_{i}\otimes [x\circ y,b_{i}]+b_{i}\otimes(x\circ y)\circ a_{i}-b_{i}\otimes x\circ(y\circ a_{i})-a_{i}\otimes x\circ[y,b_{i}]+a_{i}\otimes [x,b_{i}]\circ y\big)\\
        &\overset{(\ref{eq:defi:anti-pre-Lie algebras1})}{=}&\sum_{i}\big(a_{i}\otimes (x\circ y)\circ b_{i}-a_{i}\otimes x\circ(y\circ b_{i})+b_{i}\otimes (x\circ y)\circ a_{i}-b_{i}\otimes x\circ(y\circ a_{i})\big)\\
        &=&\big(\mathrm{id}\otimes\mathcal{L}_{\circ}(x\circ y)-\mathrm{id}\otimes\mathcal{L}_{\circ}(x)\mathcal{L}_{\circ}(y)\big)\big(r+\tau(r)\big),\\
        C(2)&=&\sum_{i}\big(-(x\circ y)\circ a_{i}\otimes b_{i}-[x\circ y,b_{i}]\otimes a_{i}+x\circ(y\circ a_{i})\otimes b_{i}+x\circ[y,b_{i}]\otimes a_{i}-[x,b_{i}]\circ y\otimes a_{i}\big)\\
        &\overset{(\ref{eq:defi:anti-pre-Lie algebras1})}{=}&\sum_{i}\big(-(x\circ y)\circ a_{i}\otimes b_{i}+x\circ(y\circ a_{i})\otimes b_{i}-(x\circ y)\circ b_{i}\otimes a_{i}+x\circ(y\circ b_{i})\otimes a_{i}\big)\\
        &=&\big(\mathcal{L}_{\circ}(x)\mathcal{L}_{\circ}(y)\otimes \mathrm{id}-\mathcal{L}_{\circ}(x\circ y)\otimes \mathrm{id}\big)\big(r+\tau(r)\big),\\
        C(3)&=&\sum_{i}(-x\circ a_{i}\otimes [y,b_{i}]+[y,b_{i}]\otimes x\circ a_{i}+y\circ a_{i}\otimes x\circ b_{i}-x\circ b_{i}\otimes y\circ a_{i}-x\circ a_{i}\otimes b_{i}\circ y\\
        &&\hspace{1cm}+b_{i}\circ y\otimes x\circ a_{i})\\
        &=&\sum_{i}(-x\circ a_{i}\otimes y\circ b_{i}+y\circ b_{i}\otimes x\circ a_{i}+y\circ a_{i}\otimes x\circ b_{i}-x\circ b_{i}\otimes y\circ a_{i})\\
        &=&\big(\mathcal{L}_{\circ}(y)\otimes\mathcal{L}_{\circ}(x)-\mathcal{L}_{\circ}(x)\otimes\mathcal{L}_{\circ}(y)\big)\big(r+\tau(r)\big).
    \end{eqnarray*}
\end{small}
Thus Eq.~(\ref{eq:defi:anti-pre-Lie bialgebra1}) holds if and only
if Eq.~(\ref{eq:pro:coboundary anti-pre-Lie bialgebras1}) holds.
\hfill $\Box$

\bigskip

\vspace{1cm}

\noindent {\it Proof of Proposition~\ref{pro:fff2}}.

 (\ref{it:1}).
Let $x\in A$. Then we have
        \begin{small}
            \begin{eqnarray*}
                &&2(\delta\otimes\mathrm{id})\Delta(x)-2(\tau\otimes\mathrm{id})(\delta\otimes\mathrm{id})\Delta(x)+(\mathrm{id}\otimes\delta)\Delta(x)-(\tau\otimes\mathrm{id})(\mathrm{id}\otimes\delta)\Delta(x)\\
                &&=\sum_{i,j}\big(2a_{j}\otimes[x\cdot a_{i},b_{j}]\otimes b_{i}-2(x\cdot a_{i})\circ a_{j}\otimes b_{j}\otimes b_{i}-2a_{j}\otimes [a_{i},b_{j}]\otimes x\cdot b_{i}+2a_{i}\circ a_{j}\otimes b_{j}\otimes x\cdot b_{i}\\
                &&\ \ -2[x\cdot a_{i},b_{j}]\otimes a_{j}\otimes b_{i}+2b_{j}\otimes (x\cdot a_{i})\circ a_{j}\otimes b_{i}+2[a_{i},b_{j}]\otimes a_{j}\otimes x\cdot b_{i}-2b_{j}\otimes a_{i}\circ a_{j}\otimes x\cdot b_{i}\\
                &&\ \ +x\cdot a_{i}\otimes a_{j}\otimes[b_{i},b_{j}]-x\cdot a_{i}\otimes b_{i}\circ a_{j}\otimes b_{j}-a_{i}\otimes a_{j}\otimes[x\cdot b_{i}, b_{j}]+a_{i}\otimes(x\cdot b_{i})\circ a_{j}\otimes b_{j}\\
                &&\ \ -a_{j}\otimes x\cdot a_{i}\otimes [b_{i},b_{j}]+b_{i}\circ a_{j}\otimes x\cdot a_{i}\otimes b_{j}+a_{j}\otimes a_{i}\otimes[x\cdot b_{i},b_{j}]-(x\cdot b_{i})\circ a_{j}\otimes a_{i}\otimes b_{j}\big)\\
                &&=D(1)+D(2)+D(3),
            \end{eqnarray*}
        \end{small}
    where
    \begin{small}
    \begin{eqnarray*}
        D(1)&=&\sum_{i,j}\big(-2(x\cdot a_{i})\circ a_{j}\otimes b_{j}\otimes b_{i}-2[x\cdot a_{i},b_{j}]\otimes a_{j}\otimes b_{i}\\
        &&+x\cdot a_{i}\otimes a_{j}\otimes[b_{i},b_{j}]-x\cdot a_{i}\otimes b_{i}\circ a_{j}\otimes b_{j}-(x\cdot b_{i})\circ a_{j}\otimes a_{i}\otimes b_{j}\big)\\
        &=&\sum_{i,j}\big(-2(x\cdot a_{i})\circ a_{j}\otimes b_{j}\otimes b_{i}-2[x\cdot a_{i},b_{j}]\otimes a_{j}\otimes b_{i}-2[x\cdot a_{i}, a_{j}]\otimes b_{j}\otimes b_{i}\\
        &&+2[x\cdot a_{i}, a_{j}]\otimes b_{j}\otimes b_{i}+x\cdot a_{i}\otimes a_{j}\otimes[b_{i},b_{j}]-x\cdot a_{i}\otimes b_{i}\circ a_{j}\otimes b_{j}\\
        &&-(x\cdot a_{i})\circ a_{j}\otimes b_{i}\otimes b_{j}-(x\cdot b_{i})\circ a_{j}\otimes a_{i}\otimes b_{j}+(x\cdot a_{i})\circ a_{j}\otimes b_{i}\otimes b_{j}\big)\\
        &\overset{(\ref{eq:defi:anti-pre-Lie Poisson2})}{=}&\sum_{i,j}\big(-x\cdot (a_{i}\circ a_{j})\otimes b_{i}\otimes b_{j} -x\cdot a_{i}\otimes b_{i}\circ a_{j}\otimes b_{j}+x\cdot a_{i}\otimes a_{j}\otimes[b_{i},b_{j}]\big)\\
        &&+ \sum_{j}\bigg(-2\big(\mathrm{ad}(x\cdot a_{j})\otimes\mathrm{id}\otimes\mathrm{id}\big)\Big(\big(r+\tau(r)\big)\otimes b_{j}\Big)-\big(\mathcal{R}_{\circ}(a_{j})\mathcal{L}_{\cdot}(x)\otimes\mathrm{id}\otimes\mathrm{id}\big)\Big(\big(r+\tau(r)\big)\otimes b_{j}\Big)\bigg)\\
        &=&-\big(\mathcal{L}_{\cdot}(x)\otimes\mathrm{id}\otimes\mathrm{id}\big)\textbf{T}(r)-\sum_{j}\big(2\mathrm{ad}(x\cdot a_{j})\otimes\mathrm{id}\otimes\mathrm{id}+\mathcal{R}_{\circ}(a_{j})\mathcal{L}_{\cdot}(x)\otimes\mathrm{id}\otimes\mathrm{id}\big)\Big(\big(r+\tau(r)\big)\otimes b_{j}\Big),\\
        D(2)&=&\sum_{i,j}\big(2a_{j}\otimes[x\cdot a_{i},b_{j}]\otimes b_{i}+2b_{j}\otimes (x\cdot a_{i})\circ a_{j}\otimes b_{i}\\
        &&\ \ +a_{i}\otimes(x\cdot b_{i})\circ a_{j}\otimes b_{j} -a_{j}\otimes x\cdot a_{i}\otimes [b_{i},b_{j}]+b_{i}\circ a_{j}\otimes x\cdot a_{i}\otimes b_{j}\big)\\
        &=&\sum_{i,j}\big(2b_{i}\otimes (x\cdot a_{j})\circ a_{i}\otimes b_{j}+2a_{i}\otimes (x\cdot a_{j})\circ b_{i}\otimes b_{j}-2a_{i}\otimes (x\cdot a_{j})\circ b_{i}\otimes b_{j}\\
        &&\ \ +2a_{i}\otimes[x\cdot a_{j},b_{i}]\otimes b_{j}+a_{i}\otimes(x\cdot b_{i})\circ a_{j}\otimes b_{j}+a_{i}\otimes x\cdot a_{j}\otimes[b_{i},b_{j}]\\
        &&\ \ +b_{i}\circ a_{j}\otimes x\cdot a_{i}\otimes b_{j}+a_{i}\circ a_{j}\otimes x\cdot b_{i}\otimes b_{j}-a_{i}\circ a_{j}\otimes x\cdot b_{i}\otimes b_{j}\big)\\
        &\overset{(\ref{eq:defi:anti-pre-Lie Poisson2})}{=}&\sum_{i,j}\big(-a_{i}\otimes x\cdot(b_{i}\circ a_{j})\otimes b_{j}+a_{i}\otimes x\cdot a_{j}\otimes [b_{i},b_{j}]-a_{i}\circ a_{j}\otimes x\cdot b_{i}\otimes b_{j}\big)\\
        &&+\sum_{j}\bigg(2\big(\mathrm{id}\otimes\mathcal{L}_{\circ}(x\cdot a_{j})\otimes\mathrm{id}\big)\Big(\big(r+\tau(r)\big)\otimes b_{j}\Big)+\big(\mathcal{R}_{\circ}(a_{j})\otimes\mathcal{L}_{\cdot}(x)\otimes\mathrm{id}\big)\Big(\big(r+\tau(r)\big)\otimes b_{j}\Big)\bigg)\\
        &=&-\big(\mathrm{id}\otimes\mathcal{L}_{\cdot}(x)\otimes\mathrm{id}\big)\textbf{T}(r)+\sum_{j}\big(2\mathrm{id}\otimes\mathcal{L}_{\circ}(x\cdot a_{j})\otimes\mathrm{id}+\mathcal{R}_{\circ}(a_{j})\otimes\mathcal{L}_{\cdot}(x)\otimes\mathrm{id}\big)\Big(\big(r+\tau(r)\big)\otimes b_{j}\Big),\\
        D(3)&=&\sum_{i,j}(-2a_{j}\otimes [a_{i},b_{j}]\otimes x\cdot b_{i}+2a_{i}\circ a_{j}\otimes b_{j}\otimes x\cdot b_{i}+2[a_{i},b_{j}]\otimes a_{j}\otimes x\cdot b_{i}\\
        &&\ \ -2b_{j}\otimes a_{i}\circ a_{j}\otimes x\cdot b_{i}-a_{i}\otimes a_{j}\otimes[x\cdot b_{i},b_{j}]+a_{j}\otimes a_{i}\otimes[x\cdot b_{i},b_{j}])\\
        &=&\sum_{i,j}(
        -2a_{i}\otimes[a_{j},b_{i}]\otimes x\cdot b_{j}+2a_{j}\circ a_{i}\otimes b_{i}\otimes x\cdot b_{j}\\
        &&\ \ +2[a_{j},b_{i}]\otimes a_{i}\otimes x\cdot b_{j}+2[a_{j},a_{i}]\otimes b_{i}\otimes x\cdot b_{j}-2[a_{j},a_{i}]\otimes b_{i}\otimes x\cdot b_{j}\\
        &&\ \ -2b_{i}\otimes a_{j}\circ a_{i}\otimes x\cdot b_{j}-2a_{i}\otimes a_{j}\circ b_{i}\otimes x\cdot b_{j}
        +2a_{i}\otimes a_{j}\circ b_{i}\otimes x\cdot b_{j}\\
        &&\ \ -a_{i}\otimes a_{j}\otimes[x\cdot b_{i},b_{j}]
        -a_{i}\otimes a_{j}\otimes[ b_{i},x\cdot b_{j}])\\
        &\overset{(\ref{eq:defi:transposed Poisson algebra})}{=}&\sum_{i,j}(2a_{i}\circ a_{j}\otimes b_{i}\otimes x\cdot b_{j}+2a_{i}\otimes b_{i}\circ a_{j}\otimes x\cdot b_{j}-2a_{i}\otimes a_{j}\otimes x\cdot [b_{i},b_{j}])\\
        &&+\sum_{j}\bigg(2\big(\mathrm{ad}(a_{j})\otimes\mathrm{id}\otimes\mathcal{L}_{\cdot}(x)\big)\Big(\big(r+\tau(r)\big)\otimes b_{j}\Big)-2\big(\mathrm{id}\otimes\mathcal{L}_{\circ}(a_{j})\otimes\mathcal{L}_{\cdot}(x)\big)\Big(\big(r+\tau(r)\big)\otimes b_{j}\Big)\bigg)\\
        &=&2\big(\mathrm{id}\otimes\mathrm{id}\otimes\mathcal{L}_{\cdot}(x)\big)\textbf{T}(r)_+\sum_{j}2\big(\mathrm{ad}(a_{j})\otimes\mathrm{id}\otimes\mathcal{L}_{\cdot}(x)-\mathrm{id}\otimes\mathcal{L}_{\circ}(a_{j})\otimes\mathcal{L}_{\cdot}(x)\big)\Big(\big(r+\tau(r)\big)\otimes b_{j}\Big).
    \end{eqnarray*}
\end{small}
Hence Eq.~(\ref{eq:defi:anti-pre-Lie Poisson coalg1}) holds if and
only if Eq.~(\ref{eq:TPBA1}) holds.

(\ref{it:2}). Let $x\in A$. Then we have
\begin{small}
\begin{eqnarray*}
&&2(\mathrm{id}\otimes\Delta)\delta(x)-(\mathrm{id}\otimes\tau)(\Delta\otimes\mathrm{id})\delta(x)+(\delta\otimes\mathrm{id})\Delta(x)\\
&&=\sum_{i,j}\big(2a_{i}\otimes[x,b_{i}]\cdot a_{j}\otimes
b_{j}-2a_{i}\otimes a_{j}\otimes[x,b_{i}]\cdot b_{j} -2x\circ
a_{i}\otimes b_{i}\cdot a_{j}\otimes b_{j}+2x\circ a_{i}\otimes
a_{j}\otimes b_{i}\cdot b_{j}
\\
&&\ \ -a_{i}\cdot a_{j}\otimes[x,b_{i}]\otimes b_{j}+a_{j}\otimes[x,b_{i}]\otimes a_{i}\cdot b_{j}+(x\circ a_{i})\cdot a_{j}\otimes b_{i}\otimes b_{j}-a_{j}\otimes b_{i}\otimes(x\circ a_{i})\cdot b_{j}\\
&&\ \ -a_{j}\otimes[x\cdot a_{i},b_{j}]\otimes b_{i}+(x\cdot a_{i})\circ a_{j}\otimes b_{j}\otimes b_{i}+a_{j}\otimes [a_{i},b_{j}]\otimes x\cdot b_{i}-a_{i}\circ a_{j}\otimes b_{j}\otimes x\cdot b_{i}\big)\\
&&=E(1)+E(2)+E(3),
\end{eqnarray*}
\end{small}
where
\begin{small}
\begin{eqnarray*}
E(1)&=&\sum_{i,j}\big(-2x\circ a_{i}\otimes b_{i}\cdot a_{j}\otimes b_{j}+2x\circ a_{i}\otimes a_{j}\otimes b_{i}\cdot b_{j}+(x\circ a_{i})\cdot a_{j}\otimes b_{i}\otimes b_{j}\\
&\mbox{}&\ \ +(x\circ a_{j})\cdot a_{i}\otimes b_{i}\otimes b_{j}\big)\\
&\overset{(\ref{eq:defi:anti-pre-Lie Poisson2})}{=}&\sum_{i,j}\big(-2x\circ a_{i}\otimes b_{i}\cdot a_{j}\otimes b_{j}+2x\circ a_{i}\otimes a_{j}\otimes b_{i}\cdot b_{j}+2x\circ (a_{i}\cdot a_{j})\otimes b_{i}\otimes b_{j}\big)\\
&=&2\big(\mathcal{L}_{\circ}(x)\otimes\mathrm{id}\otimes\mathrm{id}\big)\textbf{A}(r),\\
E(2)&=&\sum_{i,j}(2a_{i}\otimes[x,b_{i}]\cdot a_{j}\otimes b_{j}-a_{i}\cdot a_{j}\otimes [x,b_{i}]\otimes b_{j}+a_{j}\otimes[x,b_{i}]\otimes a_{i}\cdot b_{j}-a_{j}\otimes[x\cdot a_{i},b_{j}]\otimes b_{i})\\
&=&\sum_{i,j}(2a_{i}\otimes[x,b_{i}]\cdot a_{j}\otimes b_{j}-a_{i}\cdot a_{j}\otimes [x,b_{i}]\otimes b_{j}+a_{j}\otimes[x,b_{i}]\otimes a_{i}\cdot b_{j}\\
&&\ \ +a_{j}\otimes[x,a_{i}]\otimes b_{i}\cdot b_{j}-a_{i}\otimes[x,a_{j}]\otimes b_{i}\cdot b_{j}-a_{i}\otimes[x\cdot a_{j},b_{i}]\otimes b_{j})\\
&\overset{(\ref{eq:defi:transposed Poisson algebra})}{=}&\sum_{i,j}(a_{i}\otimes[x,a_{j}\cdot b_{i}]\otimes b_{j}-a_{i}\cdot a_{j}\otimes[x,b_{i}]\otimes b_{j}-a_{i}\otimes[x,a_{j}]\otimes b_{i}\cdot b_{j})\\
&&\ \ +\sum_{j}\big(\mathrm{id}\otimes\mathrm{ad}(x)\otimes\mathcal{L}_{\cdot}(b_{j})\big)\Big(a_{j}\otimes\big(r+\tau(r)\big)\Big)\\
&=&-\big(\mathrm{id}\otimes\mathrm{ad}(x)\otimes\mathrm{id}\big)\textbf{A}(r)+\sum_{j}\big(\mathrm{id}\otimes\mathrm{ad}(x)\otimes\mathcal{L}_{\cdot}(b_{j})\big)\Big(a_{j}\otimes\big(r+\tau(r)\big)\Big),\\
E(3)&=&\sum_{i,j}\big(-2a_{i}\otimes a_{j}\otimes[x,b_{i}]\cdot b_{j}-a_{j}\otimes b_{i}\otimes(x\circ a_{i})\cdot b_{j}+a_{j}\otimes [a_{i},b_{j}]\otimes x\cdot b_{i}-a_{i}\circ a_{j}\otimes b_{j}\otimes x\cdot b_{i}\big)\\
&=&\sum_{i,j}\big(-2a_{i}\otimes a_{j}\otimes[x,b_{i}]\cdot b_{j}-a_{j}\otimes b_{i}\otimes(x\circ a_{i})\cdot b_{j}-a_{j}\otimes a_{i}\otimes(x\circ b_{i})\cdot b_{j}\\
&&\ \ +a_{j}\otimes a_{i}\otimes(x\circ b_{i})\cdot
b_{j}+a_{j}\otimes[a_{i},b_{j}]\otimes x\cdot b_{i}
+a_{j}\otimes[b_{i},b_{j}]\otimes x\cdot a_{i}-a_{j}\otimes[b_{i},b_{j}]\otimes x\cdot a_{i}\\
&&\ \ -a_{i}\circ a_{j}\otimes b_{j}\otimes x\cdot b_{i}\big)\\
&=&\sum_{i,j}\big(-2a_{i}\otimes a_{j}\otimes[x,b_{i}]\cdot b_{j}+a_{j}\otimes a_{i}\otimes(x\circ b_{i})\cdot b_{j}-a_{j}\otimes[b_{i},b_{j}]\otimes x\cdot a_{i}-a_{i}\circ a_{j}\otimes b_{j}\otimes x\cdot b_{i}\big)\\
&&\ \ -\sum_{j}\big(\mathrm{id}\otimes\mathrm{id}\otimes\mathcal{L}_{\cdot}(b_{j})\mathcal{L}_{\circ}(x)\big)\Big(a_{j}\otimes\big(r+\tau(r)\big)\Big)-\sum_{j}\big(\mathrm{id}\otimes\mathrm{ad}(b_{j})\otimes\mathcal{L}_{\cdot}(x)\big)\Big(a_{j}\otimes\big(r+\tau(r)\big)\Big)\\
&\overset{(\ref{eq:defi:anti-pre-Lie Poisson2})}{=}&\sum_{i,j}\big(a_{i}\otimes a_{j}\otimes x\cdot(b_{i}\circ b_{j})+a_{i}\otimes b_{j}\otimes x\cdot(b_{i}\circ a_{j})-a_{i}\otimes b_{j}\otimes x\cdot(b_{i}\circ a_{j})\\
&&\ \ +a_{i}\otimes [b_{i},b_{j}]\otimes x\cdot a_{j}-a_{i}\circ
a_{j}\otimes b_{j}\otimes x\cdot b_{i}
\big)\\
&&\ \ -\sum_{j}\big(\mathrm{id}\otimes\mathrm{id}\otimes\mathcal{L}_{\cdot}(b_{j})\mathcal{L}_{\circ}(x)+\mathrm{id}\otimes\mathrm{ad}(b_{j})\otimes\mathcal{L}_{\cdot}(x)\big)\Big(a_{j}\otimes\big(r+\tau(r)\big)\Big)\\
&=&-\big(\mathrm{id}\otimes\mathrm{id}\otimes\mathcal{L}_{\cdot}(x)\big)(\mathrm{id}\otimes\tau)\textbf{T}(r)\\
&&\
+\sum_{j}\big(\mathrm{id}\otimes\mathrm{id}\otimes\mathcal{L}_{\cdot}(x)\mathcal{L}_{\circ}(b_{j})-\mathrm{id}\otimes\mathrm{id}\otimes\mathcal{L}_{\cdot}(b_{j})\mathcal{L}_{\circ}(x)-\mathrm{id}\otimes\mathrm{ad}(b_{j})\otimes\mathcal{L}_{\cdot}(x)\big)\Big(a_{j}\otimes\big(r+\tau(r)\big)\Big).
\end{eqnarray*}
\end{small}
Hence Eq.~(\ref{eq:defi:anti-pre-Lie Poisson coalg2}) holds if and
only if Eq.~(\ref{eq:TPBA2}) holds.

(\ref{it:3}).
 Let $x,y\in A$. Then we have
\begin{small}
\begin{eqnarray*}
&&2\big(\mathcal{L}_{\circ}(x)\otimes\mathrm{id}\big)\Delta(y)-2\big(\mathrm{id}\otimes\mathcal{L}_{\cdot}(y)\big)\delta(x)+\delta(x\cdot y)+\big(\mathcal{L}_{\cdot}(y)\otimes\mathrm{id}\big)\delta(x)-\big(\mathrm{id}\otimes\mathrm{ad}(x)\big)\Delta(y)\\
&&=\sum_{i}\big(2x\circ(y\cdot a_{i})\otimes b_{i}-2x\circ a_{i}\otimes y\cdot b_{i} -2a_{i}\otimes y\cdot[x,b_{i}]+2x\circ a_{i}\otimes y\cdot b_{i}+a_{i}\otimes[x\cdot y, b_{i}]\\
&&\ \  -(x\cdot y)\circ a_{i}\otimes b_{i} +y\cdot
a_{i}\otimes[x,b_{i}]-y\cdot(x\circ a_{i}) \otimes b_{i} -y\cdot
a_{i}\otimes[x,b_{i}]+a_{i}\otimes[x,y\cdot b_{i}]
\big)\overset{(\ref{eq:defi:anti-pre-Lie Poisson2})}{=}0.
\end{eqnarray*}
\end{small}
Thus Eq.(\ref{eq:Poisson bialg 1}) holds automatically.

(\ref{it:4}). It follows from a similar proof of
Item~(\ref{it:3}).


(\ref{it:5}). Let $x,y\in A$. Then we have
\begin{small}
\begin{eqnarray*}
&&2\big(\mathrm{id}\otimes\mathcal{L}_{\cdot}(y)\big)\delta(x)-2\big(\mathcal{L}_{\circ}(x)\otimes\mathrm{id}\big)\Delta(y)+\Delta(x\circ
y)+
{\big(\mathcal{R}_{\circ}(y)\otimes\mathrm{id}\big)\Delta(x)}+\tau\big(\mathcal{L}_{\cdot}(x)\otimes\mathrm{id}\big)\delta(y)
\\
&&\hspace{5cm}-\big(\mathrm{id}\otimes\mathcal{L}_{\cdot}(x)\big)\delta(y)\\
&&=\sum_{i}\big(2a_{i}\otimes y\cdot[x,b_{i}]-2x\circ a_{i}\otimes y\cdot b_{i}-2x\circ(y\cdot a_{i})\otimes b_{i}+2x\circ a_{i}\otimes y\cdot b_{i}\\
&&\ \ +(x\circ y)\cdot a_{i}\otimes b_{i}-a_{i}\otimes(x\circ
y)\cdot b_{i}
+(x\cdot a_{i})\circ y\otimes b_{i}-a_{i}\circ y\otimes x\cdot b_{i}\\
&&\ \ +[y,b_{i}]\otimes x\cdot a_{i}-b_{i}\otimes x\cdot(y\circ a_{i})-a_{i}\otimes x\cdot[y,b_{i}]+y\circ a_{i}\otimes x\cdot b_{i}\big)\\
&&=F(1)+F(2)+F(3)+F(4),
\end{eqnarray*}
\end{small}
where
\begin{small}
\begin{eqnarray*}
F(1)&=&\sum_{i}\big(-2x\circ(y\cdot a_{i})\otimes b_{i}+(x\circ y)\cdot a_{i}\otimes b_{i}+(x\cdot a_{i})\circ y\otimes b_{i}\big)\overset{(\ref{eq:defi:anti-pre-Lie Poisson2})}{=}0,\\
F(2)&=&\sum_{i}(-2x\circ a_{i}\otimes y\cdot b_{i}+2x\circ a_{i}\otimes y\cdot b_{i})=0,\\
F(3)&=&\sum_{i}(-a_{i}\circ y\otimes x\cdot b_{i}+[y,b_{i}]\otimes x\cdot a_{i}+y\circ a_{i}\otimes x\cdot b_{i})=\big(\mathrm{ad}(y)\otimes\mathcal{L}_{\cdot}(x)\big)\big(r+\tau(r)\big),\\
F(4)&=&\sum_{i}\big({2}a_{i}\otimes y\cdot[x,b_{i}]-a_{i}\otimes(x\circ y)\cdot b_{i}-b_{i}\otimes x\cdot (y\circ a_{i})-a_{i}\otimes x\cdot[y,b_{i}]\big)\\
&\overset{(\ref{eq:defi:anti-pre-Lie
Poisson1})}{=}&-\big({\mathrm{id}\otimes\mathcal{L}_{\cdot}(x)\mathcal{L}_{\circ}(y)}\big)\big(r+\tau(r)\big).
\end{eqnarray*}
\end{small}
Hence Eq.(\ref{eq:Poisson bialg 3}) holds if and only if
Eq.~(\ref{eq:TPBA3}) holds.

(\ref{it:6}). Let $x,y\in A$. Then we have
\begin{small}
        \begin{eqnarray*}
                &&(\tau-\mathrm{id}^{\otimes 2})\Big(2\delta(x\cdot y)-\big(\mathcal{L}_{\cdot}(x)\otimes\mathrm{id}\big)\delta(y)-\big(\mathrm{id}\otimes\mathcal{L}_{\cdot}(x)\big)\delta(y)-\big(\mathrm{id}\otimes\mathcal{R}_{\circ}(y)\big)\Delta(x)\Big)\\
                &&=\sum_{i,j}\big(
                2[x\cdot y,b_{i}]\otimes a_{i}-2b_{i}\otimes(x\cdot y)\circ a_{i}-2a_{i}\otimes[x\cdot y,b_{i}]+2(x\cdot y)\circ a_{i}\otimes b_{i}\\
                &&\ \ +[y,b_{i}]\otimes x\cdot a_{i}+b_{i}\otimes x\cdot(y\circ a_{i})+x\cdot a_{i}\otimes [y,b_{i}]-x\cdot(y\circ a_{i})\otimes b_{i}\\
                &&\ \ -x\cdot[y,b_{i}]\otimes a_{i}+x\cdot b_{i}\otimes y\circ a_{i}+a_{i}\otimes x\cdot[y,b_{i}]-y\circ a_{i}\otimes x\cdot b_{i}\\
                &&\ \ -b_{i}\circ y\otimes x\cdot a_{i}+(x\cdot b_{i})\circ y\otimes a_{i}+x\cdot a_{i}\otimes b_{i}\circ y-a_{i}\otimes(x\cdot b_{i})\circ y\big)\\
                &&=G(1)+G(2)+G(3)+G(4),
            \end{eqnarray*}
\end{small}
where
\begin{small}
\begin{eqnarray*}
    G(1)&=&\sum_{i}\big(
    2[x\cdot y,b_{i}]\otimes a_{i}+2(x\cdot y)\circ a_{i}\otimes b_{i}-x\cdot(y\circ a_{i})\otimes b_{i}-x\cdot[y,b_{i}]\otimes a_{i}+(x\cdot b_{i})\circ y\otimes a_{i}\big)\\
    &=&\sum_{i}\big(2(x\cdot y)\circ b_{i}\otimes a_{i}-2b_{i}\circ(x\cdot y)\otimes a_{i}+2(x\cdot y)\circ a_{i}\otimes b_{i}-x\cdot(y\circ a_{i})\otimes b_{i}\\
    &&-x\cdot( y\circ b_{i})\otimes a_{i}+x\cdot(  b_{i}\circ y)\otimes a_{i}+(x\cdot b_{i})\circ y\otimes a_{i}\big)\\
    &\overset{(\ref{eq:defi:anti-pre-Lie Poisson2})}{=}&2\big(\mathcal{L}_{\circ}(x\cdot y)\otimes\mathrm{id}\big)\big(r+\tau(r)\big)-\big(\mathcal{L}_{\cdot}(x)\mathcal{L}_{\circ}(y)\otimes\mathrm{id}\big)\big(r+\tau(r)\big),\\
        G(2)&=&\sum_{i}(x\cdot a_{i}\otimes[y,b_{i}]+x\cdot b_{i}\otimes y\circ a_{i}+x\cdot a_{i}\otimes b_{i}\circ y)=\big(\mathcal{L}_{\cdot}(x)\otimes\mathcal{L}_{\circ}(y)\big)\big(r+\tau(r)\big),
\end{eqnarray*}
\end{small}
and similarly
\begin{small}
    \begin{eqnarray*}
    G(3)&=&\sum_{i}(-[y,b_{i}]\otimes x\cdot a_{i}-y\circ a_{i}\otimes x\cdot b_{i} -b_{i}\circ y\otimes x\cdot a_{i})=-\big(\mathcal{L}_{\circ}(y)\otimes\mathcal{L}_{\cdot}(x)\big)\big(r+\tau(r)\big),\\
    G(4)&=&\sum_{i}\big(
    -2b_{i}\otimes(x\cdot y)\circ a_{i}-2a_{i}\otimes[x\cdot y,b_{i}]+b_{i}\otimes x\cdot(y\circ a_{i})
    +a_{i}\otimes x\cdot[y,b_{i}]-a_{i}\otimes(x\cdot b_{i})\circ y\big)\\
    &=&\big(\mathrm{id}\otimes\mathcal{L}_{\cdot}(x)\mathcal{L}_{\circ}(y)-2\mathrm{id}\otimes\mathcal{L}_{\circ}(x\cdot y)\big)\big(r+\tau(r)\big).
    \end{eqnarray*}
\end{small}
Hence Eq.(\ref{eq:Poisson bialg 4}) holds if and only if 
Eq.~(\ref{eq:TPBA4}) holds.
    \hfill $\Box$

\end{document}